\documentclass[final]{siamltex}
\usepackage{latexsym, graphicx, epsfig, amsmath, amsfonts,amssymb}
\usepackage{multirow}
\usepackage{color}
\usepackage{tikz}
 \usepackage[percent]{overpic}
 \usepackage{subfigure}
\usepackage{algorithm}
\usepackage{algorithmicx}
\usepackage{algpseudocode}
\floatname{algorithm}{Algorithm}
\usepackage{threeparttable}
\usepackage{booktabs}
\usepackage{epstopdf,mathtools,bbm}
\usepackage[left=1.25in,right=1.25in,top=1in]{geometry}

\graphicspath{{DDM_result_figures/}}
\newtheorem{rem}{Remark}[section]
\newtheorem{defn}{Definition}[section]
\newtheorem{thm}{Theorem}[section]

\newtheorem{lem}{Lemma}[section]

\title{Physics-aware deep learning framework for the limited aperture inverse obstacle scattering problem}

\author{
Yunwen Yin\thanks {School of Mathematics, Southeast University, Nanjing, China.}
\and Liang Yan \thanks {School of Mathematics, Southeast University, Nanjing, China. Email:yanliang@seu.edu.cn. LY's work was supported by the NSF of China (Nos. 92370126, 12171085).}
}

\date{} 

\begin{document}
\maketitle

\begin{abstract}
In this paper, we consider a deep learning approach to the limited aperture inverse obstacle scattering problem. It is well known that traditional deep learning relies solely on data, which may limit its performance for the inverse problem when only indirect observation data and a physical model are available. A fundamental question arises in light of these limitations: is it possible to enable deep learning to work on inverse problems without labeled data and to be aware of what it is learning? This work proposes a deep decomposition method (DDM) for such purposes, which does not require ground truth labels. It accomplishes this by providing physical operators associated with the scattering model to the neural network architecture. Additionally, a deep learning based data completion scheme is implemented in DDM to prevent distorting the solution of the inverse problem for limited aperture data. Furthermore, apart from addressing the ill-posedness imposed by the inverse problem itself, DDM is the first physics-aware machine learning technique that can have interpretability property for the obstacle detection. The convergence result of DDM is theoretically investigated. \textcolor{black}{We also prove that adding small noise to the input limited aperture data can introduce additional regularization terms and effectively improve the smoothness of the learned inverse operator.} Numerical experiments are presented to demonstrate the validity of the proposed DDM even when the incident and observation apertures are extremely limited.
\end{abstract}



\section{Introduction}
Inverse scattering problems are common in many  industrial and engineering applications, such as nondestructive testing \cite{ABF}, radar imaging \cite{B1} and medical imaging \cite{Kuchment}.  The majority of inverse scattering problems are non-linear and  specifically ill-posed,  meaning  that the solution may not exist or may not be unique, and, more importantly, it fails to depend continuously on the data, causing a minor perturbation in the data to cause a massive deviation in the solution. As a result, achieving stable and accurate numerical solutions is extremely difficult. There are numerous numerical methods using full aperture data for such inverse problems, such as the Newton-type iterative method \cite{Kress}, the optimization method \cite{KK1}, the recursive linearization method \cite{Bao_Li_Lin_1,Bao_Lu_Rundell_Xu1}, the linear sampling method \cite{CK,CPP}, the direct sampling method \cite{IJZ,LZ} and the factorization method \cite{KL1,Yangjiaqing1}. However, the incident and observation apertures are typically restricted due to the limitations of practical settings; consequently, inverse scattering problems with limited aperture data result in increased non-linearity and ill-posedness. Several reconstruction techniques have been proposed \cite{AH1,BL1,INS1,R1}, by directly utilizing limited aperture measurements. \textcolor{black}{Unfortunately, almost no classical numerical method can simultaneously ensure inversion accuracy and computational efficiency for limited aperture problems--not even quantitative optimization or iterative methods. To achieve high-quality inversion results, a natural approach is to use reconstruction techniques that first recover data across all incident and observation angles, thereby enabling the use of a full aperture dataset \cite{DLMZ1, Liu_Sun_1, Gao_Zhang1}. Because the inversion process is split into two parts, this approach also increases the computational cost even though the inversion quality is improved. Furthermore, the inverse problems in traditional approaches must inevitably be solved again when new observation data are used, which restricts the potential for real-time reconstructions, particularly for iterative algorithms. Owing to these factors, looking for an alternative method for solving limited aperture inverse scattering problems is imperative.}

Recently, deep neural networks (DNNs) have demonstrated their promising features in a variety of inverse problems \cite{LLTW1,GLWZ1,gao2023adaptive,li2023surrogate,ZHRB1,zhou2020adaptive,YYL1,NHZ1,LZZ1,Ning_Han_Zou2}. In contrast to classical physics-based inversion methods, DNNs typically do not rely on the underlying physics, but the trade-off is that they are restricted to be purely data-driven and rooted in the big data regime. As a result, they will always be unable to solve the inverse problem: one can not obtain the quantity of interest known as  labeled data, but only indirect observation data and a physical model. To address these issues, one natural solution  is to provide DNNs with the governing physics. To do so, in \cite{RPKarniadakis_1}, physics-informed neural networks (PINNs) were proposed for solving the forward and inverse problems by constraining  physical knowledge of governing equations  into the loss function. However, in the original forms of PINNs, unknown parameters in inverse problems are only taken into account  as constants, and they are updated together with weights and biases of neural networks.  In order to develop PINNs for the scenario where the unknown parameters are functions, two networks must be trained concurrently: one for the solutions of partial differential equations (PDEs) and another for the unknown parameters; see \cite{Liu2024failure,LPYWVJ1,R-B_HSK1,YLMK1} and the references therein. We also strongly refer the readers to \cite{BYZZhou1} for other relevant work called weak adversarial network \cite{ZBYZ_weak} that possesses the similar methodology for inverse problems. Although PINNs have shown great potential in solving various inverse problems, similar to classical methods, they have to be learned again for given new observation data. \textcolor{black}{Alternatively, physics informed neural operators (PINOs) \cite{goswami_2023_1,goswami_2022,jiao_2024,li2024PIFNO,wang_sifan_2021} that combine governing PDEs and operator learning technology \cite{li2020FNO,lu_deeponet} can overcome such a retraining issue arising from PINNs-type methods. However, these approaches primarily focus on learning forward operators and typically work as surrogate forward solvers incorporated in traditional inversion methods.} 

\textcolor{black}{We believe that directly learning the inverse operator  with the incorporation of underlying physical principles is a more effective way for solving inverse problems. However, due to the inherent ill-posedness of inverse problems, the inverse operator is frequently unbounded and usually lacks existence and uniqueness, making this task extremely difficult. This difficulty also explains why current efforts have been limited, with the emphasis primarily on learning the inverse operator using data-driven methods. Results from data-driven methods are heavily reliant on the quality of training data, and these methods are particularly difficult to generalize to out-of-distribution data. To address this, several approaches involving physical mechanisms have been proposed.  For example, the work in\cite{kaltenbach2023SINO}  introduced an  invertible network architecture within  PINOs to solve Bayesian inverse problems. The studies in  \cite{PMAG1,NB1} incorporated forward mapping constraints into their loss functions to utilize physical information, resulting in physics-aware neural networks and model-constrained Tikhonov networks, respectively. In contrast, the works in \cite{Khoo_Ying1,MYEM1} focused on carefully designing hidden layer structures in neural networks by leveraging the properties of specific inverse problems.} These works provide effective methods for leveraging physics information; however, many inverse problems require an alternative approach due to the high difficulty of embedding PDEs or forward operators into loss functions or indirectly introducing physics information into hidden layers.

Inspired by the aforementioned discussions, we present in this work a physics-aware deep decomposition approach for the limited aperture inverse obstacle scattering problem. Due to the unbounded nature of the inverse operator associated with acoustic obstacle scattering and the high ill-posedness caused by the limited aperture, learning the inverse operator will be extremely challenging. As a result, regularization techniques must be used in the network structure or optimization process. This is accomplished by constraining a penalty term created by the Herglotz operator into the deep decomposition method (DDM) loss function. Moreover,  we use the scattering information, which includes the Herglotz operator, the far-field operator, and the fundamental solution, in place of the forward mapping to construct the loss function.  In fact, it is numerically intractable to directly embed the forward mapping of the acoustic obstacle scattering problem into the network because the forward solver, which is typically based on the boundary integral equation method or the finite element method, is quite complicated.  On the other hand, DDM takes into consideration a deep learning-based data retrieval strategy that incorporates a convolution neural network, which is motivated by the research presented in \cite{Gao_Zhang1}. \textcolor{black}{To the best of our knowledge, the proposed DDM is the first to retrieve full aperture data while also incorporating physical information from the acoustic obstacle scattering model into the neural network for impenetrable obstacle detection. DDM reconstructs both the boundary and the entire aperture data simultaneously, unlike the works \cite{DLMZ1, Liu_Sun_1, Gao_Zhang1}, where the inversion process is split into two parts. More importantly, once trained, DDM can perform real-time reconstruction with newly acquired limited aperture data.}  We summarize the main features and novelties of our proposed DDM as follows:

\begin{itemize}
\item DDM is the first physics-aware machine learning approach to tackle the limited aperture inverse obstacle scattering problem.  It combines deep learning, physical information, data retrieval, and boundary recovery techniques.  More importantly, DDM is more aware of what it is learning and has some interpretability thanks to the physical information.
\item DDM does not require exact boundary information, also referred to as labeled data, during the training phase. DDM can resolve the ill-posedness caused by the relevant inverse problem by adding a regularization term associated to the Herglotz operator into the loss function.  Because DDM is trained with the guidance of the underlying physics information, it can learn the regularized inverse operator more efficiently.
\item Theoretically, we rigorously prove DDM's convergence result using the properties of the far-field and Herglotz operators. \textcolor{black}{We also demonstrate that adding relatively small noise to measured limited aperture data is useful for promoting the smoothness of the learned inverse operator.}
\item We demonstrate the effectiveness of our proposed DDM using numerical examples. It clearly shows that DDM can produce satisfactory reconstructions even when the incident and observation apertures are both extremely limited. Moreover, DDM has the benefit of real-time numerical computation because, once trained, it can solve the inverse problem in terms of forward propagation.
 \end{itemize}

The rest of the paper is organized as follows. In the following section, we present the basis setup and preliminary results for the acoustic inverse obstacle scattering model.  In Section 3, we present the DDM with its convergence result and discretization form. Numerical experiments are given in Section 4 to show the promising features of DDM and conclusions are finally made in Section 5.

\section{Problem setup and preliminaries}
Assume that $D \subset \mathbb{R}^2$ is an open and bounded simply connected domain with a $C^{2}$-boundary $\partial D$. The incident field $u^i$ is given by
\begin{equation}\label{eq:inc}
    \begin{aligned}
u^{i}(x):=u^i(x, d)=\mathrm{e}^{\mathrm{i} k x \cdot d},\ x \in \mathbb{R}^2,
    \end{aligned}
\end{equation}
where $\mathrm{i}:=\sqrt{-1}$ is the imaginary unit, $k \in \mathbb{R}_{+}$ is the wavenumber and $d \in S:=\left\{x \in \mathbb{R}^2:|x|=1\right\}$ is the direction of the propagation. The presence of the impenetrable obstacle $D$ interrupts the propagation of the incident wave $u^{i}$, yielding the exterior scattered field $u^{s}$. The direct scattering problem for the sound-soft obstacle is to find $u^{s}=u-u^{i}$ satisfying the Helmholtz equation
\begin{equation}\label{eq:Helmholtz}
\begin{aligned}
\Delta u^{s}+k^2 u^{s}=0,\ \text { in } \mathbb{R}^2 \backslash \overline{D},
\end{aligned}
\end{equation}
with the Dirichlet boundary condition
\begin{equation}\label{eq:bc}
\begin{aligned}
u^{s}+u^{i}=0, \ \text { on } \partial D,
\end{aligned}
\end{equation}
where $u$ is the total field. In addition, the following Sommerfeld radiation condition
\begin{equation}\label{eq:Sommerfeld}
\begin{aligned}
\lim\limits _{r \rightarrow \infty} r^{\frac{1}{2}}\left(\frac{\partial u^s}{\partial r}-\mathrm{i} k u^s\right)=0,\ r=|x|,
\end{aligned}
\end{equation}
holds, which characterizes the outgoing nature of the scattered field $u^s$. It is well known that the forward scattering problem \eqref{eq:Helmholtz}-\eqref{eq:Sommerfeld} is well-posed (cf. \cite{DRIA,McLean1}). The scattered field $u^s$ also has the asymptotic behavior
\begin{equation}
    \begin{aligned}
    u^{s}(x)=\frac{{\mathrm{e}}^{\mathrm{i}kr}}{\sqrt{r}}\left\{u^{\infty}(\hat{x})+\mathcal{O}\left(\frac{1}{r}\right)\right\},\ r=|x|\rightarrow\infty,
    \end{aligned}
\end{equation}
uniformly in all directions $\hat{x}=x/|x|\in S$. Here, $u^{\infty}(\hat{x})$, which is called the far-field pattern of the scattered field, is an analytic function on $S$. In what follows, we write $u^{\infty}(\hat{x},d)$ to signify such far-field data and specify its dependence on the observation direction $\hat{x}$ and the incident direction $d$.

The inverse scattering problem we are concerned with is to recover the boundary $\partial D$ from the limited aperture data $u^{\infty}(\hat{x},d)$ for $(\hat{x},d)\in \gamma^{o}\times \gamma^{i}$ at a fixed wavenumber $k$, where $\gamma^{o}\subset S$ and $\gamma^{i}\subseteq S$ are respectively the limited observation and incident aperture.  Fig. \ref{fig:scattering_model} shows an illustration of the inverse problem. Define a non-linear forward operator $\mathcal{G}$ that maps the boundary to the corresponding far-field pattern, the above mentioned inverse problem can be formulated as follows:
\begin{equation}\label{eq:limited_inverse_scattering}
    \begin{aligned}
    \partial{D}=\mathcal{G}^{-1}(u^{\infty}(\hat{x},d)),\ (\hat{x},d)\in \gamma^{o}\times \gamma^{i}.
    \end{aligned}
\end{equation}
\begin{figure}[t]
    \centering
    \subfigure{
    \includegraphics[width=0.5\textwidth]{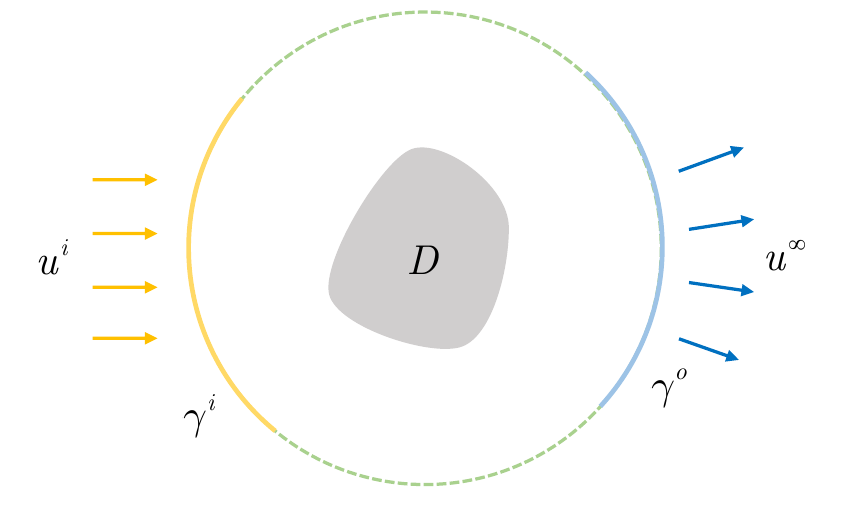}
}
    \caption{\label{fig:scattering_model}A schematic illustration of the limited aperture inverse obstacle scattering problem.}
\end{figure}

In this work, we aim to apply a physics-aware machine learning approach to solve  the inverse problem defined in\eqref{eq:limited_inverse_scattering}. Unlike traditional  data-driven neural networks, we focus on providing a specially designed neural network along with  physics information. \textcolor{black}{Our deep decomposition method(DDM) is inspired by the  decomposition method introduced by Colton and Monk in\cite{DRIA,Colton_Monk1,Colton_Monk2}. We briefly review this decomposition method, which involves two main phases:  first, an approximate solution to the interior boundary value problem within the obstacle is obtained using Herglotz wavefunctions; then, the boundary of the obstacle is identified as the curve where the boundary condition of the interior problem is satisfied.} We begin by defining a superposition of incident plane waves as follows:
\begin{equation}\label{eq:v_i}
\begin{aligned}
v^{i}(x) = \int_{S}\mathrm{e}^{\mathrm{i}kx\cdot d}g(d)\mathrm{d}s(d),\ x\in \mathbb{R}^{2},
\end{aligned}
\end{equation}
with the weight function $g\in L^{2}(S)$. In this case, $v^{i}$ is also referred to a Herglotz wave function, where $g$ is the Herglotz kernel.

Let us define the following interior boundary value problem, assuming that $k^{2}$ is not a Dirichlet eigenvalue to the negative Laplacian in $D$:
\begin{equation}\label{eq:interior_bvp}
\begin{aligned}
\begin{cases}\Delta v^{i}(x)+k^2 v^{i}(x)=0, & \text { in } D, \\ v^{i}(x)=-\Phi(x,z), & \text { on } \partial D, \end{cases}
\end{aligned}
\end{equation}
where $\Phi(x,z)=\frac{{\mathrm{i}}}{{4}}H_{0}^{(1)}(k|x-z|)$ is the outgoing Green function of Helmholtz equation in $\mathbb{R}^2$ and $H_{0}^{(1)}$ is the Hankel function of the first kind of order zero. Moreover, the point $z$ is contained in the sound-soft obstacle $D$ \textcolor{black}{such that the Herglotz kernel $g$ is bounded. Define the  far-field operator $F:L^{2}(S)\rightarrow L^{2}(S)$ as:
\begin{equation}\label{eq:farfield_operator}
    \begin{aligned}
    (Fg)(\hat{x})=\int_{S}u^{\infty}(\hat{x},d)g(d)\mathrm{d}s(d),\ \hat{x} \in S.
    \end{aligned}
\end{equation}
Applying the Dirichlet boundary condition \eqref{eq:bc} and treating $v^{i}$ as the incident wave in \eqref{eq:Helmholtz}-\eqref{eq:Sommerfeld},  we define a  bounded and injective operator $A: L^{2}(\partial D)\rightarrow L^{2}(S)$ that maps the boundary values of radiating solutions to their corresponding far-field pattern \cite{DRIA} as follow:
\begin{equation}
A(-v^{i}|_{\partial D}):= Fg.
\end{equation} 
Moreover, we have  $A(\Phi(x,z)|_{\partial D})=\Phi^{\infty}(\hat{x},z)$, where $\Phi^{\infty}(\hat{x},z)=\frac{\mathrm{e}^{\mathrm{i}\pi/4}}{\sqrt{8\pi k}}\mathrm{e}^{-\mathrm{i}k\hat{x}\cdot z}$ represents the far-field pattern of the fundamental solution $\Phi(x,z)$. Because of the boundary condition $v^{i}|_{\partial D} = -\Phi(x,z)|_{\partial D}$ in \eqref{eq:interior_bvp} and the fact that $A$ is injective, we can obtain the integral equation of the first kind
\begin{equation}\label{eq:far_operator_g}
    \begin{aligned}
    Fg=\Phi^{\infty}(\hat{x},z).
    \end{aligned}
\end{equation} 
It is interesting to notice that \eqref{eq:far_operator_g} does not involve the knowledge of the exact boundary $\partial D$, which provides an effective way of determining $g$. However, the equation \eqref{eq:far_operator_g} is ill-posed, and hence the regularization scheme is required. If we acquire the approximate solution $g_{\alpha}$ of \eqref{eq:far_operator_g} by the regularization scheme with the corresponding regularization parameter $\alpha$, the boundary of the obstacle $D$ can then be found by solving the non-linear equation
\begin{equation}\label{eq:g_operator_bc}
    \begin{aligned}
    v^{i}_{\alpha}(x)+\Phi(x,z)=0,
    \end{aligned}
\end{equation}
where $v^{i}_{\alpha}(x)$ is the approximate Herglotz wave function with kernel $g_{\alpha}$.}

\textcolor{black}{Note that solving \eqref{eq:g_operator_bc} over the entire domain is not conducive to optimization.  To streamline the solution process into an optimization framework, we introduce an admissible curve, $\Lambda$, belonging to the compact class $U$, which is used to approximate $\partial D$. Similar to the Herglotz wave function $v^i$ defined in \eqref{eq:v_i}, we define the following Herglotz operator $H:L^{2}(S)\rightarrow L^{2}(\Lambda)$:
\begin{equation}\label{eq:Herglotz}
\begin{aligned}
(Hg)(x) = \int_{S}\mathrm{e}^{\mathrm{i}kx\cdot d}g(d)\mathrm{d}s(d),\ x\in \Lambda.
\end{aligned}
\end{equation}
The approximate boundary $\partial D$ can then be identified by solving the following equation over the admissible cure $\Lambda$:
\begin{equation}\label{eq:Hg_operator_for_solving_bc}
    \begin{aligned}
    (Hg)(x)+\Phi(x,z)=0,\quad x\in \Lambda.
    \end{aligned}
\end{equation}
With this reformulation, the inverse obstacle scattering problem becomes an optimization problem in both  $g$  and  $\Lambda$, with a regularization term. The objective function for this optimization is given by:
\begin{equation}\label{eq:CM_functional}
\mathcal{J}(g, \Lambda ; \alpha):=\left\|F g-\Phi^{\infty}(\hat{x},z)\right\|_{L^2\left(S\right)}^2+\gamma\left\|H g+\Phi(x,z)\right\|_{L^2(\Lambda)}^2+\alpha\|H g\|_{L^2\left(\Gamma\right)}^2,
\end{equation}
where $\Gamma$ is the $C^{2}$-curve  containing $\Lambda$ in its interior, $\alpha\in \mathbb{R}_{+}$ is the regularization parameter and $\gamma$ is the coupling parameter. Here, the first and second terms in the right hand side of \eqref{eq:CM_functional} are derived by \eqref{eq:far_operator_g} and \eqref{eq:Hg_operator_for_solving_bc}, respectively.  The regularization parameter  $\alpha$ is used to resolve the ill-posed nature by $\alpha\|H g\|_{L^2\left(\Gamma\right)}^2$, while  $\gamma$ is used to  balance the terms $\left\|F g-\Phi^{\infty}(\hat{x},z)\right\|_{L^2\left(S\right)}^2$ and $\left\|H g+\Phi(x,z)\right\|_{L^2(\Lambda)}^2$. } It is worth noting that \eqref{eq:CM_functional} may work similarly to the neural network's loss function. This is the mathematical basis for creating a scattering-based network architecture in the following section.

\section{Deep decomposition method}
In this section, we propose a deep decomposition method (DDM) to approximate the inverse map $\mathcal{G}^{-1}$. Furthermore, we will discuss its convergence result and discretization form.

\subsection{DDM for the inverse obstacle scattering problem}
\textcolor{black}{The DDM addresses the limited aperture inverse obstacle scattering problem in \eqref{eq:limited_inverse_scattering} by dividing it into two primary sub-problems: data retrieval and boundary recovery.   Additionally, DDM further decomposes the boundary recovery process into simultaneous Herglotz kernel computation and boundary reconstruction, similar to the decomposition method introduced by Colton and Monk \cite{DRIA, Colton_Monk1, Colton_Monk2}. This approach enables boundary recovery from a physics-informed perspective.  The entire procedure  is implemented using a deep learning approach, hence the name deep decomposition method.} 

In order to retrieve the full aperture data, a data completion network (DCnet) is designed in the first phase of DDM, that is
\begin{equation}\label{eq:DCnet}
u_{\theta_{DC}}^{\infty}(\hat{x},d)|_{S \times S}=\mathrm{DCnet}(u^{\infty}(\hat{x},d)|_{\gamma^{o}\times \gamma^{i}}),
\end{equation}
where $u^{\infty}(\hat{x},d)|_{\gamma^{o}\times \gamma^{i}}$ denoting $u^{\infty}(\hat{x},d)$ for $(\hat{x},d)\in \gamma^{o}\times \gamma^{i}$ is the input of DDM and $u_{\theta_{DC}}^{\infty}(\hat{x},d)|_{S \times S}$ means the recovered $u_{\theta_{DC}}^{\infty}(\hat{x},d)$ for $(\hat{x},d)\in S\times S$. \eqref{eq:DCnet} can be done by minimizing the following functional
\begin{equation}\label{eq:DC_functional}
\mathcal{J}_{DC}(u_{\theta_{DC}}^{\infty}(\hat{x},d)):= \int_{S}\int_{S}|u^{\infty}(\hat{x},d)-u_{\theta_{DC}}^{\infty}(\hat{x},d)|^{2}\mathrm{d}s(d)\mathrm{d}s(\hat{x}).
\end{equation}
In fact, by using the reciprocity relation $u^{\infty}(\hat{x},d)=u^{\infty}(-d, -\hat{x})$, one can recover the unknown
far-field data on some part from given limited data $u^{\infty}(\hat{x},d)|_{\gamma^{o}\times \gamma^{i}}$. However, in this work, we are more concerned with directly learning the analytic continuation property from limited aperture data to the corresponding full aperture data. This is why we establish the functional \eqref{eq:DC_functional} from the data-based viewpoint for \eqref{eq:DCnet}.

After finding $u_{\theta_{DC}}^{\infty}(\hat{x},d)$, inspired by the preliminary results  for the full aperture data, we further build a Herglotz kernel network named as HKnet for $g$ and a boundary reconstruction network named as BRnet for $\Lambda$ in the second phase of DDM. The specific forms of these two parts are as follows:
\begin{equation}\label{eq:HK_net}
g_{\theta_{HK}}=\mathrm{HKnet}(u_{\theta_{DC}}^{\infty}(\hat{x},d)|_{S \times S}),
\end{equation}
and
\begin{equation}\label{eq:BR_net}
\Lambda_{\theta_{BR}}=\mathrm{BRnet}(u_{\theta_{DC}}^{\infty}(\hat{x},d)|_{S \times S}).
\end{equation}
Then, we focus on designing a physics-aware loss function comparable to \eqref{eq:CM_functional}. To do so, we first define a new far-field operator $F_{\theta_{DC}}:L^{2}(S)\rightarrow L^{2}(S)$
\begin{equation}\label{eq:farfield_operator_theta}
    \begin{aligned}
    (F_{\theta_{DC}}g_{\theta_{HK}})(\hat{x})=\int_{S}u_{\theta_{DC}}^{\infty}(\hat{x},d)g_{\theta_{HK}}(d)\mathrm{d}s(d),\ \hat{x} \in S.
    \end{aligned}
\end{equation}
Therefore, by \eqref{eq:farfield_operator_theta} and the Herglotz operator, the physics-aware loss function for learning $g_{\theta_{HK}}$ and $\Lambda_{\theta_{BR}}$ is defined as follows:
\begin{equation}\label{eq:CM_functional_theta}
    \begin{aligned}
\mathcal{J}_{phy}(g_{\theta_{HK}}, \Lambda_{\theta_{BR}};\alpha):=&\left\|F_{\theta_{DC}}g_{\theta_{HK}}-\Phi^{\infty}(\hat{x},z)\right\|_{L^2\left(S\right)}^2+\gamma\left\|Hg_{\theta_{HK}}+\Phi(x,z)\right\|_{L^2(\Lambda_{\theta_{BR}})}^2\\
&+\alpha\|Hg_{\theta_{HK}}\|_{L^2\left(\Gamma_{\theta_{BR}}\right)}^2,
    \end{aligned}
\end{equation}
where the $C^{2}$-curve $\Gamma_{\theta_{BR}}$ contains $\Lambda_{\theta_{BR}}$ in its interior. Note that the optimization functional \eqref{eq:CM_functional_theta} does not require any information of the exact boundary $\partial D$.

Let $\Theta$ be a finite dimensional parameter space. The proposed DDM aims to minimize the combined cost functional
\begin{equation}\label{eq:DDM_functional_theta}
\mathcal{J}_{DDM}(u_{\theta_{DC}}^{\infty}(\hat{x},d), g_{\theta_{HK}}, \Lambda_{\theta_{BR}};\alpha)=\mathcal{J}_{phy}(g_{\theta_{HK}}, \textcolor{black}{\Lambda_{\theta_{BR}}};\alpha)+\beta_{DC}\mathcal{J}_{DC}(u_{\theta_{DC}}^{\infty}(\hat{x},d)),
\end{equation}
for determining the optimal $\Lambda_{\theta_{BR}^{\ast}}$ by finding the optimal parameters $\theta_{DC}^{\ast},\theta_{HK}^{\ast},\theta_{BR}^{\ast}\in\Theta$, and $\beta_{DC}$ is a penalty factor that balances the physics-based loss $\mathcal{J}_{phy}$ and the data-based loss $\mathcal{J}_{DC}$. The designed loss function \eqref{eq:DDM_functional_theta} can well address both the nonlinearity and the ill-posedness caused by the inverse problem itself. Thus, the parametric map $\mathcal{NN}_{\theta_{DC},\theta_{BR}}$ constructed by DDM for approximating the inverse operator $\mathcal{G}^{-1}$ can be summarized as
\begin{equation}\label{eq:DDM_net}
    \begin{aligned}
    \mathcal{NN}_{\theta_{DC},\theta_{BR}}=\mathrm{BRnet}(\mathrm{DCnet})).
    \end{aligned}
\end{equation}
The general workflow of DDM is presented in Fig. \ref{fig:DDM_model}.
\begin{rem}\label{rem_ddm_two_inverse}
As we have pointed out earlier, retrieving the full aperture data and recovering the boundary from given limited aperture data are addressed simultaneously in the learning stage of DDM, since the above proposed three networks are simultaneously trained and their parameters $\theta_{DC},\theta_{HK},\theta_{BR}\in\Theta$ are simultaneously updated by optimizing \eqref{eq:DDM_functional_theta}.
\end{rem}
\begin{figure}[t]
    \centering
    \subfigure{
    \includegraphics[width=0.85\textwidth]{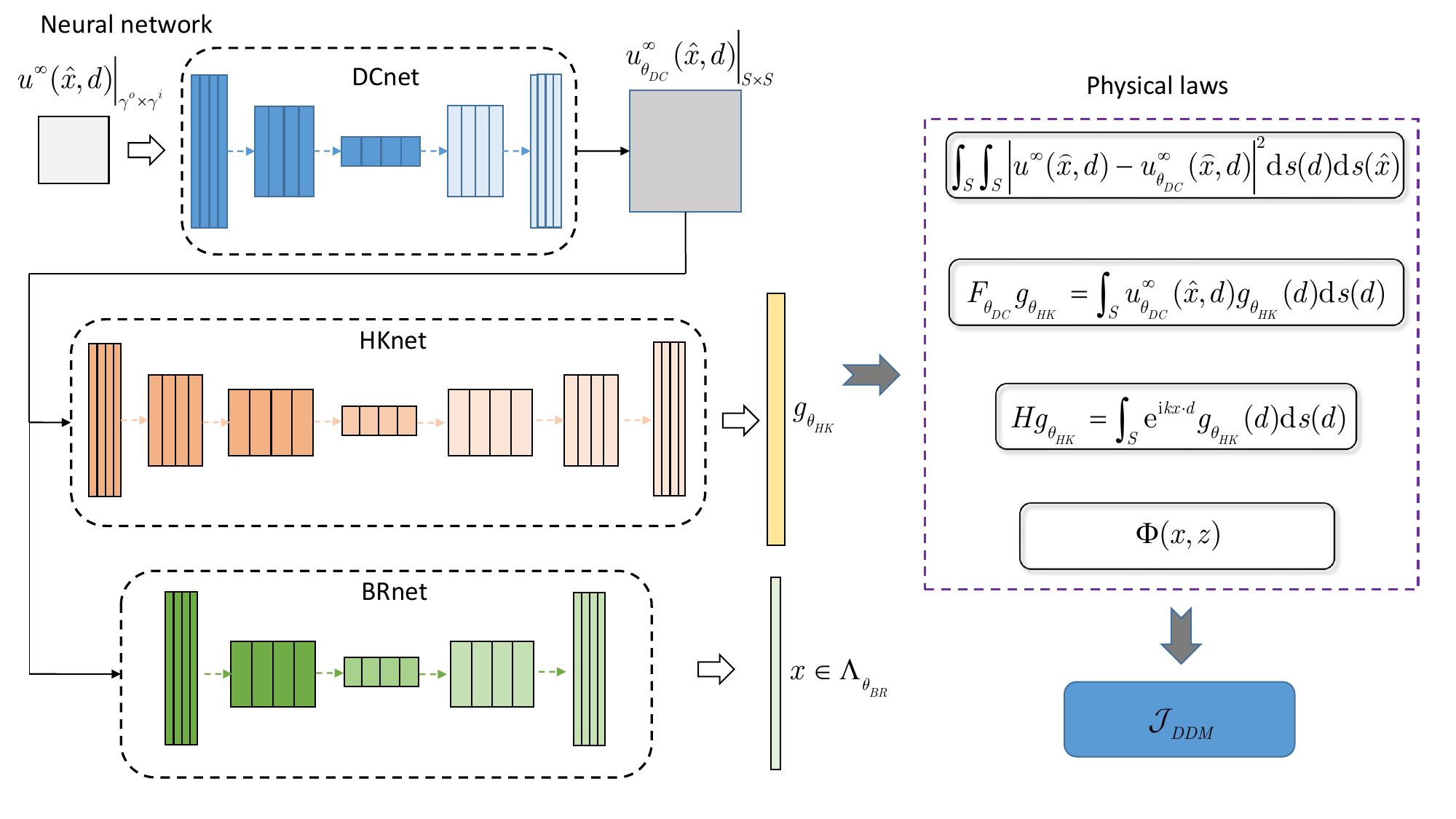}
}
    \caption{\label{fig:DDM_model}A schematic illustration of DDM.}
\end{figure}

\subsection{Convergence analysis}
We are going to investigate the convergence results of the DDM in this subsection.  To this end, we first introduce the Herglotz wave function
\begin{equation}\label{eq:Herglotz wave}
\begin{aligned}
v(x) = \int_{S}\mathrm{e}^{\mathrm{i}kx\cdot d}g(d)\mathrm{d}s(d),\ x\in \mathbb{R}^{2},
\end{aligned}
\end{equation}
which shares the same form of \eqref{eq:v_i}. Moreover, similar to \cite{DRIA}, we set $\gamma=1$ and \textcolor{black}{$\beta_{DC}=1$} for the following theory. Because the Herglotz operator $H$ does not have a bounded inverse, we present the following definition of an optimal curve.
\begin{defn}\label{defn:admissible_set}
Given the limited aperture far-field $u^{\infty}(\hat{x},d) \in L^{2}(\gamma^{o}\times \gamma^{i})$ and a regularization parameter $\alpha\in \mathbb{R}_{+}$, a boundary $\Lambda_{\theta_{BR}^{\ast}}\in U$ is called optimal if there exists an optimal $u_{\theta_{DC}^{\ast}}^{\infty}(\hat{x},d)\in L^{2}(S\times S)$ such that
\begin{equation}\label{eq:mini_J_DDM}
\inf\limits_{g_{\theta_{HK}}\in L^{2}(S)}\mathcal{J}_{DDM}(u_{\theta_{DC}^{\ast}}^{\infty}(\hat{x},d), g_{\theta_{HK}}, \Lambda_{\theta_{BR}^{\ast}};\alpha)=m(\alpha),
\end{equation}
where
$$
m(\alpha)=\inf _{\theta_{DC},\theta_{HK},\theta_{BR}\in\Theta} \mathcal{J}_{DDM}(u_{\theta_{DC}}^{\infty}(\hat{x},d), g_{\theta_{HK}}, \Lambda_{\theta_{BR}};\alpha),
$$
and $U$ is the compact admissible class.
\end{defn}

The next two lemmas, which represent the existence of the optimal curve and the convergence property of the cost functional, respectively, are provided for the DDM convergence analysis.

\begin{lem}\label{lem:opt_lambda}
For each $\alpha\in \mathbb{R}_{+}$, there exists an optimal curve $\Lambda_{\theta_{BR}^{\ast}}\in U$.
\end{lem}
\begin{proof}
Let $(u_{\theta_{DC}^{n}}^{\infty}(\hat{x},d), g_{\theta_{HK}^{n}},\Lambda_{\theta_{BR}^{n}})$ be a minimizing sequence from $L^{2}(S\times S)\times L^{2}(S) \times U$, that is,
$$
\lim\limits_{n\rightarrow \infty}\mathcal{J}_{DDM}(u_{\theta_{DC}^{n}}^{\infty}(\hat{x},d), g_{\theta_{HK}^{n}}, \Lambda_{\theta_{BR}^{n}};\alpha)=m(\alpha).
$$

Assume that $\Lambda_{\theta_{BR}^{n}}\rightarrow \Lambda_{\theta_{BR}^{\ast}},n\rightarrow\infty$, since $U$ is compact. We can further assume that the sequence $(Hg_{\theta_{HK}^{n}})$ is weakly convergent in the space $L^{2}(\Gamma_{\theta_{BR}^{n}})$, due to the following boundedness
$$\alpha\|Hg_{\theta_{HK}^{n}}\|_{L^2\left(\Gamma_{\theta_{BR}^{n}}\right)}^2\leq \mathcal{J}_{DDM}(u_{\theta_{DC}^{n}}^{\infty}(\hat{x},d), g_{\theta_{HK}^{n}}, \Lambda_{\theta_{BR}^{n}};\alpha)\rightarrow m(\alpha),\ n\rightarrow\infty.$$
By Theorem 5.27 in \cite{DRIA}, the weak convergence of the boundary data $(Hg_{\theta_{HK}^{n}})$ on $\Gamma_{\theta_{BR}^{n}}$ implies that the Herglotz wave functions $v_{n}$ with the Herglotz kernel $g_{\theta_{HK}^{n}}$ converge uniformly to a solution of the Helmholtz equation on the compact subsets of the interior of the curve $\Gamma_{\theta_{BR}^{n}}$. Then,
\begin{equation}\label{eq:limit_Hg}
\lim\limits_{n\rightarrow\infty}\left\|Hg_{\theta_{HK}^{n}}+\Phi(x,z)\right\|_{L^2(\Lambda_{\theta_{BR}^{n}})}^2=\lim\limits_{n\rightarrow\infty}\left\|Hg_{\theta_{HK}^{n}}+\Phi(x,z)\right\|_{L^2(\Lambda_{\theta_{BR}^{\ast}})}^2,
\end{equation}
holds. Therefore, it is clear that
$$\lim\limits_{n\rightarrow\infty}\mathcal{J}_{DDM}(u_{\theta_{DC}^{n}}^{\infty}(\hat{x},d), g_{\theta_{HK}^{n}}, \Lambda_{\theta_{BR}^{n}};\alpha)=\lim\limits_{n\rightarrow\infty}\mathcal{J}_{DDM}(u_{\theta_{DC}^{n}}^{\infty}(\hat{x},d), g_{\theta_{HK}^{n}}, \Lambda_{\theta_{BR}^{\ast}};\alpha).$$
This completes the proof.
\end{proof}
\begin{lem}\label{lem:m_alpha_convergence}
Let $u^{\infty}(\hat{x},d)$ be the exact far-field pattern of the obstacle $D$ for all incident directions $d$ such that $\partial D$ belongs to $U$. Moreover, assume that given $\varepsilon_{1}$ and $\varepsilon_2$, the recovered far-field pattern $u_{\theta_{DC}}^{\infty}(\hat{x},d)$ and far-field operator $F_{\theta_{DC}}$ satisfy
\begin{equation}\label{eq:limit_far_farop}
\|u^{\infty}(\hat{x},d)-u_{\theta_{DC}}^{\infty}(\hat{x},d)\|_{L^{2}(S\times S)}<\varepsilon_{1},\ \|F-F_{\theta_{DC}}\|<\varepsilon_2.
\end{equation}
Then, the functional $m(\alpha)$ is convergent to zero when the parameter $\alpha$ tends to zero, i.e.,
\begin{equation}\label{eq:m_alpha_0}
\lim\limits_{\alpha\rightarrow 0}m(\alpha)=0.
\end{equation}
\end{lem}
\begin{proof}
Using Theorem 5.22 in \cite{DRIA} with given $\varepsilon_{3}$, there exists $g_{\theta_{HK}}\in L^{2}(S)$ which satisfies
$$\|H g_{\theta_{HK}}+\Phi(x,z)\|_{L^{2}(\partial D)}< \varepsilon_{3}.$$
Then, it holds that
$$\|Fg_{\theta_{HK}}-\Phi^{\infty}(\hat{x},z)\|_{L^{2}(S)}\leq \|A\|\|Hg_{\theta_{HK}}+\Phi(x,z)\|_{L^{2}(\partial D)}.$$
From \eqref{eq:limit_far_farop}, we further have
\begin{equation}
\begin{aligned}
&\|F_{\theta_{DC}}g_{\theta_{HK}}-\Phi^{\infty}(\hat{x},z)\|_{L^{2}(S)}\\
\leq &\|Fg_{\theta_{HK}}-\Phi^{\infty}(\hat{x},z)\|_{L^{2}(S)}+\|F_{\theta_{DC}}g_{\theta_{HK}}-Fg_{\theta_{HK}}\|_{L^{2}(S)}\\
\leq &\|A\|\|Hg_{\theta_{HK}}+\Phi(x,z)\|_{L^{2}(\partial D)} + \|F_{\theta_{DC}}-F\|\|g_{\theta_{HK}}\|_{L^{2}(S)}\\
\leq & \varepsilon_{3}\|A\|+ \varepsilon_2\|g_{\theta_{HK}}\|_{L^{2}(S)}.
\nonumber
\end{aligned}
\end{equation}
Using \eqref{eq:limit_far_farop} again, we have
\textcolor{black}{\begin{equation}
\begin{aligned}
&\mathcal{J}_{DDM}(u_{\theta_{DC}}^{\infty}(\hat{x},d), g_{\theta_{HK}}, \partial D;\alpha)\\
&\leq\varepsilon^{2}_{1}+(\varepsilon_{3}\|A\|+ \varepsilon_2\|g_{\theta_{HK}}\|_{L^{2}(S)})^{2}+\alpha\|Hg_{\theta_{HK}}\|_{L^2\left(\Gamma_{\theta_{BR}}\right)}^2+\varepsilon^{2}_{3}\\
&\overset{\alpha\rightarrow 0}{\longrightarrow}\varepsilon^{2}_{1} + (\varepsilon_{3}\|A\|+ \varepsilon_2\|g_{\theta_{HK}}\|_{L^{2}(S)})^{2}+\varepsilon^{2}_{3}.
\nonumber
\end{aligned}
\end{equation}}
Since $\varepsilon_1$, $\varepsilon_2$ and $\varepsilon_3$ are arbitrary, \eqref{eq:m_alpha_0} follows. This completes the proof.
\end{proof}

Now, we are ready to present the main convergence result.
\begin{thm}\label{thm:convergence_result}
Let $(\alpha_{n})$ be a null sequence and $(\Lambda_{\theta_{BR}^{n}})$ be a corresponding sequence of optimal curves for the regularization parameter $\alpha_{n}$. Then, there exists a convergent subsequence of $(\Lambda_{\theta_{BR}^{n}})$. Suppose that for all incident directions $u^{\infty}(\hat{x},d)$ is the exact far-field pattern of a sound-soft obstacle $D$ such that $\partial D$ belongs to the compact set $U$. For $j\rightarrow\infty$, the recovered far-field pattern $u_{\theta_{DC}^{j}}^{\infty}(\hat{x},d)$ and far-field operator $F_{\theta_{DC}^{j}}$ are assumed to satisfy
\begin{equation}\label{eq:limit_far_farop_for_convergence}
\|u^{\infty}(\hat{x},d)-u_{\theta_{DC}^{j}}^{\infty}(\hat{x},d)\|_{L^{2}(S\times S)}\rightarrow 0,\ \|F-F_{\theta_{DC}^{j}}\|\rightarrow 0.
\end{equation}
Furthermore, assume that the solution $v^{i}$ to the associated interior Dirichlet problem \eqref{eq:interior_bvp} can be extended as a solution to the Helmholtz equation across the boundary $\partial D$ into the
interior of $\Gamma_{\theta_{BR}}$ with continuous boundary values on $\Gamma_{\theta_{BR}}$. Then every limit point $\Lambda_{\theta_{BR}^{\ast}}$ of $(\Lambda_{\theta_{BR}^{n}})$ denotes the curve on which the boundary condition to the interior Dirichlet problem \eqref{eq:interior_bvp} is satisfied, that is,
\begin{equation}\label{eq:convergence_res_lambda}
\begin{aligned}
v^{i}(x)+\Phi(x,z)=0,\ \ x\in\Lambda_{\theta_{BR}^{\ast}}.
\end{aligned}
\end{equation}
\end{thm}
\begin{proof}
Because of the compactness of $U$, there exists a convergent subsequence of $(\Lambda_{\theta_{BR}^{n}})$. Denote the limit point by $\Lambda_{\theta_{BR}^{\ast}}$. Similarly, assume that $\Lambda_{\theta_{BR}^{n}}\rightarrow \Lambda_{\theta_{BR}^{\ast}},n\rightarrow\infty$. By Theorem 5.22 in \cite{DRIA} and \eqref{eq:v_i}, there exists a sequence $(g_{\theta_{HK}^{j}})$ in $L^{2}(S)$ satisfying
\begin{equation}\label{eq:Hg_vi}
\begin{aligned}
\|Hg_{\theta_{HK}^{j}}-v^{i}\|_{L^{2}(\Gamma_{\theta_{BR}^{j}})}\rightarrow 0,\ \ j\rightarrow\infty.
\end{aligned}
\end{equation}
Then, using Theorem 5.27 in \cite{DRIA}, one can have that the Herglotz wave functions with the kernel $g_{\theta_{HK}^{j}}$ uniformly converge to $v^{i}$ on the compact subsets of the interior of $\Gamma_{\theta_{BR}^{j}}$. Furthermore, in view
of the boundary condition in \eqref{eq:interior_bvp} for $v^{i}$ on $\partial D$, we have
\begin{equation}\label{eq:Hg_Psi}
\begin{aligned}
\|Hg_{\theta_{HK}^{j}}+\Phi(x,z)\|_{L^{2}(\partial D)}\rightarrow 0,\ \ j\rightarrow\infty.
\end{aligned}
\end{equation}
Using the operator $A$, we further have
\begin{equation}\label{eq:Fg_Psi}
\begin{aligned}
\|Fg_{\theta_{HK}^{j}}-\Phi^{\infty}(\hat{x},z)\|_{L^{2}(S)}\rightarrow 0,\ \ j\rightarrow\infty.
\end{aligned}
\end{equation}
\textcolor{black}{By the definition of $m(\alpha)$ we can obtain}
\begin{equation}\label{eq:m_alpha_leg_Jddm_vi}
\textcolor{black}{\begin{aligned}
&m(\alpha)\leq\mathcal{J}_{DC}(u_{\theta_{DC}^{j}}^{\infty}(\hat{x},d))+\|F_{\theta_{DC}^{j}}g_{\theta_{HK}^{j}}-\Phi^{\infty}(\hat{x},z)\|^{2}_{L^{2}(S)}\\
&\quad\quad+\alpha\|Hg_{\theta_{HK}^{j}}\|^{2}_{L^{2}(\Gamma_{\theta_{BR}^{j}})}+\|Hg_{\theta_{HK}^{j}}+\Phi(x,z)\|^{2}_{L^{2}(\partial D)}\\
&\quad\quad\leq\mathcal{J}_{DC}(u_{\theta_{DC}^{j}}^{\infty}(\hat{x},d))+\|F_{\theta_{DC}^{j}}g_{\theta_{HK}^{j}}-\Phi^{\infty}(\hat{x},z)\|^{2}_{L^{2}(S)}\\
&\quad\quad+\alpha\|Hg_{\theta_{HK}^{j}}-v^{i}\|^{2}_{L^{2}(\Gamma_{\theta_{BR}^{j}})}+\alpha\|v^{i}\|^{2}_{L^{2}(\Gamma_{\theta_{BR}^{j}})}+\|Hg_{\theta_{HK}^{j}}+\Phi(x,z)\|^{2}_{L^{2}(\partial D)}.
\end{aligned}}
\end{equation}
Then for $j\rightarrow\infty$ and all $\alpha\in \mathbb{R}_{+}$, combining \eqref{eq:limit_far_farop_for_convergence} and \eqref{eq:Hg_vi}-\eqref{eq:m_alpha_leg_Jddm_vi} together gives rise to
\begin{equation}\label{eq:m_alpha_leg_vi}
\begin{aligned}
m(\alpha)\leq \alpha\|v^{i}\|^{2}_{L^{2}(\Gamma_{\theta_{BR}^{j}})},
\end{aligned}
\end{equation}
In addition, by Lemma \ref{lem:opt_lambda}, for each $n$ there exist $u_{\theta_{DC}^{n}}^{\infty}(\hat{x},d)\in L^{2}(S\times S)$ and $g_{\theta_{HK}^{n}}\in L^{2}(S)$ satisfying
\begin{equation}\label{eq:thm1_use}
\begin{aligned}
\mathcal{J}_{DDM}(u_{\theta_{DC}^{n}}^{\infty}(\hat{x},d), g_{\theta_{HK}^{n}}, \Lambda_{\theta_{BR}^{n}};\alpha_{n}) \leq m(\alpha_{n}) + \alpha_{n}^{2}.
\end{aligned}
\end{equation}
By \eqref{eq:m_alpha_leg_vi}, \eqref{eq:thm1_use} and $\alpha_{n}\|Hg_{\theta_{HK}^{n}}\|^{2}_{L^{2}(\Gamma_{\theta_{BR}^{n}})}\leq\mathcal{J}_{DDM}(u_{\theta_{DC}^{n}}^{\infty}(\hat{x},d), g_{\theta_{HK}^{n}}, \Lambda_{\theta_{BR}^{n}};\alpha_{n})$, for all $n$ we have that
\begin{equation}
\begin{aligned}
\|Hg_{\theta_{HK}^{n}}\|^{2}_{L^{2}(\Gamma_{\theta_{BR}^{n}})} \leq \|v^{i}\|^{2}_{L^{2}(\Gamma_{\theta_{BR}^{n}})} + \alpha_{n},
\nonumber
\end{aligned}
\end{equation}
which implies that the sequence $Hg_{\theta_{HK}^{n}}$ is weakly convergent in $L^{2}(\Gamma_{\theta_{BR}^{n}})$. Then, by Theorem 5.27 in \cite{DRIA}, the Herglotz wave functions $v_{n}$ with kernels $g_{\theta_{HK}^{n}}$ converge uniformly to the solution $v^{\ast}$ of the Helmholtz equation on compact subsets of the interior of $\Gamma_{\theta_{BR}^{n}}$. Moreover, by Lemma \ref{lem:m_alpha_convergence}, we obtain $m(\alpha_{n})\rightarrow 0,n\rightarrow\infty$. From \eqref{eq:thm1_use}, we also have
\begin{equation}
\begin{aligned}
\|F_{\theta_{DC}^{n}}g_{\theta_{HK}^{n}}-\Phi^{\infty}(\hat{x},z)\|^{2}_{L^{2}(S)}\leq m(\alpha_{n}) + \alpha_{n}^{2}.
\nonumber
\end{aligned}
\end{equation}
Therefore, there exists $v_{n}$ which satisfies $A(v_{n}+\Phi(x,z))=Fg_{\theta_{HK}^{n}}-\Phi^{\infty}(\hat{x},z)$ for the boundary $\partial D$, then
\begin{equation}
\begin{aligned}
&\|A(v_{n}+\Phi(x,z))\|_{L^{2}(S)}
=\|Fg_{\theta_{HK}^{n}}-\Phi^{\infty}(\hat{x},z)\|^{2}_{L^{2}(S)}\\
\leq &m(\alpha_{n}) + \alpha_{n}^{2}+ \|Fg_{\theta_{HK}^{n}}-F_{\theta_{DC}^{n}}g_{\theta_{HK}^{n}}\|^{2}_{L^{2}(S)}.
\nonumber
\end{aligned}
\end{equation}
By \eqref{eq:limit_far_farop_for_convergence}, for $n\rightarrow\infty$, we have
\begin{equation}
\begin{aligned}
\|A(v_{n}+\Phi(x,z))\|_{L^{2}(S)}\rightarrow 0,
\nonumber
\end{aligned}
\end{equation}
and
\begin{equation}
\begin{aligned}
A(v^{\ast}+\Phi(x,z))=0.
\nonumber
\end{aligned}
\end{equation}
The fact that the operator $A$ is injective yields
\begin{equation}\label{eq:v_star_bc}
\begin{aligned}
v^{\ast}+\Phi(x,z)=0, \ \ \mathrm{on}\ \partial D.
\end{aligned}
\end{equation}
From \eqref{eq:interior_bvp} and \eqref{eq:v_star_bc}, we find that $v^{\ast}$ and $v^{i}$ have the same boundary condition on $\partial D$, implying that they possess the same Herglotz kernel. Therefore, one can further have $v^{i}=v^{\ast}$ since $k^2$ is assumed not to be the Dirichlet eigenvalue for the obstacle $D$. From \eqref{eq:thm1_use}, we obtain
\begin{equation}\label{eq:bc_convergence}
\begin{aligned}
\|v_{n}+\Phi(x,z)\|^{2}_{L^{2}(\Lambda_{\theta_{BR}^{n}})}\leq m(\alpha_{n}) + \alpha_{n}^{2}\rightarrow 0, \ n\rightarrow \infty,
\end{aligned}
\end{equation}
which then gives rise to \eqref{eq:convergence_res_lambda}. This completes the proof.
\end{proof}

Under additional assumptions, we can further get the convergence towards the exact boundary of the obstacle.
\begin{thm}\label{thm:convergence_result_tend_exact_boundary}
Assume that $D$ is contained in a circle $C_{r_{b}}$ of radius $r_{b}$ such that $0<k<\frac{k_{01}}{r_{b}}$, where $k_{01}$ means the first zero of the Bessel function $J_{0}$. Then the sequence $\{\Lambda_{\theta_{BR}^{n}}\}$ only has one limit point $\Lambda_{\theta_{BR}^{\ast}}$, which coincides with $\partial {D}$.
\end{thm}
\begin{proof}
Suppose that $\Lambda_{\theta_{BR}^{(1)}},\Lambda_{\theta_{BR}^{(2)}}\in U$ are distinct limit points of $\{\Lambda_{\theta_{BR}^{n}}\}$, which are respectively related to two obstacles $D_{1}, D_{2}$. Without loss of generality, let the domain $D_{12}:=D_{1}\backslash(\overline{D}_{1}\cap \overline{D}_{2})$ be non-empty and $D_{12}\subset C_{r_{b}}\backslash \overline{D}$. By the Theorem \ref{thm:convergence_result}, $v^{i}(x)+\Phi(x,z)$ is an eigenfunction of the negative Laplacian in $D_{12}$ with the Dirichlet eigenvalue $k^{2}$. For $C_{r_{b}}$, the smallest positive eigenvalue is $\frac{k_{01}^{2}}{r_{b}^{2}}$. Denote $k^{2}_{D_{12}}$ the first positive eigenvalue of the negative Laplacian in $D_{12}$. Since $D_{12}\subset C_{r_{b}}$, we can get $\frac{k_{01}^{2}}{r_{b}^{2}}<k^{2}_{D_{12}}$. By the assumption $0<k<\frac{k_{01}}{r_{b}}$, we further have $k<k_{D_{12}}$, this contradicts the monotonicity of the Dirichlet eigenvalues. Therefore, if $0<k<\frac{k_{01}}{r_{b}}$, the sequence $\{\Lambda_{\theta_{BR}^{n}}\}$ only has one limit point $\Lambda_{\theta_{BR}^{\ast}}$. Since $D$ is contained in a circle $C_{r_{b}}$ and $v^{i}(x)+\Phi(x,z)=0$ is also satisfied on the boundary $\partial D$, then the only one limit point is $\Lambda_{\theta_{BR}^{\ast}}=\partial D$.
\end{proof}

\begin{rem}\label{rem_ddm_convergcen_conditions}
In fact, the subsequent numerical experiments also verify the convergence of DDM. Furthermore, although we restrict ourselves to the limited aperture case, clearly, the methodology of DDM can be applied to the full aperture case in which DCnet, the condition \eqref{eq:limit_far_farop} and the condition \eqref{eq:limit_far_farop_for_convergence} are not required. In addition, in the full aperture case, higher accuracy can be achieved because we do not need to deal with  the approximation error generated by \eqref{eq:DC_functional}.
\end{rem}

\subsection{The discretization of DDM}
This section presents the discretization of DDM for the subsequent numerical computation. To this end, we first give the discretization of the exact full aperture data $u^{\infty}(\hat{x},d)|_{S \times S}$, which is regarded as multi-static response matrix (MSRM) $\mathbb{F}_{f}$. Taking $\tau_{i}:=\frac{(i-1)\pi}{m},i=1,2,\cdots,2m$, we can determine the direction of the incident and observation by
\begin{equation}
\begin{aligned}
d_{i}:=(\cos \tau_{i}, \sin \tau_{i}),\ i=1,2,\cdots,2m,
\nonumber
\end{aligned}
\end{equation}
and
\begin{equation}
\begin{aligned}
\hat{x}_{j}:=(\cos \tau_{j}, \sin \tau_{j}),\ j=1,2,\cdots,2m.
\nonumber
\end{aligned}
\end{equation}
Thus, the corresponding MSRM is
\begin{equation}\label{eq:MSR}
    \begin{aligned}
    \mathbb{F}_{f}=\begin{bmatrix} u^{\infty}(\hat{x}_{1},d_{1}),u^{\infty}(\hat{x}_{1},d_{2}),\cdots, u^{\infty}(\hat{x}_{1},d_{2m})\\ u^{\infty}(\hat{x}_{2},d_{1}),u^{\infty}(\hat{x}_{2},d_{2}),\cdots, u^{\infty}(\hat{x}_{2},d_{2m})\\ \vdots\ \ \ \ \ \ \ \ \ \ \vdots\ \ \ \ \ \ \ \ \ \ \vdots\ \ \ \ \ \ \ \ \ \ \vdots \\
    u^{\infty}(\hat{x}_{2m},d_{1}),u^{\infty}(\hat{x}_{2m},d_{2}),\cdots, u^{\infty}(\hat{x}_{2m},d_{2m})\end{bmatrix}.
    \end{aligned}
\end{equation}
Denote the limited incident aperture $\gamma^{i}$ and limited observation aperture $\gamma^{o}$ by
\begin{equation}
\begin{aligned}
\gamma^{i}:=(\cos \psi, \sin \psi), \ \psi\subseteq [0,2\pi],
\nonumber
\end{aligned}
\end{equation}
and
\begin{equation}
\begin{aligned}
\gamma^{o}:=(\cos \phi, \sin \phi), \ \phi\subset [0,2\pi].
\nonumber
\end{aligned}
\end{equation}
Similar to MSRM, the discretization of the DDM input, which is the exact limited aperture data $u^{\infty}(\hat{x},d)|_{\gamma^{o} \times \gamma^{i}}$, is defined as
\begin{equation}\label{eq:limited_data_discre}
    \begin{aligned}
    \mathbb{F}_{l}=\begin{bmatrix} u^{\infty}(\hat{x}_{n^{o}},d_{n^{i}}),u^{\infty}(\hat{x}_{n^{o}},d_{n^{i}+1}),\cdots, u^{\infty}(\hat{x}_{n^{o}},d_{N^{i}})\\ u^{\infty}(\hat{x}_{n^{o}+1},d_{n^{i}}),u^{\infty}(\hat{x}_{n^{o}+1},d_{n^{i}+1}),\cdots, u^{\infty}(\hat{x}_{n^{o}+1},d_{N^{i}})\\ \ \ \ \ \vdots\ \ \ \ \ \ \ \ \ \ \ \ \ \ \vdots\ \ \ \ \ \ \ \ \ \ \ \ \ \ \ \ \vdots\ \ \ \ \ \ \ \ \ \ \ \ \ \ \vdots \\
    u^{\infty}(\hat{x}_{N^{o}},d_{n^{i}}),u^{\infty}(\hat{x}_{N^{o}},d_{n^{i}+1}),\cdots, u^{\infty}(\hat{x}_{N^{o}},d_{N^{i}})\end{bmatrix},
    \end{aligned}
\end{equation}
with the same step size $\frac{\pi}{m}$. For example, if $\phi=[0,\pi]$ and $\psi=[0,2\pi]$, then $n^{o}=1$, $N^{o}=m$, $n^{i}=1$ and $N^{i}=2m$. If $\phi=[0,\pi/2]$ and $\psi=[\pi/2,3\pi/2]$, then $n^{o}=1$, $N^{o}=m/2$, $n^{i}=m/2+1$ and $N^{i}=3m/2$. In addition, the discretization of the recovered full aperture data $u_{\theta_{DC}}^{\infty}(\hat{x},d)|_{S \times S}$ by DCnet is given by
\begin{equation}\label{eq:MSR_recovered}
    \begin{aligned}
    \mathbb{F}_{\theta_{DC},f}=\begin{bmatrix} u_{\theta_{DC}}^{\infty}(\hat{x}_{1},d_{1}),u_{\theta_{DC}}^{\infty}(\hat{x}_{1},d_{2}),\cdots, u_{\theta_{DC}}^{\infty}(\hat{x}_{1},d_{2m})\\ u_{\theta_{DC}}^{\infty}(\hat{x}_{2},d_{1}),u_{\theta_{DC}}^{\infty}(\hat{x}_{2},d_{2}),\cdots, u_{\theta_{DC}}^{\infty}(\hat{x}_{2},d_{2m})\\ \vdots\ \ \ \ \ \ \ \ \ \ \vdots\ \ \ \ \ \ \ \ \ \ \ \ \ \ \vdots\ \ \ \ \ \ \ \ \ \ \ \ \vdots \\
    u_{\theta_{DC}}^{\infty}(\hat{x}_{2m},d_{1}),u_{\theta_{DC}}^{\infty}(\hat{x}_{2m},d_{2}),\cdots, u_{\theta_{DC}}^{\infty}(\hat{x}_{2m},d_{2m})\end{bmatrix}.
    \end{aligned}
\end{equation}
Therefore, using \eqref{eq:MSR} and \eqref{eq:MSR_recovered} for discretizing the functional \eqref{eq:DC_functional} derives
\begin{equation}\label{eq:DC_functional_discre}
\widetilde{\mathcal{J}}_{DC}(\mathbb{F}_{\theta_{DC},f}):= (\frac{\pi}{m})^2 \sum\limits_{i=1}^{2m}\sum\limits_{j=1}^{2m}|\mathbb{F}_{f}^{(ji)}-\mathbb{F}_{\theta_{DC},f}^{(ji)}|^{2},
\end{equation}
where
\begin{equation}
\begin{aligned}
\mathbb{F}_{f}^{(ji)}=u^{\infty}(\hat{x}_{j},d_{i}),\ \mathbb{F}_{\theta_{DC},f}^{(ji)}=u_{\theta_{DC}}^{\infty}(\hat{x}_{j},d_{i}).
\nonumber
\end{aligned}
\end{equation}
Next, we shall discretize the recovered Herglotz kernel $g_{\theta_{HK}}(d)$ and recovered curve $\Lambda_{\theta_{BR}}$. Discretizing $g_{\theta_{HK}}(d)$ yields
\begin{equation}\label{eq:g_recovered_discre}
    \begin{aligned}
    \widetilde{g}_{\theta_{HK}}=\begin{bmatrix} g_{\theta_{HK}}(d_{1})\\ g_{\theta_{HK}}(d_{2})\\ \vdots\ \\
    g_{\theta_{HK}}(d_{2m})\end{bmatrix}.
    \end{aligned}
\end{equation}
To work with $\Lambda_{\theta_{BR}}$ numerically, we assume $\Lambda_{\theta_{BR}}$ is a starlike curve with respect to the origin, which is defined by
\begin{equation}\label{eq:lambda_star}
    \begin{aligned}
    \Lambda_{\theta_{BR}} :=\mathrm{e}^{q_{\theta_{BR}}(t)}(\mathrm{cos}t,\mathrm{sin}t),\ t \in\ [0,2\pi],
    \end{aligned}
\end{equation}
where $q_{\theta_{BR}}(t)$ is in the form of truncated fourier expansion
\begin{equation}\label{eq:q}
    \begin{aligned}
    q_{\theta_{BR}}(t)=\frac{q_{0}}{\sqrt{2\pi}}+\sum\limits_{n=1}^{N_{\Lambda}}\left(\frac{a_{n}}{n^{s}}\frac{\mathrm{cos}(nt)}{\sqrt{\pi}}+\frac{b_{n}}{n^{s}}\frac{\mathrm{sin}(nt)}{\sqrt{\pi}}\right),
    \end{aligned}
\end{equation}
and $N_{\Lambda}$ is the cut-off frequency, $s$ controls the decreasing rate of the corresponding Fourier coefficients. Instead of directly expanding $\mathrm{e}^{q_{\theta_{BR}}(t)}$, the expansion form \eqref{eq:q} is conducive to ensuring the back propagation of DDM. Note that $\Lambda_{\theta_{BR}}$ is uniquely determined by a finite set
\begin{equation}\label{eq:Q_set}
    \begin{aligned}
    Q_{\theta_{BR}}=(q_{0},a_{1},b_{1},a_{2},b_{2},\cdots,a_{N_{\Lambda}},b_{N_{\Lambda}})\in\mathbb{R}^{2N_{\Lambda}+1},
    \end{aligned}
\end{equation}
which implies that determining $\Lambda_{\theta_{BR}}$ is equivalent to determining $Q_{\theta_{BR}}$. Let $t_{l}:=\frac{2(l-1)\pi}{N_{t}},l=1,2,\cdots,N_{t}$, where $N_{t}$ is the number of surface discretization. Then, each boundary point $x^{(l)}_{\Lambda_{\theta_{BR}}}\in \Lambda_{\theta_{BR}}$ is defined by
\begin{equation}\label{eq:x_lambda_discre}
    \begin{aligned}
    x^{(l)}_{\Lambda_{\theta_{BR}}}:=\mathrm{e}^{q_{\theta_{BR}}(t_{l})}(\mathrm{cos}t_{l},\mathrm{sin}t_{l}),\ l=1,2,\cdots,N_{t}.
    \end{aligned}
\end{equation}
Moreover, throughout the paper, to meet the above theoretical requirement, we set $\Gamma_{\theta_{BR}}=1.001\Lambda_{\theta_{BR}}$ to ensure that $\Lambda_{\theta_{BR}}$ is contained in the interior of $\Gamma_{\theta_{BR}}$. Clearly, each point $x^{(l)}_{\Gamma_{\theta_{BR}}}\in \Gamma_{\theta_{BR}}$ is $x^{(l)}_{\Gamma_{\theta_{BR}}}=1.001 x^{(l)}_{\Lambda_{\theta_{BR}}}$. Therefore, discretizing the functional \eqref{eq:CM_functional_theta} derives
\begin{equation}\label{eq:CM_functional_theta_discre}
\begin{aligned}
&\widetilde{\mathcal{J}}_{phy}(\widetilde{g}_{\theta_{HK}}, Q_{\theta_{BR}};\alpha)\\
:=&\frac{\pi}{m}\sum\limits_{j=1}^{2m}\left|\frac{\pi}{m}\sum\limits_{i=1}^{2m}\mathbb{F}_{\theta_{DC},f}^{(ji)}\widetilde{g}_{\theta_{HK}}^{(i)}-\Phi^{\infty}(\hat{x}_{j},z)\right|^{2}\\
+&\alpha\frac{2\pi}{N_{t}}\sum\limits_{l=1}^{N_{t}}\left|\frac{\pi}{m}\sum\limits_{i=1}^{2m}\mathrm{e}^{\mathrm{i}kx^{(l)}_{\Gamma_{\theta_{BR}}}\cdot d_{i}}\widetilde{g}_{\theta_{HK}}^{(i)}\right|^{2}\\
+&\textcolor{black}{\gamma}\frac{2\pi}{N_{t}}\sum\limits_{l=1}^{N_{t}}\left|\frac{\pi}{m}\sum\limits_{i=1}^{2m}\mathrm{e}^{\mathrm{i}kx^{(l)}_{\Lambda_{\theta_{BR}}}\cdot d_{i}}\widetilde{g}_{\theta_{HK}}^{(i)}+\Phi(x^{(l)}_{\Lambda_{\theta_{BR}}},z)\right|^{2},
\end{aligned}
\end{equation}
where $\widetilde{g}_{\theta_{HK}}^{(i)}=g_{\theta_{HK}}(d_{i})$. We would like to point out that, the Jacobian terms $\mathrm{e}^{q_{\theta_{BR}}(t_{l})}\sqrt{1+(q_{\theta_{BR}}^{'}(t_{l}))^{2}}$ and $1.001\mathrm{e}^{q_{\theta_{BR}}(t_{l})}\sqrt{1+(q_{\theta_{BR}}^{'}(t_{l}))^{2}}$ are omitted in \eqref{eq:CM_functional_theta_discre} for simplicity, which will not destroy the solution of the inverse problem. Finally, to learn $\theta_{DC},\theta_{HK},\theta_{BR}$ in DDM, we focus on minimizing
\begin{equation}\label{eq:DDM_functional_theta_discre}
\widetilde{\mathcal{J}}_{DDM}(\mathbb{F}_{\theta_{DC},f},\widetilde{g}_{\theta_{HK}}, x^{(l)}_{\Lambda_{\theta_{BR}}};\alpha)=\widetilde{\mathcal{J}}_{phy}(\widetilde{g}_{\theta_{HK}}, Q_{\theta_{BR}};\alpha)+\beta_{DC}\widetilde{\mathcal{J}}_{DC}(\mathbb{F}_{\theta_{DC},f}).
\end{equation}

The aforementioned discussions are based on the noise-free case. \textcolor{black}{Next, similar to \cite{NB1}, we show that adding the relatively small noise to the input $\mathbb{F}_{l}$ will promote the smoothness of the learned inverse map and introduce regularization terms.} Since $\mathbb{F}_{\theta_{DC},f}$, $\widetilde{g}_{\theta_{HK}}$ and $\Lambda_{\theta_{BR}}$ are functions with respect to $\mathbb{F}_{l}$, we write $\widetilde{\mathcal{J}}_{DDM}(\mathbb{F}_{l})$ to specify the dependence of DDM on the input $\mathbb{F}_{l}$. Let $\mathbb{E}[\cdot]$ signify the expectation. To analyze the influence of the noise, we perturb the limited aperture data $\mathbb{F}_{l}$ in the form
\begin{equation}\label{eq:noisy_input_form}
\begin{aligned}
\mathbb{F}_{l,\eta}=\mathbb{F}_{l}+\eta,
\end{aligned}
\end{equation}
where $\mathbb{E}[\eta]=0$ and the elements of $\eta$ are independent of each other. Then the following theorem holds.
\begin{thm}\label{thm:expectation_noisy}
Let $\mathbb{F}_{l,\eta}$ be the noisy input defined in the form \eqref{eq:noisy_input_form}, then
\begin{equation}\label{eq:noisy_input_J_ddm}
\begin{aligned}
\mathbb{E}[\widetilde{\mathcal{J}}_{DDM}(\mathbb{F}_{l,\eta})]=&\widetilde{\mathcal{J}}_{DDM}(\mathbb{F}_{l})+\frac{1}{2}\sum\limits_{j_{1}=n^{o}}^{N^{o}}\sum\limits_{\widetilde{j}_{1}=n^{i}}^{N^{i}}\frac{\partial ^{2}\widetilde{\mathcal{J}}_{DDM}}{\partial \mathbb{F}^{(j_{1}\widetilde{j}_{1})}_{l} \partial \mathbb{F}^{(j_{1}\widetilde{j}_{1})}_{l}}\mathbb{E}[(\eta^{(j_{1}\widetilde{j}_{1})})^{2}]\\
&+\mathbb{E}[o(\|\eta\|^{2}_{F})],
\end{aligned}
\end{equation}
where $\eta^{(j_{1}\widetilde{j}_{1})}$ is the element located at the $j_{1}\mathrm{th}$ row and the $\widetilde{j}_{1}\mathrm{th}$ column of $\eta$, and $\|\cdot\|_{F}$ is the Frobenius norm.
\end{thm}
\begin{proof}
For $\widetilde{\mathcal{J}}_{DDM}(\mathbb{F}_{l,\eta})$, we perform the Taylor expansion around $\mathbb{F}_{l}$ up to second-order to get
\begin{equation}
\begin{aligned}
\widetilde{\mathcal{J}}_{DDM}(\mathbb{F}_{l,\eta})=&\widetilde{\mathcal{J}}_{DDM}(\mathbb{F}_{l})+\sum\limits_{j_{1}=n^{o}}^{N^{o}}\sum\limits_{\widetilde{j}_{1}=n^{i}}^{N^{i}}\frac{\partial \widetilde{\mathcal{J}}_{DDM}}{\partial \mathbb{F}^{(j_{1}\widetilde{j}_{1})}_{l}}\eta^{(j_{1}\widetilde{j}_{1})}\\
&+\frac{1}{2}\sum\limits_{j_{2}=n^{o}}^{N^{o}}\sum\limits_{\widetilde{j}_{2}=n^{i}}^{N^{i}}\sum\limits_{j_{1}=n^{o}}^{N^{o}}\sum\limits_{\widetilde{j}_{1}=n^{i}}^{N^{i}}\frac{\partial ^{2}\widetilde{\mathcal{J}}_{DDM}}{\partial \mathbb{F}^{(j_{2}\widetilde{j}_{2})}_{l} \partial \mathbb{F}^{(j_{1}\widetilde{j}_{1})}_{l}}\eta^{(j_{2}\widetilde{j}_{2})}\eta^{(j_{1}\widetilde{j}_{1})}+o(\|\eta\|^{2}_{F}).
\nonumber
\end{aligned}
\end{equation}
Since $\mathbb{E}[\eta]=0$ and the elements of $\eta$ are independent of each other, then it immediately follows that $\mathbb{E}[\widetilde{\mathcal{J}}_{DDM}(\mathbb{F}_{l,\eta})]$ is in the form \eqref{eq:noisy_input_J_ddm}.
\end{proof}

It is noted in \eqref{eq:noisy_input_J_ddm} that $\mathbb{E}[\widetilde{\mathcal{J}}_{DDM}(\mathbb{F}_{l,\eta})]$ is the sum of the original loss $\widetilde{\mathcal{J}}_{DDM}(\mathbb{F}_{l})$ and the induced regularization terms. These regularization terms deeply rely on $\frac{\partial ^{2}\widetilde{\mathcal{J}}_{DDM}}{\partial \mathbb{F}^{(j_{1}\widetilde{j}_{1})}_{l} \partial \mathbb{F}^{(j_{1}\widetilde{j}_{1})}_{l}}$. Since $\widetilde{\mathcal{J}}_{DC}$ is mainly concerned with the data completion, we only need to focus on $\frac{\partial ^{2}\widetilde{\mathcal{J}}_{phy}}{\partial \mathbb{F}^{(j_{1}\widetilde{j}_{1})}_{l} \partial \mathbb{F}^{(j_{1}\widetilde{j}_{1})}_{l}}$. Here, $\frac{\partial ^{2}\widetilde{\mathcal{J}}_{phy}}{\partial \mathbb{F}^{(j_{1}\widetilde{j}_{1})}_{l} \partial \mathbb{F}^{(j_{1}\widetilde{j}_{1})}_{l}}$ can be written as
\begin{equation}
\begin{aligned}
\frac{\partial ^{2}\widetilde{\mathcal{J}}_{phy}}{\partial \mathbb{F}^{(j_{1}\widetilde{j}_{1})}_{l} \partial \mathbb{F}^{(j_{1}\widetilde{j}_{1})}_{l}}=\mathcal{P}_{1}+\mathcal{P}_{2}+\mathcal{P}_{3},
\nonumber
\end{aligned}
\end{equation}
where
\begin{equation}
\begin{aligned}
&\mathcal{P}_{1}=\frac{2\pi}{m}\sum\limits_{j=1}^{2m}\left(\frac{\pi}{m}\sum\limits_{i=1}^{2m}\frac{\partial \mathbb{F}_{\theta_{DC},f}^{(ji)}\widetilde{g}_{\theta_{HK}}^{(i)}}{\partial \mathbb{F}^{(j_{1}\widetilde{j}_{1})}_{l} }\right)^{2}\\
&\quad\quad +\frac{2\pi}{m}\sum\limits_{j=1}^{2m}\left(\frac{\pi}{m}\sum\limits_{i=1}^{2m}\mathbb{F}_{\theta_{DC},f}^{(ji)}\widetilde{g}_{\theta_{HK}}^{(i)}-\Phi^{\infty}(\hat{x}_{j},z)\right)\left(\frac{\pi}{m}\sum\limits_{i=1}^{2m}\frac{\partial ^{2} \mathbb{F}_{\theta_{DC},f}^{(ji)}\widetilde{g}_{\theta_{HK}}^{(i)}}{\partial \mathbb{F}^{(j_{1}\widetilde{j}_{1})}_{l} \partial \mathbb{F}^{(j_{1}\widetilde{j}_{1})}_{l}}\right),
\nonumber
\end{aligned}
\end{equation}
\begin{equation}
\begin{aligned}
&\mathcal{P}_{2}=\alpha\frac{4\pi}{N_{t}}\sum\limits_{l=1}^{N_{t}}\left(\frac{\pi}{m}\sum\limits_{i=1}^{2m}\frac{\partial \mathrm{e}^{\mathrm{i}kx^{(l)}_{\Gamma_{\theta_{BR}}}\cdot d_{i}}\widetilde{g}_{\theta_{HK}}^{(i)}}{\partial \mathbb{F}^{(j_{1}\widetilde{j}_{1})}_{l}}\right)^{2}\\
&\quad\quad+\alpha\frac{4\pi}{N_{t}}\sum\limits_{l=1}^{N_{t}}\left(\frac{\pi}{m}\sum\limits_{i=1}^{2m}
\mathrm{e}^{\mathrm{i}kx^{(l)}_{\Gamma_{\theta_{BR}}}\cdot d_{i}}\widetilde{g}_{\theta_{HK}}^{(i)}\right)\left(\frac{\pi}{m}\sum\limits_{i=1}^{2m}\frac{\partial^{2} \mathrm{e}^{\mathrm{i}kx^{(l)}_{\Gamma_{\theta_{BR}}}\cdot d_{i}}\widetilde{g}_{\theta_{HK}}^{(i)}}{\partial \mathbb{F}^{(j_{1}\widetilde{j}_{1})}_{l}\partial \mathbb{F}^{(j_{1}\widetilde{j}_{1})}_{l}}\right),
\nonumber
\end{aligned}
\end{equation}
and
\begin{equation}
\begin{aligned}
&\mathcal{P}_{3}=\textcolor{black}{\gamma}\frac{4\pi}{N_{t}}\sum\limits_{l=1}^{N_{t}}\left(\frac{\partial \Phi(x^{(l)}_{\Lambda_{\theta_{BR}}},z)}{\partial \mathbb{F}^{(j_{1}\widetilde{j}_{1})}_{l}}+\frac{\pi}{m}\sum\limits_{i=1}^{2m}\frac{\partial \mathrm{e}^{\mathrm{i}kx^{(l)}_{\Lambda_{\theta_{BR}}}\cdot d_{i}}\widetilde{g}_{\theta_{HK}}^{(i)}}{\partial \mathbb{F}^{(j_{1}\widetilde{j}_{1})}_{l}}\right)^{2}\\
&+\textcolor{black}{\gamma}\frac{4\pi}{N_{t}}\sum\limits_{l=1}^{N_{t}}\left(\Phi(x^{(l)}_{\Lambda_{\theta_{BR}}},z)+\frac{\pi}{m}\sum\limits_{i=1}^{2m}
\mathrm{e}^{\mathrm{i}kx^{(l)}_{\Lambda_{\theta_{BR}}}\cdot d_{i}}\widetilde{g}_{\theta_{HK}}^{(i)}\right)\left(\frac{\partial^{2} \Phi(x^{(l)}_{\Lambda_{\theta_{BR}}},z)}{\partial \mathbb{F}^{(j_{1}\widetilde{j}_{1})}_{l}\partial \mathbb{F}^{(j_{1}\widetilde{j}_{1})}_{l}}+\frac{\pi}{m}\sum\limits_{i=1}^{2m}\frac{\partial^{2} \mathrm{e}^{\mathrm{i}kx^{(l)}_{\Lambda_{\theta_{BR}}}\cdot d_{i}}\widetilde{g}_{\theta_{HK}}^{(i)}}{\partial \mathbb{F}^{(j_{1}\widetilde{j}_{1})}_{l}\partial \mathbb{F}^{(j_{1}\widetilde{j}_{1})}_{l}}\right).
\nonumber
\end{aligned}
\end{equation}
\textcolor{black}{It is observed that $\mathcal{P}_{1}$, $\mathcal{P}_{2}$ and $\mathcal{P}_{3}$ can be regarded as additionally induced regularization terms. In addition, they provide some constraints on the first and second derivatives of $\mathbb{F}_{\theta_{DC},f}^{(ji)}\widetilde{g}_{\theta_{HK}}^{(i)}$, $
\mathrm{e}^{\mathrm{i}kx^{(l)}_{\Gamma_{\theta_{BR}}}\cdot d_{i}}\widetilde{g}_{\theta_{HK}}^{(i)}$ and $
\mathrm{e}^{\mathrm{i}kx^{(l)}_{\Lambda_{\theta_{BR}}}\cdot d_{i}}\widetilde{g}_{\theta_{HK}}^{(i)}$, hence improving the smoothness of the learned $\mathbb{F}_{\theta_{DC},f}$, $\widetilde{g}_{\theta_{HK}}$ and $\Lambda_{\theta_{BR}}$.}

\section{Numerical experiments}
In this section, we will present numerical examples to show the effectiveness of DDM for the limited aperture inverse scattering problem \eqref{eq:limited_inverse_scattering}. \textcolor{black}{All the following numerical examples are carried out on a computing server with a NVIDIA GeForce RTX 3090 GPU.}

\begin{figure}[t]
    \centering
    \subfigure{
    \includegraphics[width=.95\textwidth]{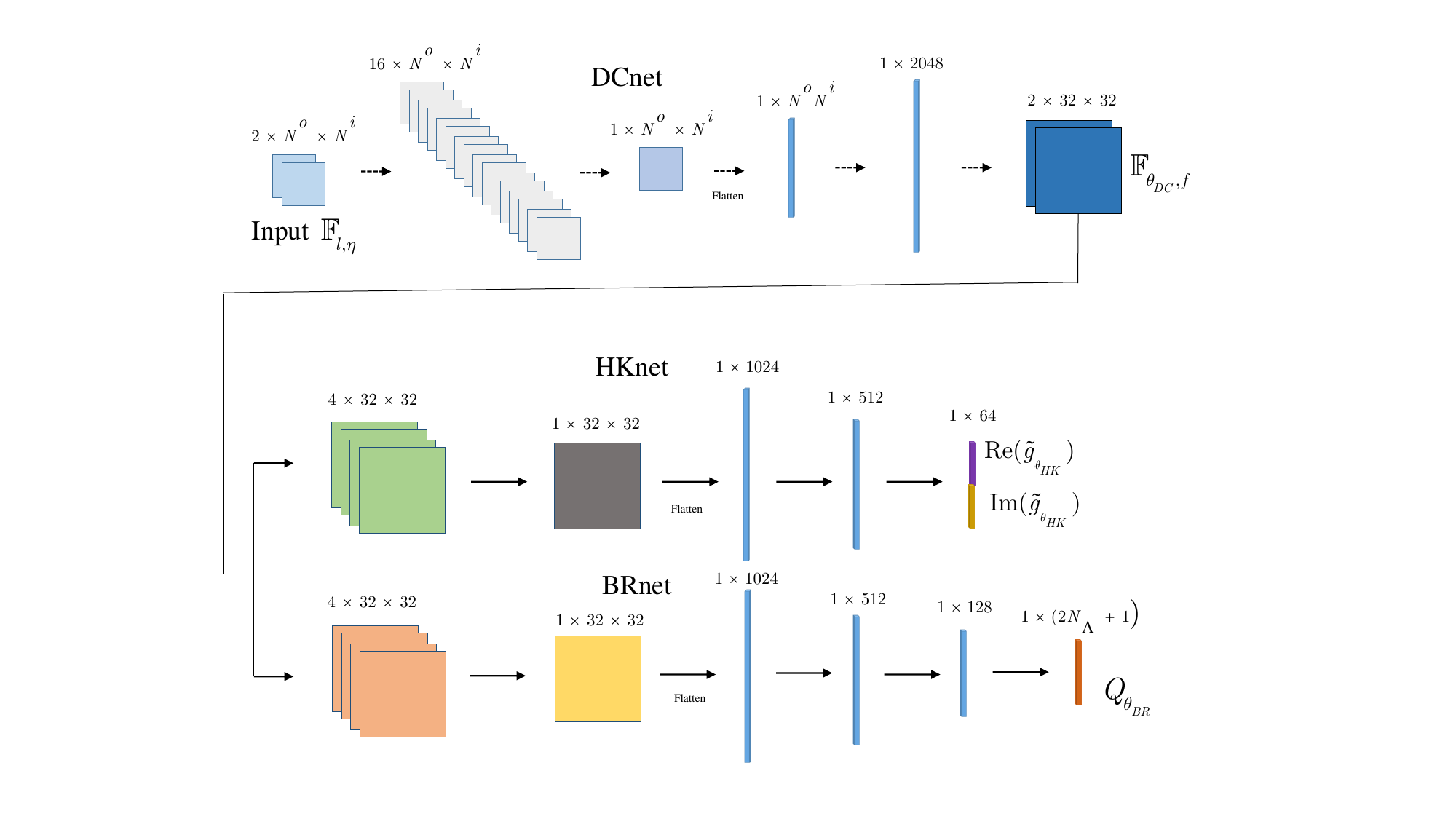}
}
    \caption{\label{fig:DDM_numerical_model}A schematic illustration of the network structure of DDM in all the numerical examples.}
\end{figure}
\subsection{Experiment Setup}
In all the following examples, the noise $\eta$ in \eqref{eq:noisy_input_form} is generated by the formula
\begin{equation}
\begin{aligned}
\eta=\sigma\triangle\mathbb{F}_{l},
\nonumber
\end{aligned}
\end{equation}
where $\sigma$ refers to the noise level, $\triangle$ is a random number drawn from the uniform distribution $\mathcal{U}(-1,1)$. $\mathbb{F}_{l}$ is generated by the boundary integral equation method with the single layer potential. Unless otherwise specified, we shall use the following parameters: $k=3$, $m=16$, $N_{t}=64$, $N_{\Lambda}=8$, $s=1$, $z=(0,0)$, \textcolor{black}{$\beta_{DC}=10$, $\alpha=10^{-8}$ and $\gamma=1$.}

\emph{Data generation and Network structure.} To generate training and test samples, similar to \eqref{eq:lambda_star}-\eqref{eq:Q_set}, the boundary set $\{\partial D^{(n_{s})}\}_{n_{s}=1}^{N_{s}}$ is built by $q^{(n_{s})}_{0}\sim \mathcal{U}(0.5,1.5)$, $a^{(n_{s})}_{n}, b^{(n_{s})}_{n}\sim \mathcal{N}(0,0.2^{2}), n=1,2,3,4$ and $a^{(n_{s})}_{n}, b^{(n_{s})}_{n}\sim \mathcal{U}(0,0.1), n=5,6,7,8$. \textcolor{black}{Note that the parameters $a^{(n_{s})}_{n}, b^{(n_{s})}_{n}$ for $n\leq 4$ and  $n > 4$ are assigned different distributions. This approach reduces parameter similarity and ensures that the pear-shaped and rounded square-shaped obstacles considered later are out-of-distribution cases.  Using the boundary integral equation method, we generate data pairs $\{\mathbb{F}_{l,\eta;n_{s}},\mathbb{F}_{f;n_{s}}\}_{n_{s}=1}^{N_{s}}$ with a  total sample size $N_{s}=5000$. Here, $\mathbb{F}_{l,\eta;n_{s}}$ denotes the $n_{s}$-th noisy  limited aperture data, while  $\mathbb{F}_{f;n_{s}}$ represents the corresponding  exact full aperture data. The dataset is split into $N_{train}=0.8N_{s}$ pairs for training  and $N_{test}=0.2N_{s}$ pairs for testing. Notably,  DDM does not require the exact boundary set $\{\partial D^{(n_{s})}\}_{n_{s}=1}^{N_{train}}$ during the training phase.}

The structure of DCnet is a feed-forward stack of two sequential combinations of the convolution, batch normalization and ReLU layers, followed by one fully connected layer. The number of filters in the two convolutional layers of DCnet are respectively 16 and 1. HKnet is a feed-forward stack consisting of two sequential combinations of the convolution, batch normalization and ReLU layers, followed by two fully connected layers. The number of filters in these two convolutional layers are 4 and 1, respectively. The two fully connected layers in HKnet have 512 and $4m$ neurons, respectively. Finally, the architecture of BRnet is a feed-forward stack of two sequential combinations of the convolution, batch normalization and ReLU layers, followed by three fully connected layers. The number of filters in these two convolutional layers are also 4 and 1, respectively. The three fully connected layers in BRnet have 512, 128 and $2N_{\Lambda}+1$ neurons, respectively. The initial weights are randomly drawn from $\mathcal{N}(0,0.02^{2})$ and initial biases are set to be zero. Moreover, in all the convolutional layers of DDM, one-stride, zero-padding and $3\times3$ kernel are employed. See Fig. \ref{fig:DDM_numerical_model} for the schematic illustration of the network structure of DDM in all the numerical examples.

\emph{Loss function for a mini-batch.} With the above constructed structure of DDM, we train DDM using the Adam optimizer for \textcolor{black}{1000} epochs with a learning rate of 0.0001. In addition, a mini-batch of $N_{b}$ samples are used in each iteration, where $N_{b}\ll N_{train}$. We set $N_{b}=64$. Thus, the loss function $\mathcal{L}_{DDM}$ for $N_{b}$ samples in every iteration is defined as
\begin{equation}
\begin{aligned}
\mathcal{L}_{DDM}=\mathcal{L}_{phy}+\beta_{DC}\mathcal{L}_{DC},
\nonumber
\end{aligned}
\end{equation}
where
\begin{equation}
\begin{aligned}
\mathcal{L}_{phy}:=\frac{1}{N_{b}}\sum\limits_{n_{b}=1}^{N_{b}}\widetilde{\mathcal{J}}_{phy}(\widetilde{g}_{\theta_{HK};n_{b}}, Q^{(n_{b})}_{\theta_{BR}};\alpha),\ \mathcal{L}_{DC}:=\frac{1}{N_{b}}\sum\limits_{n_{b}=1}^{N_{b}}\widetilde{\mathcal{J}}_{DC}(\mathbb{F}_{\theta_{DC},f;n_{b}}).
\nonumber
\end{aligned}
\end{equation}

\emph{Training and test relative errors.} \textcolor{black}{To investigate the convergence property of DDM, we define the discrete training relative error $\mathrm{Err}$ and test relative error $\mathrm{TErr}$ as follows}:
\begin{equation}
\begin{aligned}
\mathrm{Err}= \frac{\sqrt{\sum_{i=1}^{N_{b}}\|Q^{(i)}_{\theta_{BR}}-Q^{(i)}_{true}\|_{2}^{2}}}{\sqrt{\sum_{i=1}^{N_{b}}\|Q^{(i)}_{true}\|_{2}^{2}}},\ \ \textcolor{black}{\mathrm{TErr}= \frac{\sqrt{\sum_{i=1}^{N_{test}}\|Q^{(i)}_{\theta^{\ast}_{BR},test}-Q^{(i)}_{true,test}\|_{2}^{2}}}{\sqrt{\sum_{i=1}^{N_{test}}\|Q^{(i)}_{true,test}\|_{2}^{2}}},}
\nonumber
\end{aligned}
\end{equation}
\textcolor{black}{where $Q^{(i)}_{true}$ denotes the $i$-th exact parameter set in the mini-batch training data, and $Q^{(i)}_{\theta_{BR}}$ is the corresponding  recovered training parameter set. Similarly,  $Q^{(i)}_{true,test}$ represents the $i$-th exact parameter set in the test data, while $Q^{(i)}_{\theta^{\ast}_{BR},test}$ is the corresponding  recovered test parameter set.  In this setup, $\theta_{BR}$ is used in computing $\mathrm{Err}$ to describe the convergence of DDM during the training process, whereas  $\mathrm{TErr}$ is computed using the  optimal  learned parameter $\theta^{\ast}_{BR}$ to demonstrate the generalization performance on the test set after DDM has been trained.}

\subsection{Numerical examples}
To demonstrate that DDM can recover complex scatterers, we will perform the following three types of examples:

\begin{itemize}
\item Case 1: \textcolor{black}{an obstacle randomly selected from the test set.}
\item Case 2: a pear-shaped obstacle. The boundary is parameterized as:
$$ x(t)=(1.5+0.3\sin 3t)(\cos t, \sin t),\ 0\leq t\leq 2 \pi.$$
\item Case 3: a rounded square-shaped obstacle. The boundary is parameterized as:
$$x(t)=\frac{3}{4}(\cos^{3} t + \cos t, \sin^{3} t + \sin t),\ 0\leq t\leq 2 \pi.$$
\end{itemize}
\textcolor{black}{Here, Case 1 is used to evaluate in-distribution generalization. The last two scenarios, especially Case 3, which is non-starlike, are more complex and represent out-of-distribution cases. These three examples are used to test the performance of DDM and other methods in terms of inversion accuracy and computational efficiency.}

\emph{Other methods for comparison.} \textcolor{black}{We primarily compare our DDM framework to the classical decomposition method (CDM) \cite{DRIA,Colton_Monk1,Colton_Monk2,ochs1} in terms of computational cost and accuracy.  For a fair comparison, we employ an optimization functional from \cite{DRIA} similar to that used in DDM and ensure CDM shares the same discretization scheme.  Given noisy limited aperture data $\mathbb{F}_{l,\eta}$, CDM minimizes the following functional:
\begin{equation}
\begin{aligned}
\widetilde{\mathcal{J}}_{CDM}(\widetilde{g}, Q;\alpha):=&\frac{\pi}{m}\sum\limits_{j=n^{o}}^{N^{o}}\left|\frac{\pi}{m}\sum\limits_{i=n^{i}}^{N^{i}}\mathbb{F}_{l,\eta}^{(ji)}\widetilde{g}^{(i)}-\Phi^{\infty}(\hat{x}_{j},z)\right|^{2}\\
+&\alpha\frac{2\pi}{N_{t}}\sum\limits_{l=1}^{N_{t}}\left|\frac{\pi}{m}\sum\limits_{i=n^{i}}^{N^{i}}\mathrm{e}^{\mathrm{i}kx^{(l)}_{\Gamma}\cdot d_{i}}\widetilde{g}^{(i)}\right|^{2}\\
+&\gamma\frac{2\pi}{N_{t}}\sum\limits_{l=1}^{N_{t}}\left|\frac{\pi}{m}\sum\limits_{i=n^{i}}^{N^{i}}\mathrm{e}^{\mathrm{i}kx^{(l)}_{\Lambda}\cdot d_{i}}\widetilde{g}^{(i)}+\Phi(x^{(l)}_{\Lambda},z)\right|^{2},
\nonumber
\end{aligned}
\end{equation}
for $\widetilde{g}$ and $Q$. Here, $\widetilde{g}^{(i)}=g(d_{i})$ is the $i$-th element of $\widetilde{g}$, and $x^{(l)}_{\Gamma}$ and $x^{(l)}_{\Lambda}$,  determined by the set $Q$, share the same forms of $x^{(l)}_{\Gamma_{\theta_{BR}}}$ and $x^{(l)}_{\Lambda_{\theta_{BR}}}$ used in DDM. CDM  minimizes $\widetilde{\mathcal{J}}_{CDM}(\widetilde{g}, Q;\alpha)$ using the default BFGS algorithm with zero initial values in  Python's  Scipy library. }


\textcolor{black}{We also consider a qualitative approach known as the direct sampling method (DSM), proposed in \cite{Liu_xd1}, which introduces an indicator functional as follows:
\begin{equation}
\begin{aligned}
I(h):=\left|\varphi(h ;-d) \mathbb{F}^{T}_{l,\eta} \varphi^{T}(h ; \hat{x})\right|,
\nonumber
\end{aligned}
\end{equation}
where $h$ represents the probing point,} $\varphi(h ;\hat{x}):=(\mathrm{e}^{\mathrm{i}kh\cdot \hat{x}_{n^{o}}},\mathrm{e}^{\mathrm{i}kh\cdot \hat{x}_{n^{o}+1}},\cdots,\mathrm{e}^{\mathrm{i}kh\cdot \hat{x}_{N^{o}}})$ and $\varphi(h ;-d):=(\mathrm{e}^{-\mathrm{i}kh\cdot d_{n^{i}}},\mathrm{e}^{-\mathrm{i}kh\cdot d_{n^{i}+1}},\cdots,\mathrm{e}^{-\mathrm{i}kh\cdot d_{N^{i}}})$. We set $h\in V$, where $V$ is a grid of $100 \times 100$ equally spaced sampling points on $[-4,4]\times[-4,4]$. \textcolor{black}{This indicator $I(h)$ achieves its maximum value near the boundary of the scatterer, allowing the values of $I(h)$ across the domain $V$ to qualitatively visualize the scatterer. Although this approach is inherently qualitative and typically provides rough reconstructions for the limited aperture case, it is  computationally efficient.  In the numerical experiments, we employ DSM primarily to compare computational costs.}

\subsubsection{Example 1: Testing the computational cost and the accuracy} 

\begin{figure}[h]
    \begin{center}
        \begin{overpic}[width=0.32\textwidth]{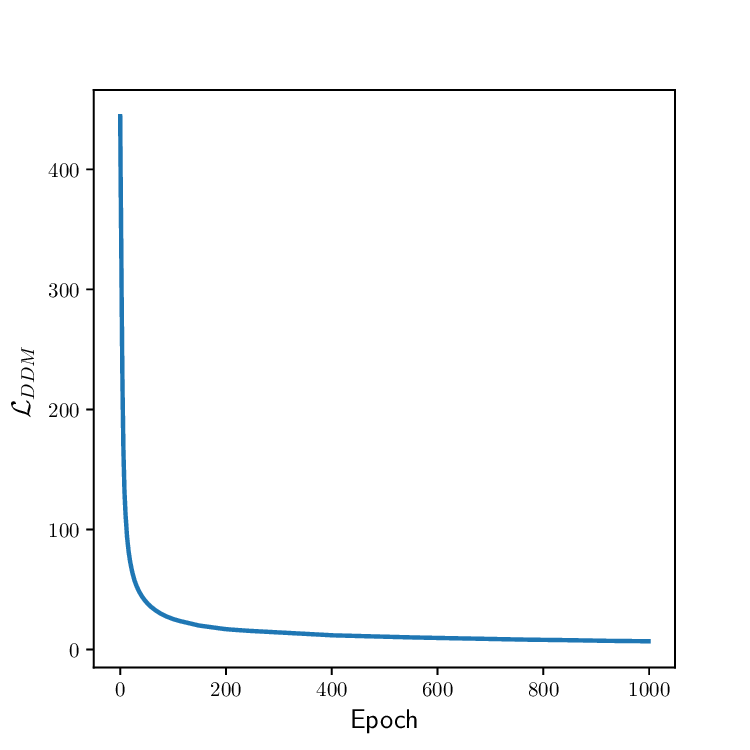}
           \put (40,91) {\scriptsize {$\mathcal{L}_{DDM}$}}
        \end{overpic}
        \hspace{-0.4cm}
        \begin{overpic}[width=0.32\textwidth]{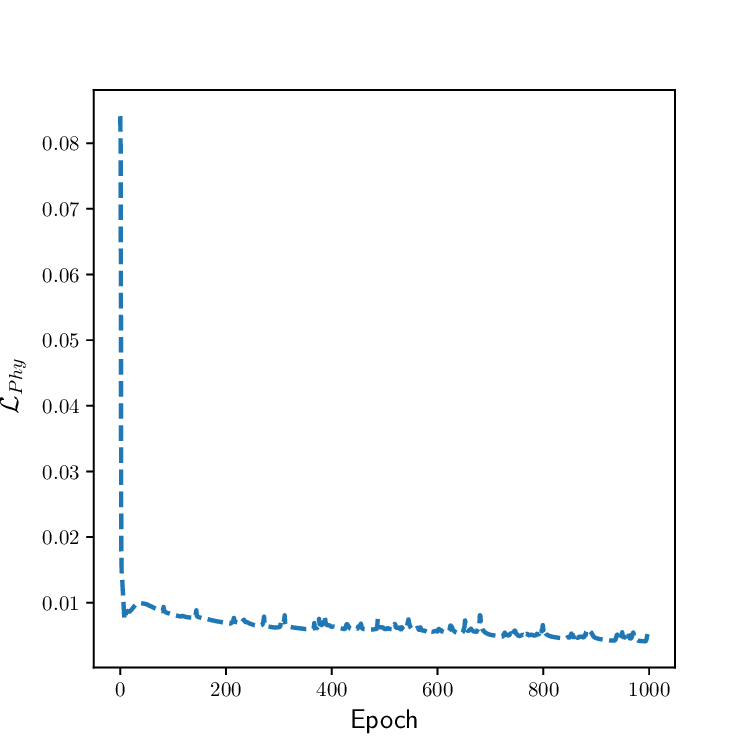}
           \put (44,91) {\scriptsize {$\mathcal{L}_{phy}$}}
        \end{overpic}
        \hspace{-0.4cm}
        \begin{overpic}[width=0.32\textwidth]{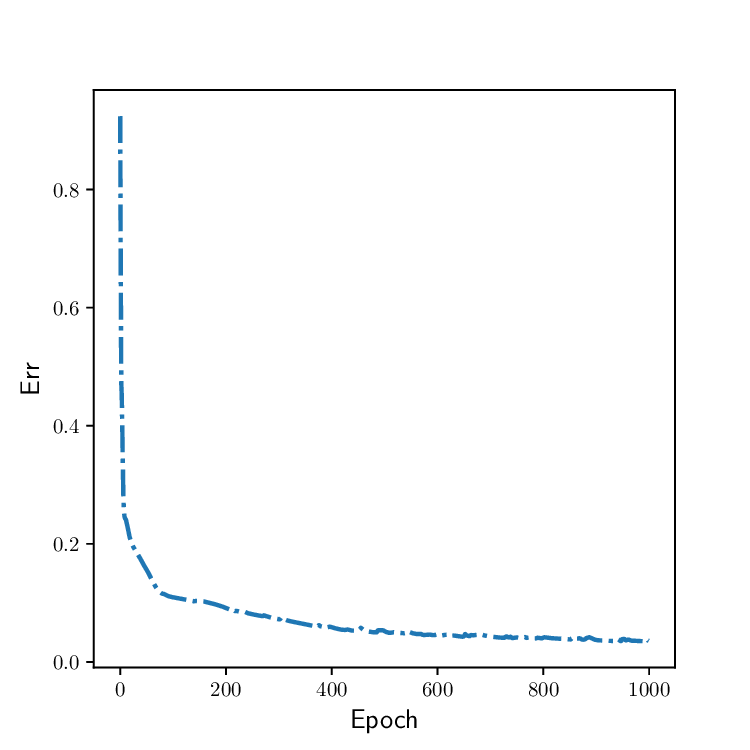}
           \put (47,91) {\scriptsize {Err}}
        \end{overpic}
    \end{center}
    \caption{\textcolor{black}{ Example 1:  The evolution of the DDM loss function $\mathcal{L}_{DDM}$, the physics-based loss function $\mathcal{L}_{phy}$, and the training relative error $\mathrm{Err}$, throughout the training process.}}
    \label{fig:ex1_loss_convergence}
\end{figure}

\textcolor{black}{In this example, we aim to evaluate the computational cost and the accuracy of the proposed DDM. For this purpose, we use the previously mentioned CDM and DSM for comparison. In addition, we set the incident aperture $\psi=[0,2\pi]$, the partial observation aperture $\phi=[0,\pi/2]$ and the noise level $\sigma=0$.}

 \begin{figure}
\begin{center}
  \begin{overpic}[width=0.32\textwidth,trim=55 20 55 15, clip=true,tics=10]{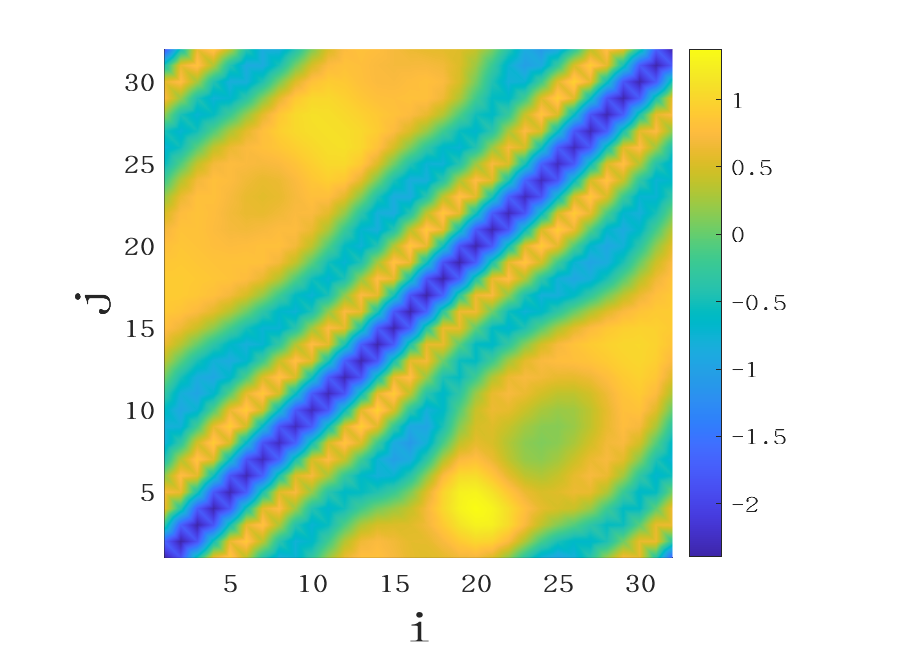}
  \put (40,86) {\footnotesize{Case 1}}
  \end{overpic}
    \begin{overpic}[width=0.32\textwidth,trim= 55 20 55 15, clip=true,tics=10]{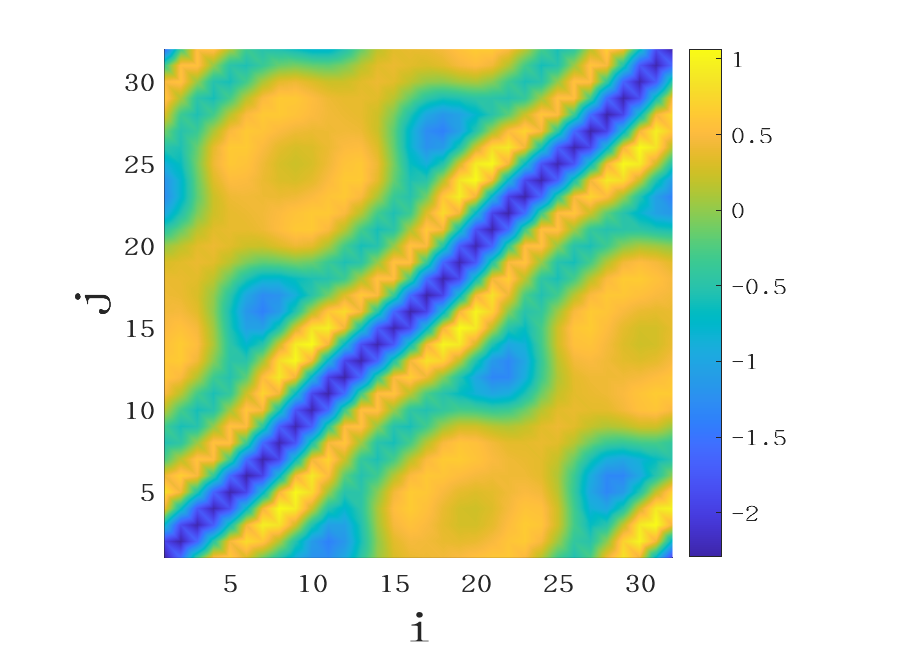}
      \put (40,86) {\footnotesize{Case 2}}
  \end{overpic}
    \begin{overpic}[width=0.32\textwidth,trim= 55 20 55 15, clip=true,tics=10]{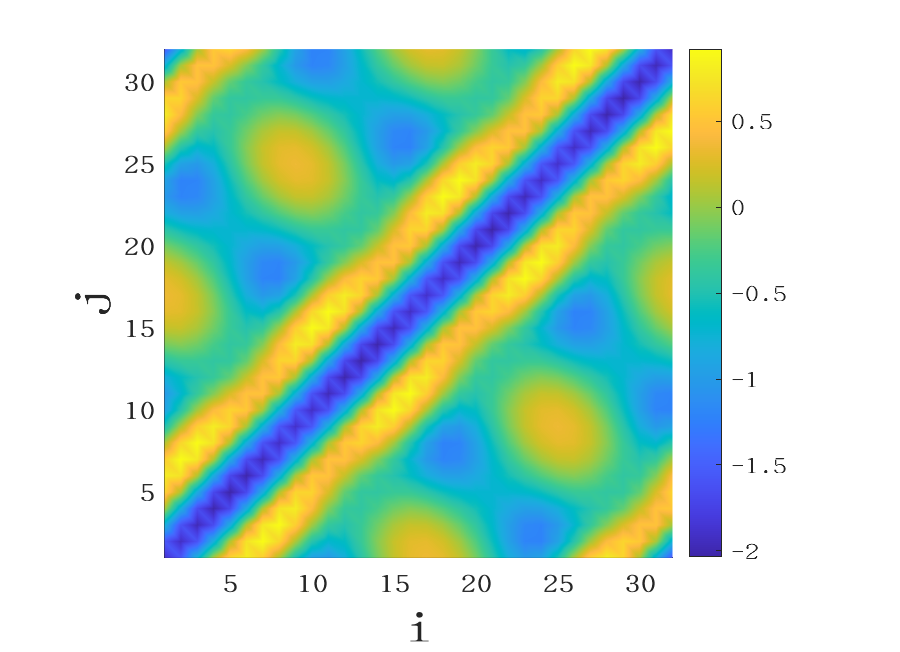}
      \put (40,86) {\footnotesize{Case 3}}
  \end{overpic}
  \begin{overpic}[width=0.32\textwidth,trim=55 20 55 15, clip=true,tics=10]{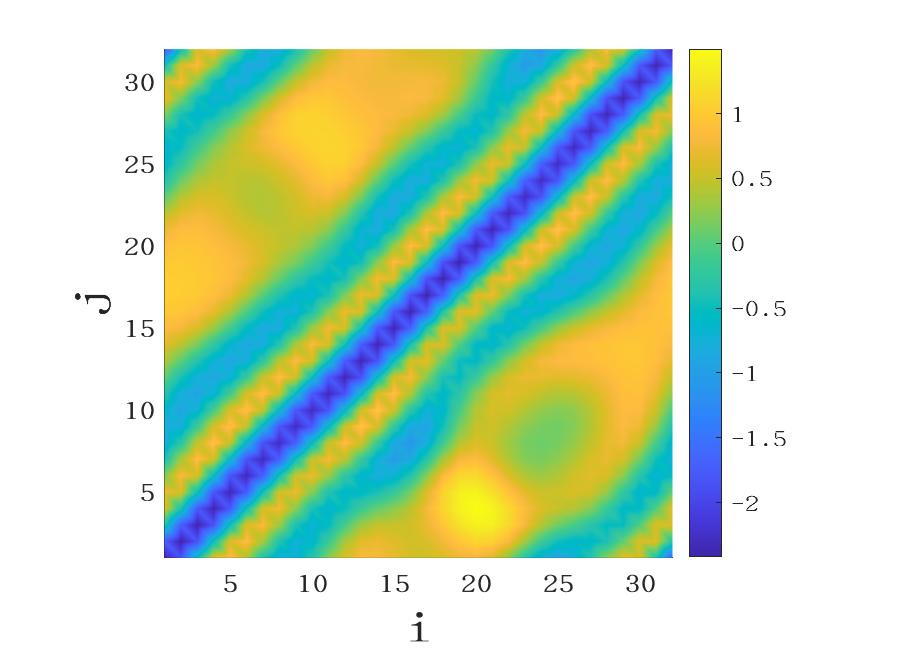}
  \end{overpic}
    \begin{overpic}[width=0.32\textwidth,trim= 55 20 55 15, clip=true,tics=10]{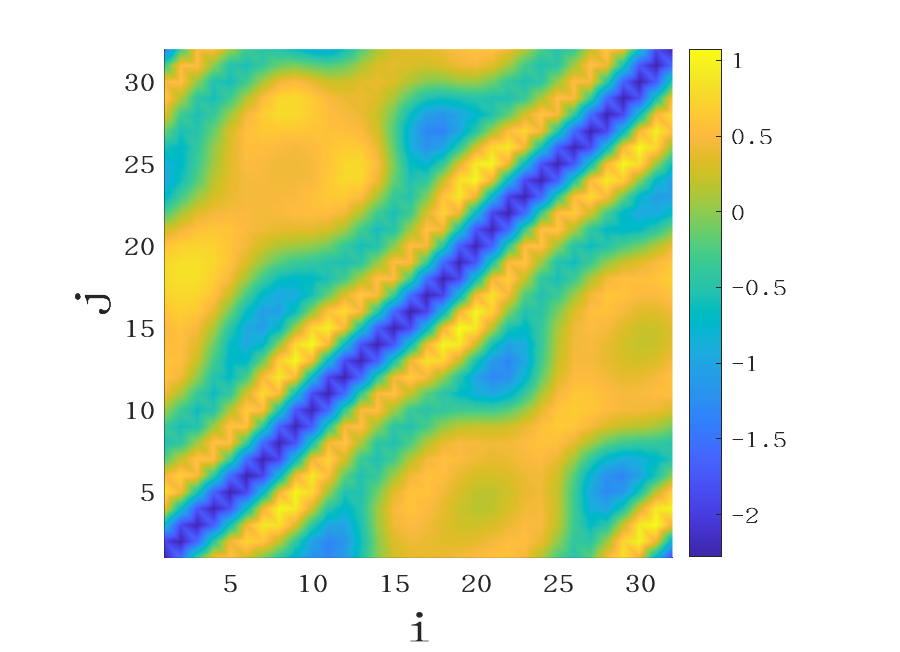}
  \end{overpic}
    \begin{overpic}[width=0.32\textwidth,trim= 55 20 55 15, clip=true,tics=10]{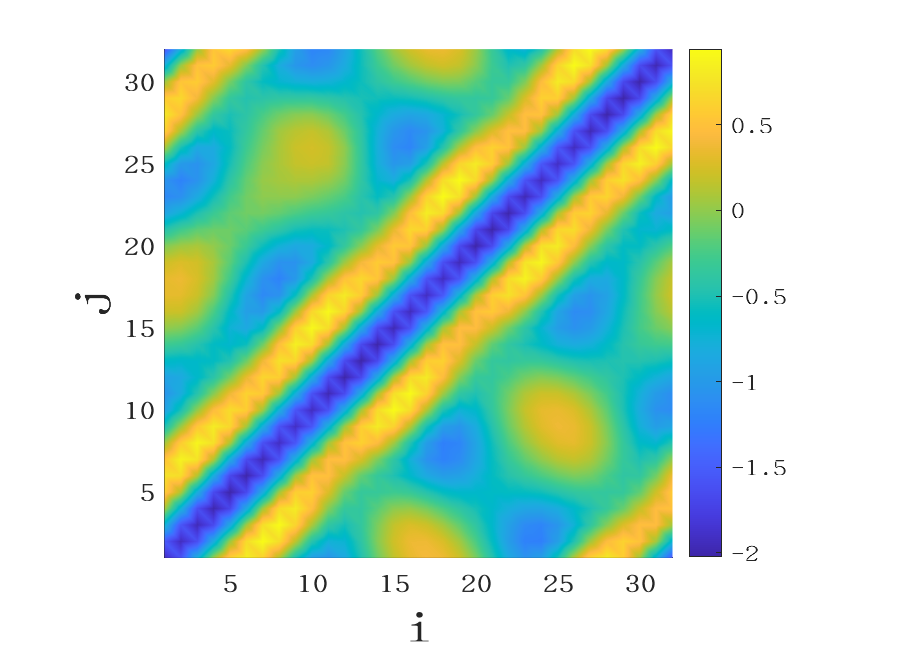}
  \end{overpic}
\end{center}
\caption{Example 1:  \textcolor{black}{The real parts of the exact (top) and recovered far-field data (bottom). The x-axis label $i$ represents the $i$-th incident direction, while the y-axis label $j$ indicates the $j$-th observation direction.}}    \label{fig:ex1_re}
  \end{figure}

\begin{figure}
\begin{center}
  \begin{overpic}[width=0.32\textwidth,trim=55 20 55 15, clip=true,tics=10]{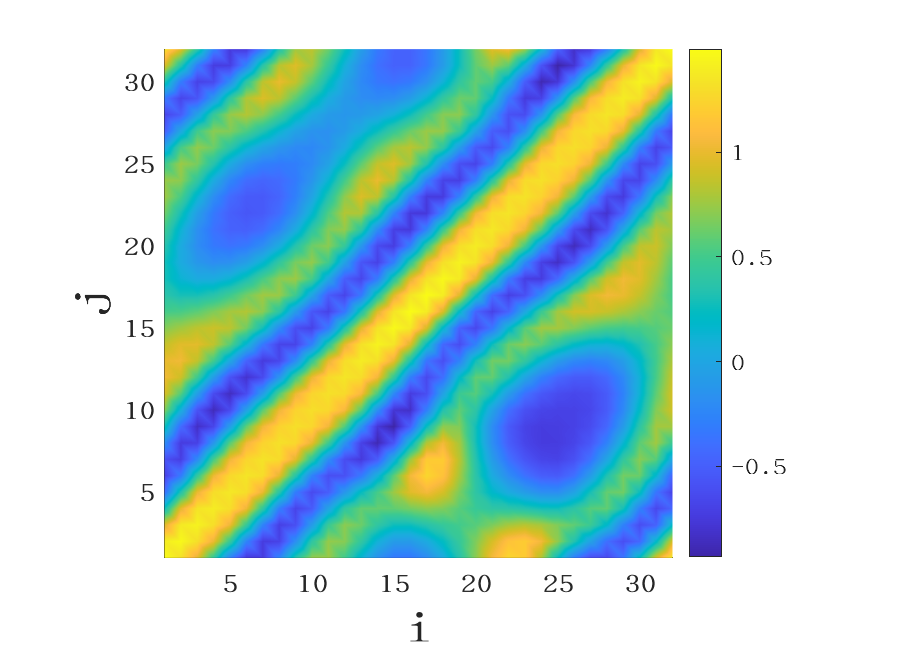}
  \put (40,86) {\footnotesize{Case 1}}
  \end{overpic}
    \begin{overpic}[width=0.32\textwidth,trim= 55 20 55 15, clip=true,tics=10]{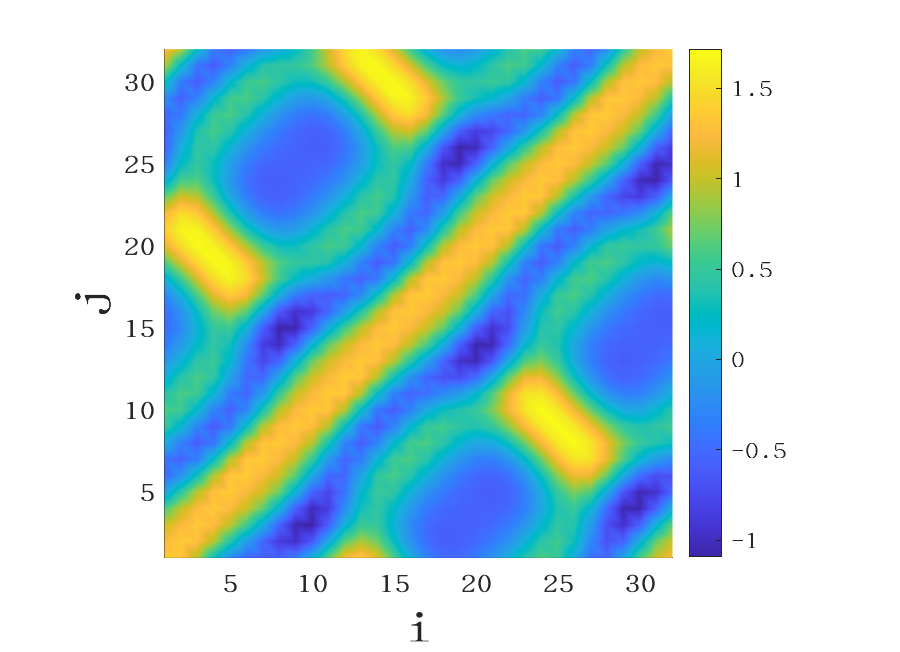}
      \put (40,86) {\footnotesize{Case 2}}
  \end{overpic}
    \begin{overpic}[width=0.32\textwidth,trim= 55 20 55 15, clip=true,tics=10]{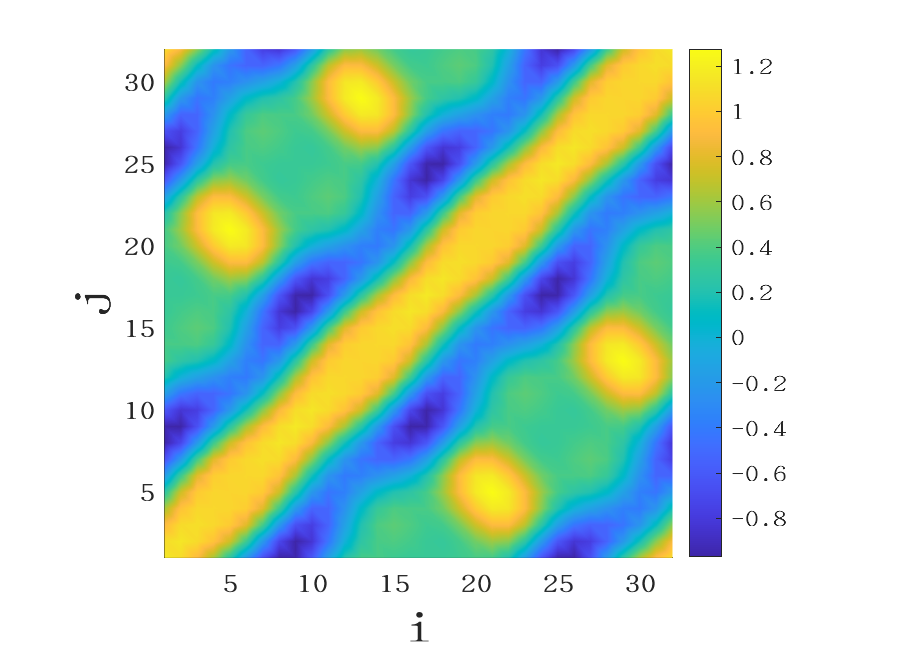}
      \put (40,86) {\footnotesize{Case 3}}
  \end{overpic}
  \begin{overpic}[width=0.32\textwidth,trim=55 20 55 15, clip=true,tics=10]{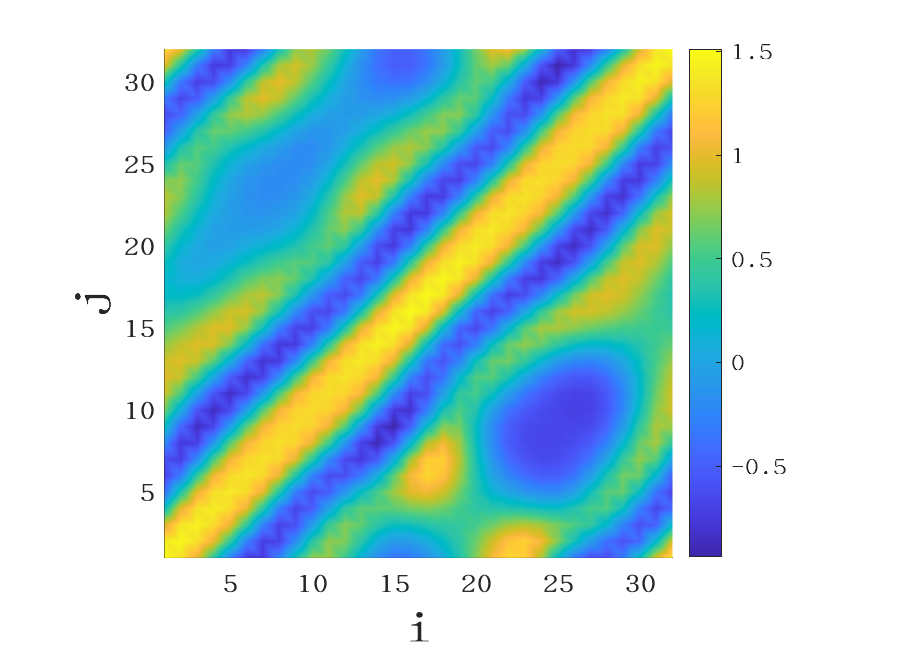}
  \end{overpic}
    \begin{overpic}[width=0.32\textwidth,trim= 55 20 55 15, clip=true,tics=10]{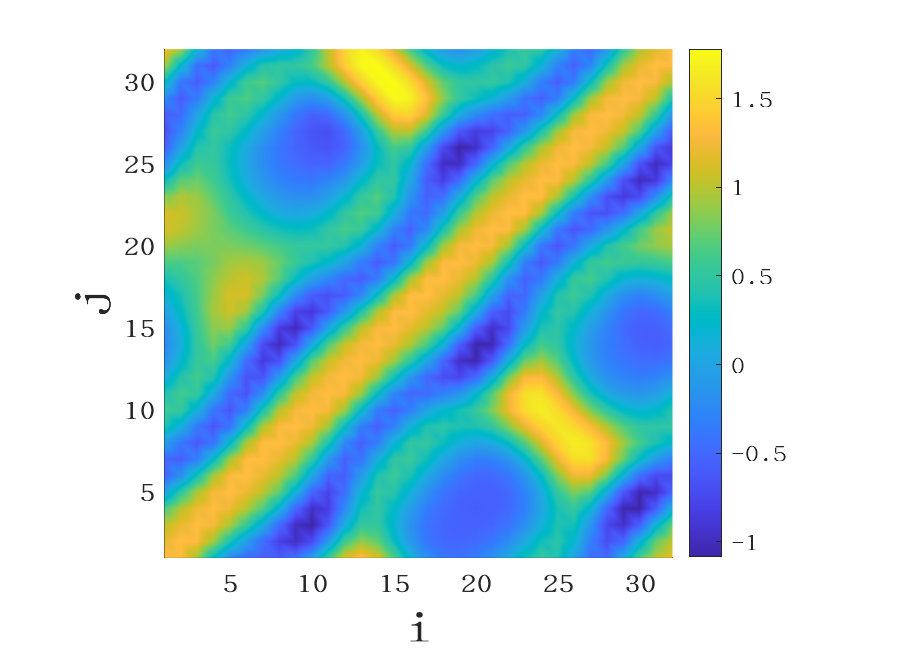}
  \end{overpic}
    \begin{overpic}[width=0.32\textwidth,trim= 55 20 55 15, clip=true,tics=10]{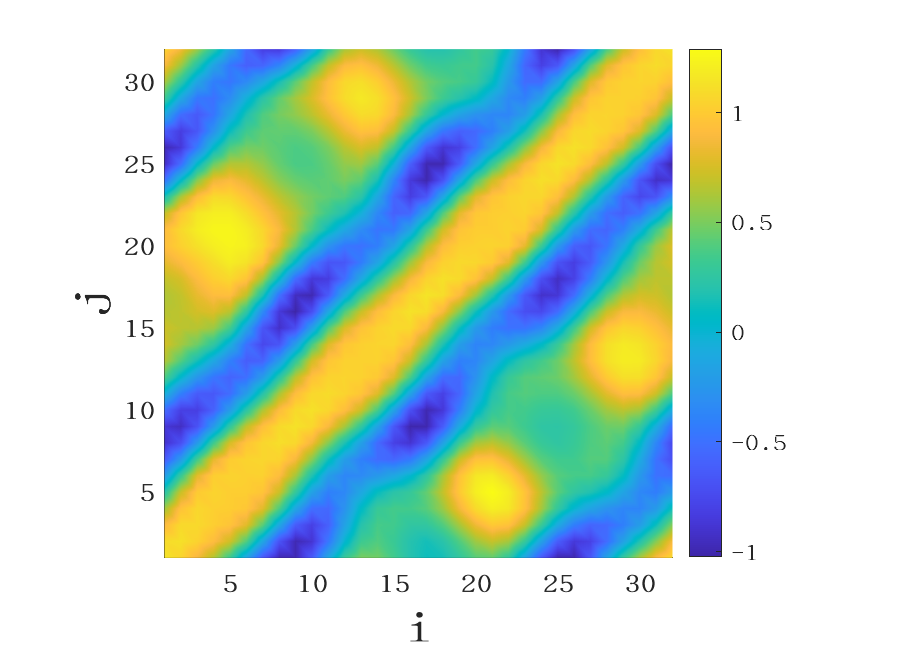}
  \end{overpic}
\end{center}
\caption{Example 1:  \textcolor{black}{The imaginary parts of the exact (top) and recovered far-field data (bottom). The x-axis label $i$ represents the $i$-th incident direction, while the y-axis label $j$ indicates the $j$-th observation direction.}}    \label{fig:ex1_im}
  \end{figure}
 
\begin{figure}[ht]
        \begin{center}
        \begin{overpic}[width=0.32\textwidth,trim=40 0 26 15, clip=true,tics=10]{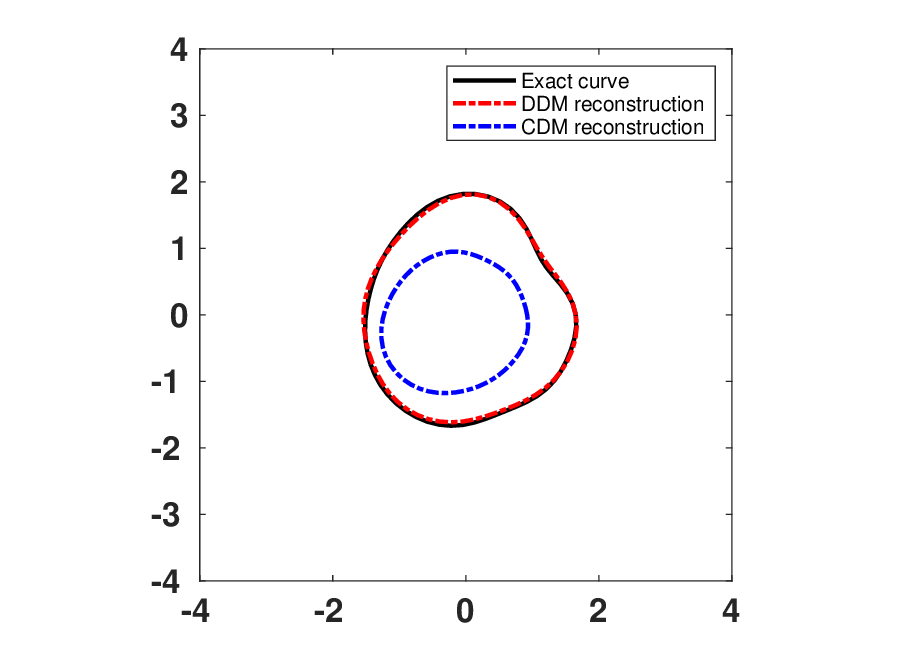}
           \put (40,83) {\scriptsize Case 1}
        \end{overpic}
        \hspace{-0.8cm}        
        \begin{overpic}[width=0.32\textwidth,trim=40 0 26 15, clip=true,tics=10]{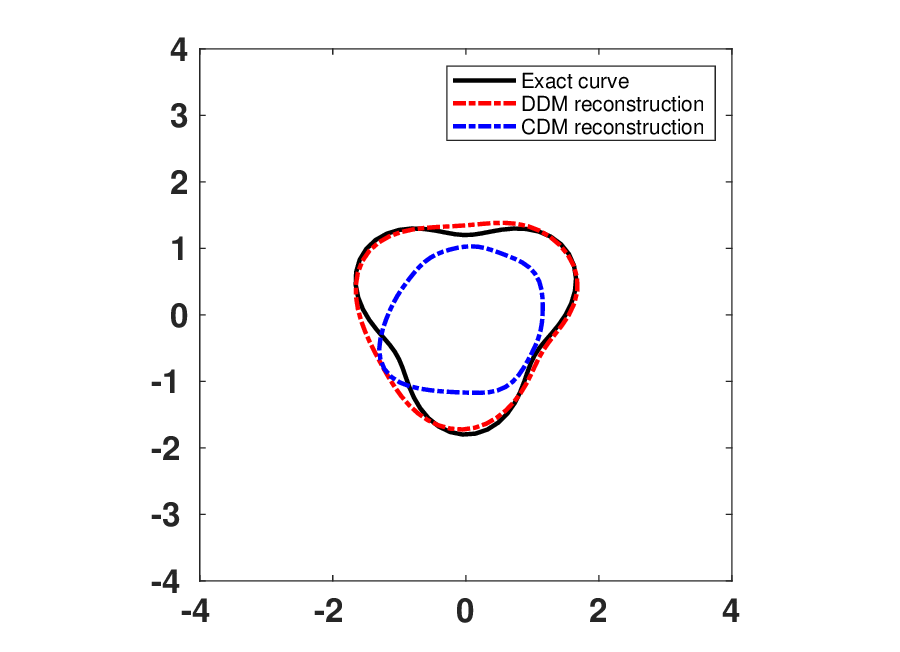}
           \put (40,83) {\scriptsize Case 2}
        \end{overpic}
        \hspace{-0.8cm}        
        \begin{overpic}[width=0.32\textwidth,trim=40 0 26 15, clip=true,tics=10]{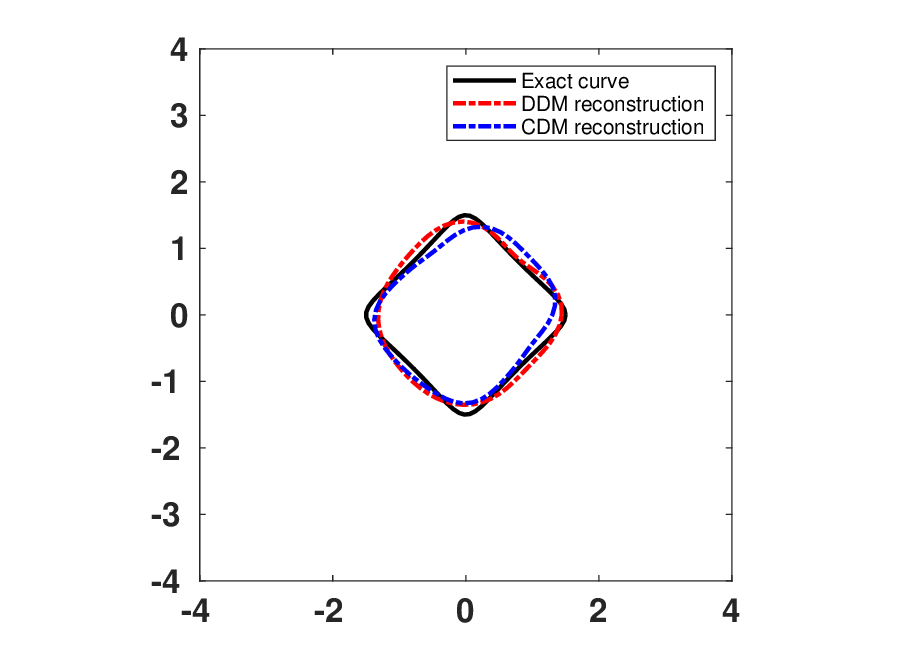}
           \put (40,83) {\scriptsize Case 3}
        \end{overpic}
    \end{center}
    \caption{\textcolor{black}{Example 1: Reconstructions produced by DDM and CDM. }}
    \label{fig:ex1_ddm_cdm}
\end{figure}

\begin{figure}[h]
        \begin{center}
        \begin{overpic}[width=0.32\textwidth,trim=50 0 45 15, clip=true,tics=10]{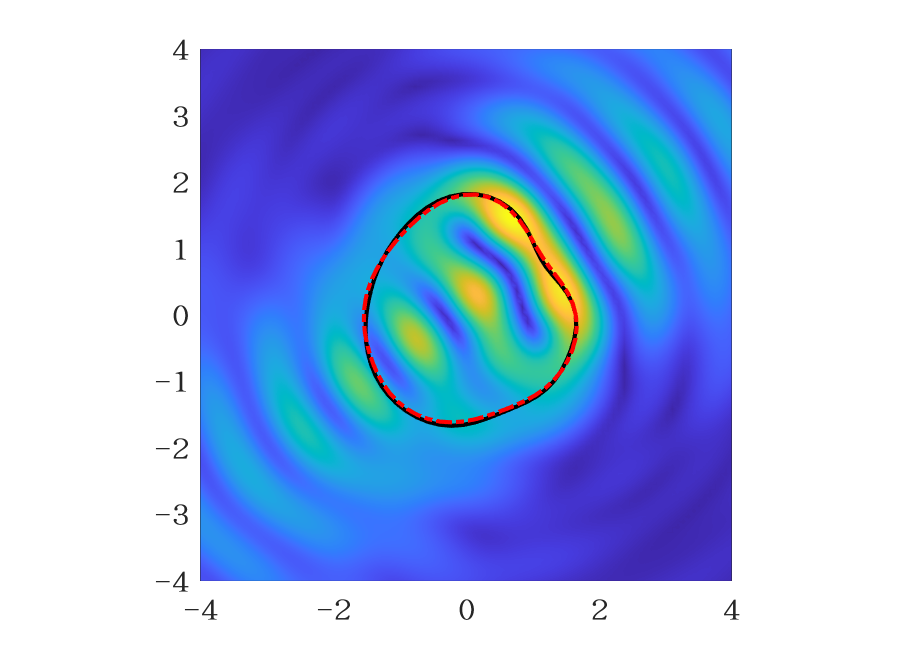}
           \put (40,88) {\scriptsize Case 1}
        \end{overpic}
        \hspace{-0.7cm}        
        \begin{overpic}[width=0.32\textwidth,trim=50 0 45 15, clip=true,tics=10]{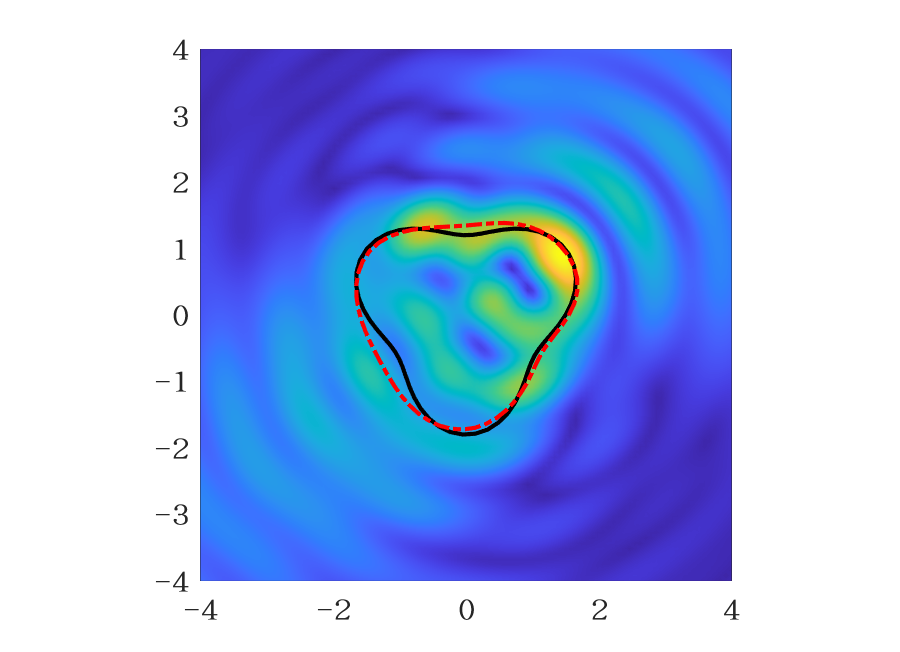}
           \put (40,88) {\scriptsize Case 2}
        \end{overpic}
        \hspace{-0.4cm}        
        \begin{overpic}[width=0.32\textwidth,trim=50 0 45 15, clip=true,tics=10]{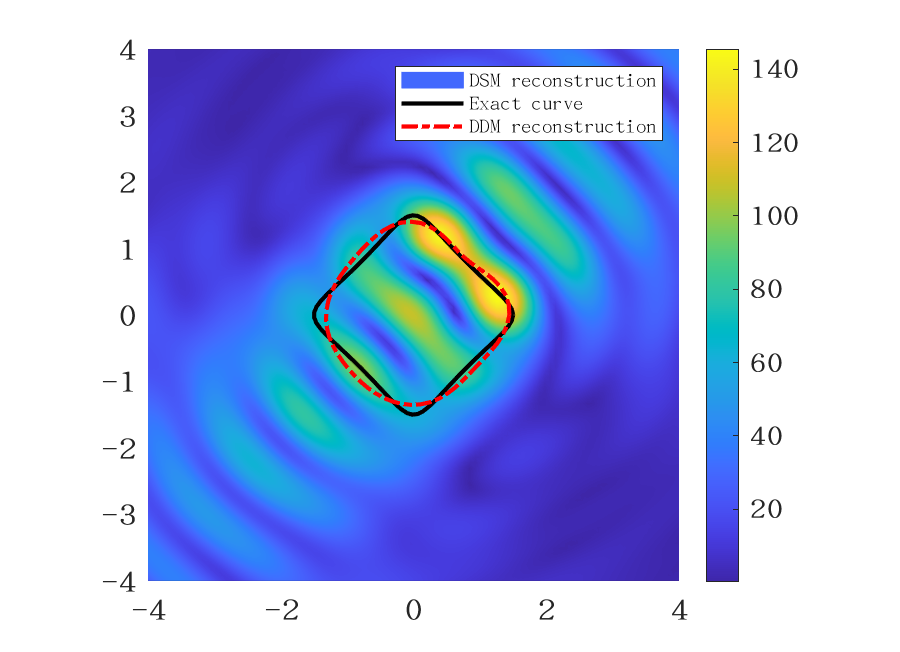}
           \put (40,88) {\scriptsize Case 3}
        \end{overpic}
    \end{center}
    \caption{\textcolor{black}{Example 1: Reconstructions made by DDM and DSM. }}
    \label{fig:ex1_ddm_dsm}
\end{figure}

\begin{figure}[ht]
        \begin{center}
        \begin{overpic}[width=0.45\textwidth,trim=40 0 26 15, clip=true,tics=10]{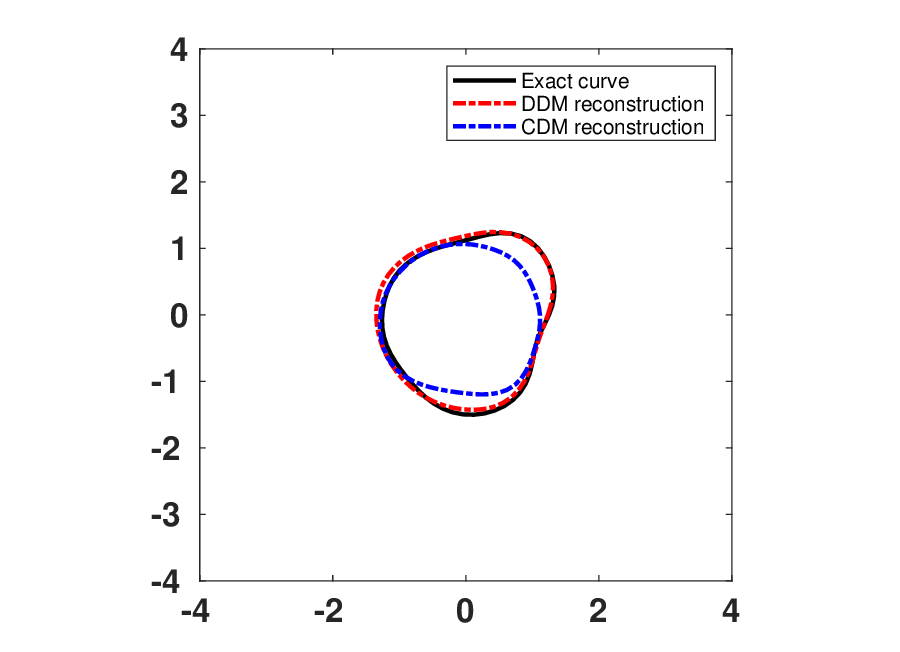}
        \put (30,81) {\footnotesize First random sample}
        \end{overpic}
        \hspace{-1cm} 
        \begin{overpic}[width=0.45\textwidth,trim=40 0 26 15, clip=true,tics=10]{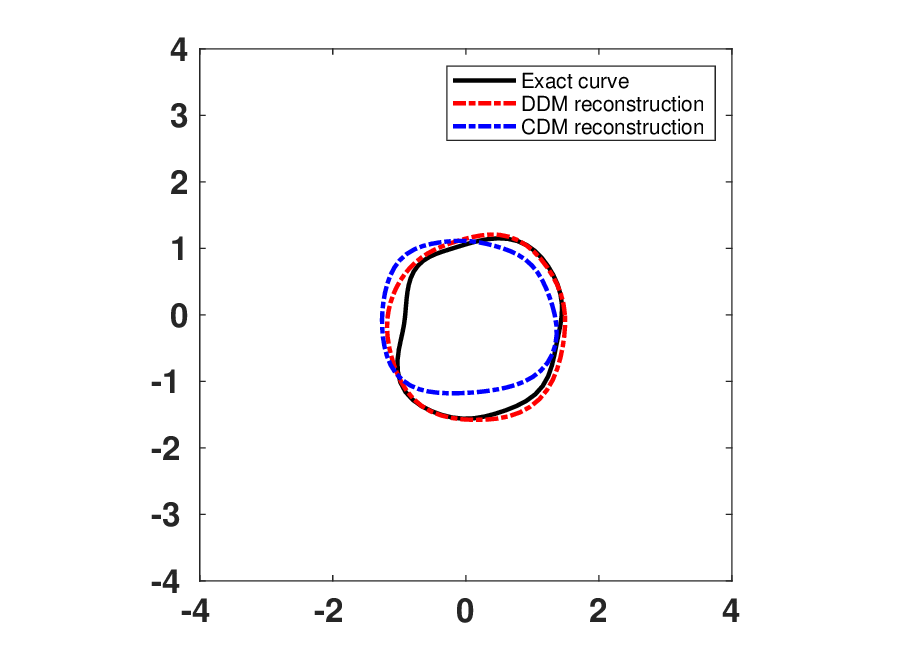}
        \put (30,81) {\footnotesize  Second random sample}
        \end{overpic}
    \end{center}
    \caption{\textcolor{black}{Example 1: Reconstructions of two randomly selected samples from the test set of Case 1.}}
    \label{fig:ex1_ddm_random}
\end{figure}

\begin{table}[htbp]
\caption{\textcolor{black}{Example 1: Numerical results for the three methods across the three cases. Here, the computational time for DDM represents the total time, including both data acquisition and inverse problem-solving.}}
\centering
\begin{tabular}{ccccccc}
\toprule
&\multicolumn{3}{c}{$\bar{e}$} &\multicolumn{3}{c}{time [s]}\\
\cmidrule(lr){2-4} \cmidrule(lr){5-7}
 & Case 1 & Case 2 & Case 3 & Case 1 & Case 2 & Case 3\\
\midrule
DDM  & 0.0183 & 0.0721 & 0.0677 & 0.94 & 0.86 & 0.84\\
CDM  & 0.3606 & 0.2843 & 0.0956 & 40.60 & 61.16 & 51.44\\
DSM  & - & - & - & 0.86 & 0.62 & 0.63\\
\bottomrule
\end{tabular}
\label{ex1_accuracy_cost}
\end{table}

\textcolor{black}{Before presenting the numerical inversion results, we first validate our theoretical findings, specifically Theorem \ref{thm:convergence_result_tend_exact_boundary}.  Fig.\ref{fig:ex1_loss_convergence} shows the evolution of the loss functions $\mathcal{L}_{DDM}$ and $\mathcal{L}_{phy}$, as well as the training relative error $\mathrm{Err}$, throughout the training process. As shown in the right display of Fig.\ref{fig:ex1_loss_convergence}, the predicted boundary $\Lambda_{\theta_{BR}}$ gradually converges to the exact boundary $\partial D$,  fully confirming Theorem \ref{thm:convergence_result_tend_exact_boundary}. For the following comparison of boundary recovery of different methods, we first pesent the far-field data reconstructed by DDM.  The real parts of the reconstructed far-field data for the three obstacles are shown in Fig.\ref{fig:ex1_re}, while the imaginary parts are shown in Fig.\ref{fig:ex1_im}. As seen in Figs.\ref{fig:ex1_re} and \ref{fig:ex1_im}, a substantial portion of the far-field data is accurately reconstructed. This observation suggests that, although we do not explicitly exploit the reciprocity relation $u^{\infty}(\hat{x},d)=u^{\infty}(-d, -\hat{x})$, the DDM framework is still capable of capturing this property.} 

\textcolor{black}{Figs. \ref{fig:ex1_ddm_cdm} and \ref{fig:ex1_ddm_dsm} present the boundary reconstructions using DDM, CDM, and DSM. It is evident that DDM outperforms both CDM and DSM, including for the more challenging out-of-distribution cases, Case 2 and Case 3. DSM can only recover the illuminated portion of the three obstacles under consideration, while CDM’s solution closely approximates a unit circle due to its tendency to converge to a local minimum. These results are expected, as both CDM and DSM rely directly on limited-aperture data, whereas DDM leverages DCnet, a deep-learning-based data retrieval approach, to enhance reconstruction.  Since only one test sample (Case 1) has been presented above, we further evaluate the performance of DDM on additional test data by randomly selecting two more samples. The reconstruction results produced by DDM and CDM are shown in Fig.\ref{fig:ex1_ddm_random}, illustrating that DDM performs well on these samples as well. The test relative error, $\mathrm{TErr}$, calculated over all test samples for DDM is 0.2521, indicating robust generalization across the test set. }

\textcolor{black}{To further quantitatively compare the reconstruction accuracy of DDM and CDM, we define the following $L^{2}$ relative error:
\begin{equation}
\begin{aligned}
\bar{e}=\frac{\sqrt{\sum_{l=1}^{N_{t}}|r^{\star}(\zeta_{l})-r^{\ast}(\zeta_{l})|^{2}}}{\sqrt{\sum_{l=1}^{N_{t}}|r^{\star}(\zeta_{l})|^{2}}},
\nonumber
\end{aligned}
\end{equation}
where ${r^{\star}}$ denotes the exact radius, ${r^{\ast}}$ denotes the recovered radius by DDM or CDM, and $\zeta_{l}=\zeta(t_{l})$  depends on whether the considered obstacle is starlike. For Case 2, $\zeta_{l}=t_{l}$, while for  Case 3, $\tan\zeta_{l}=\frac{x_{2}(t_{l})}{x_{1}(t_{l})}$, where $x_{1}(t_{l})$ and $x_{2}(t_{l})$ are the coordinates of the boundary point $x(t_{l})$  at $t_{l}$.  Tab.\ref{ex1_accuracy_cost} provides the corresponding numerical results, along with the computational costs for DDM, CDM, and DSM. For the three obstacles considered, DDM achieves total computational times of approximately 0.94, 0.86, and 0.84 seconds, respectively. The data also indicate that, while DDM’s computational cost is similar to DSM’s, it is more than 50 times faster than CDM across all cases. This highlights that DDM not only provides more accurate reconstructions but does so with substantially lower computational time. Furthermore, the  relative error for DDM is significantly smaller than that of CDM, particularly in Case 1.}

\subsubsection{Example 2: Different noise levels} 

\begin{figure}[h]
    \begin{center}
        \begin{overpic}[width=0.32\textwidth]{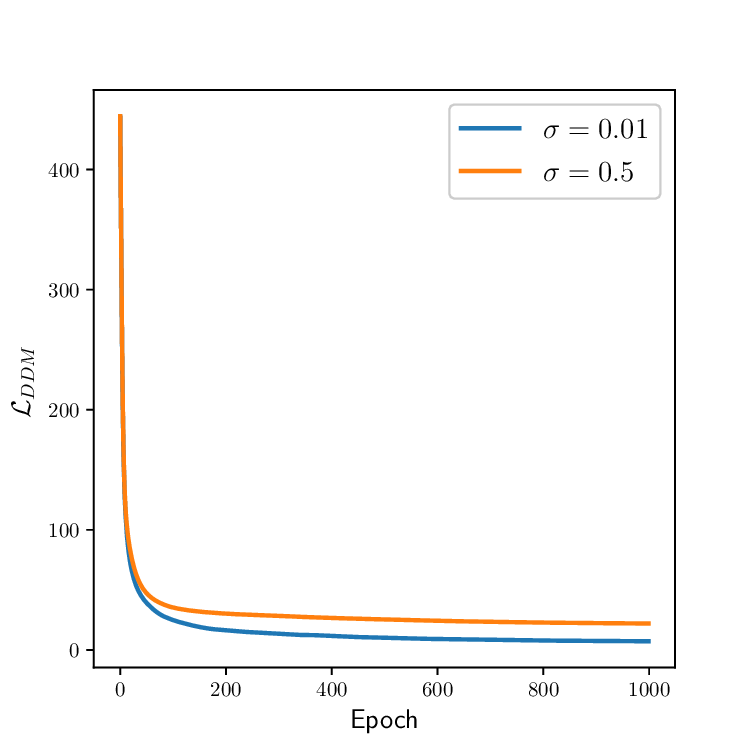}
           \put (40,91) {\scriptsize {$\mathcal{L}_{DDM}$}}
        \end{overpic}
        \hspace{-0.4cm}
        \begin{overpic}[width=0.32\textwidth]{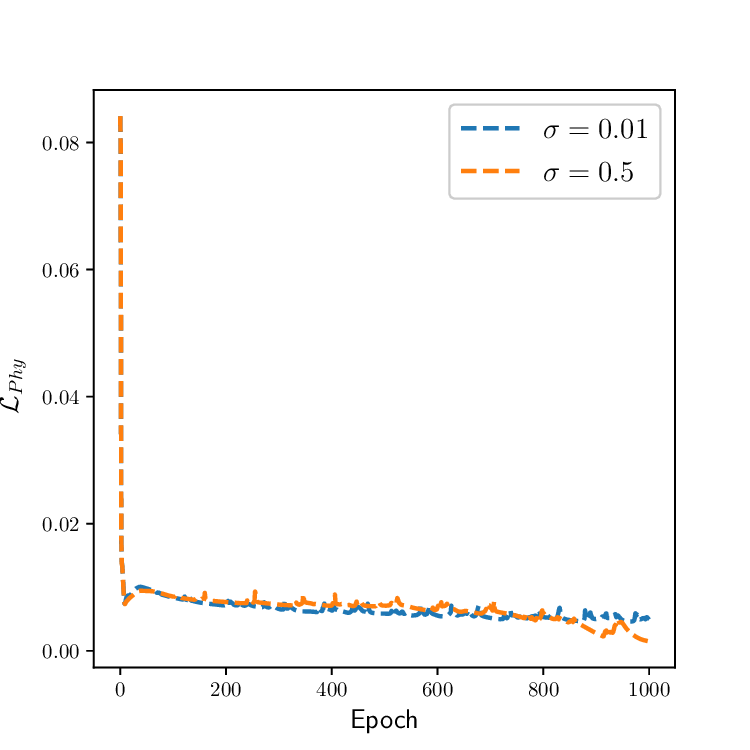}
           \put (44,91) {\scriptsize {$\mathcal{L}_{phy}$}}
        \end{overpic}
        \hspace{-0.4cm}
        \begin{overpic}[width=0.32\textwidth]{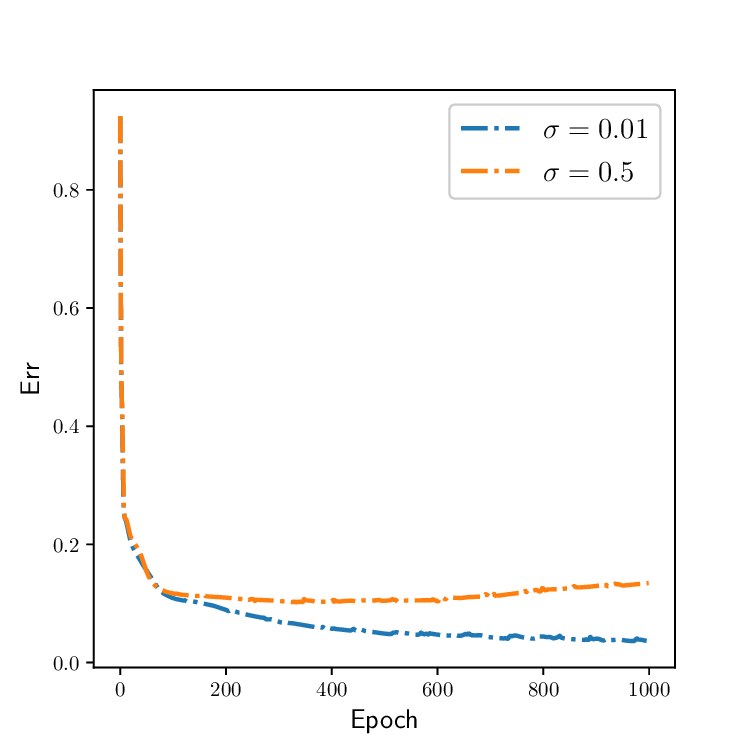}
           \put (47,91) {\scriptsize {Err}}
        \end{overpic}
    \end{center}
    \caption{\textcolor{black}{ Example 2: The evolution of the DDM loss function $\mathcal{L}_{DDM}$, the physics-based loss function $\mathcal{L}_{phy}$, and the training relative error $\mathrm{Err}$, throughout the training process.}}
     \label{fig:ex2_loss_convergence}
\end{figure}

\begin{figure}[ht]
    \begin{center}
        \begin{overpic}[width=0.32\textwidth,trim=55 20 55 15, clip=true,tics=10]{test_re_exact.eps}
           \put (19,86) {\scriptsize {Exact real parts}}
        \end{overpic}
        \hspace{-0.1cm}   
        \begin{overpic}[width=0.32\textwidth,trim=55 20 55 15, clip=true,tics=10]{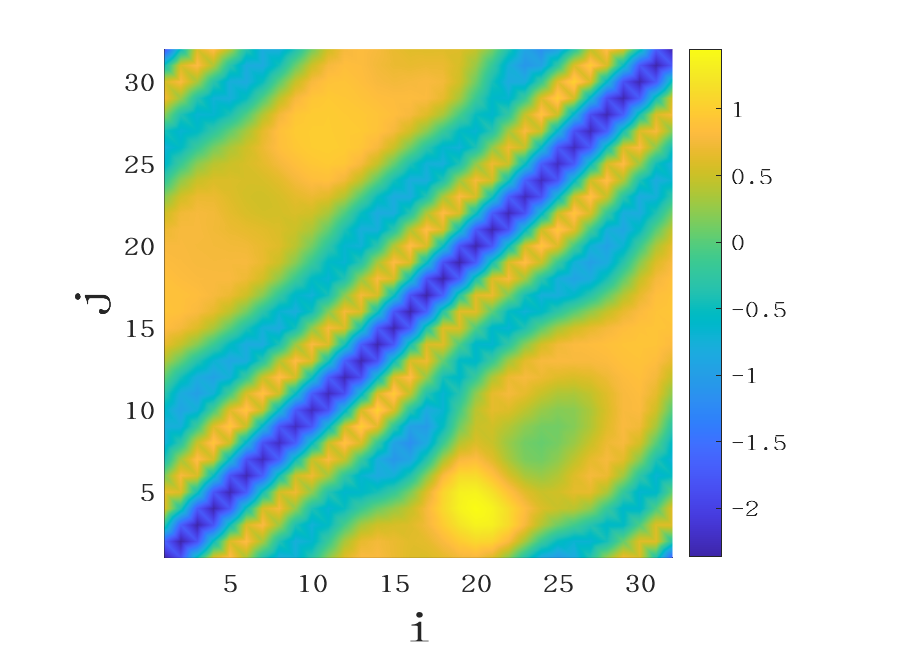}
           \put (27,86) {\scriptsize {$\sigma=0.01$}}
        \end{overpic}
        \hspace{-0.1cm}   
        \begin{overpic}[width=0.32\textwidth,trim=55 20 55 15, clip=true,tics=10]{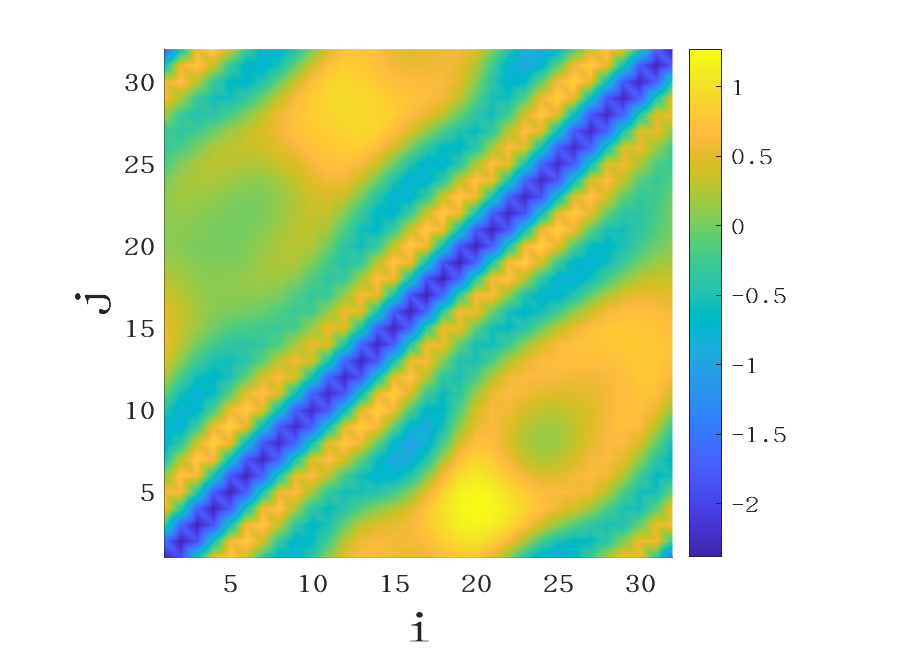}
           \put (30,86) {\scriptsize {$\sigma=0.5$}}
        \end{overpic}
    \end{center}
        \begin{center}
        \begin{overpic}[width=0.32\textwidth,trim=55 20 55 15, clip=true,tics=10]{pear_re_exact.eps}
        \end{overpic}
        \hspace{-0.1cm}   
        \begin{overpic}[width=0.32\textwidth,trim=55 20 55 15, clip=true,tics=10]{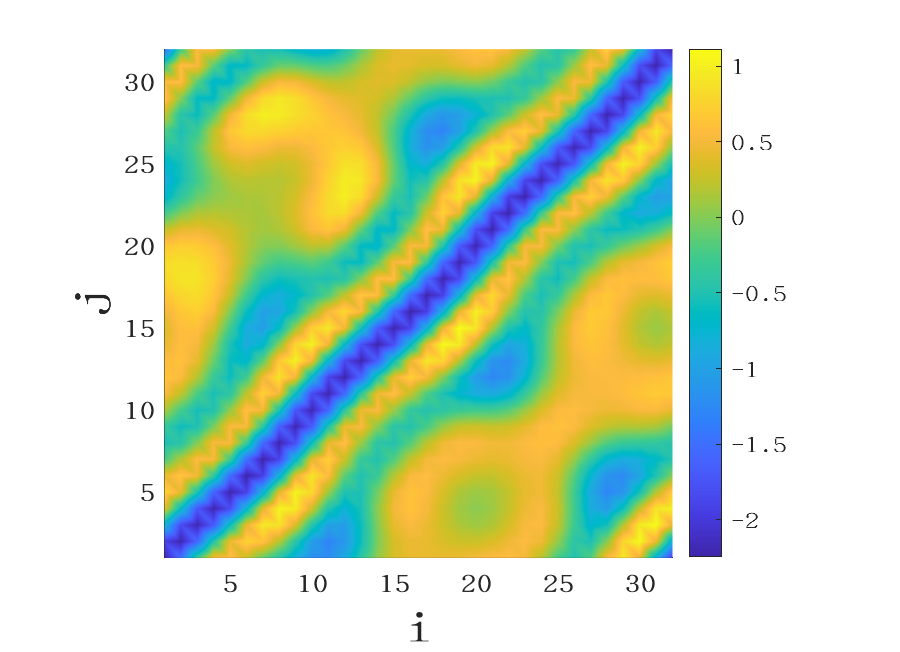}
        \end{overpic}
        \hspace{-0.1cm}   
        \begin{overpic}[width=0.32\textwidth,trim=55 20 55 15, clip=true,tics=10]{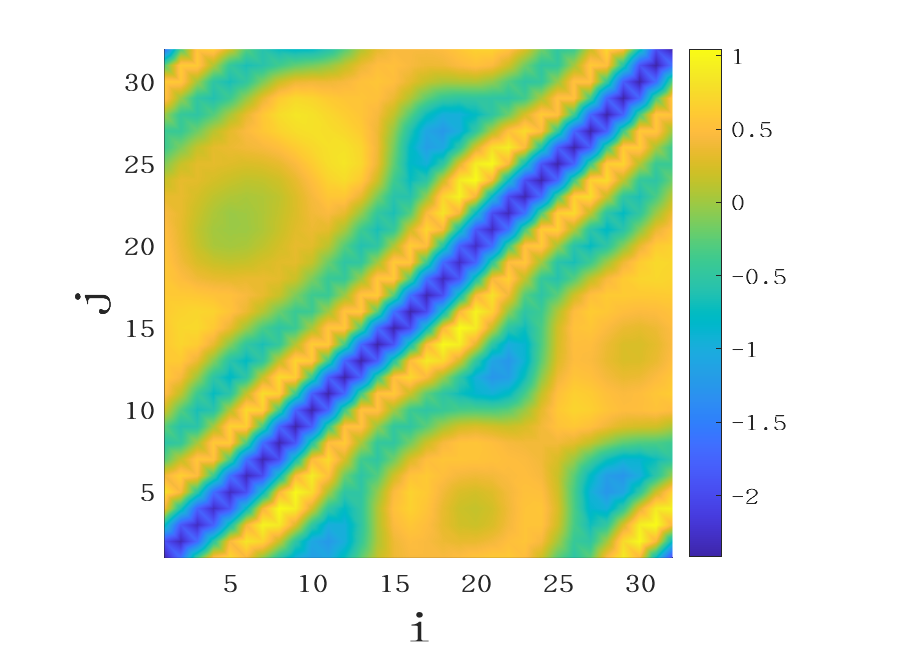}
        \end{overpic}
    \end{center}
        \begin{center}
        \begin{overpic}[width=0.32\textwidth,trim=55 20 55 15, clip=true,tics=10]{square_re_exact.eps}
        \end{overpic}
        \hspace{-0.1cm}   
        \begin{overpic}[width=0.32\textwidth,trim=55 20 55 15, clip=true,tics=10]{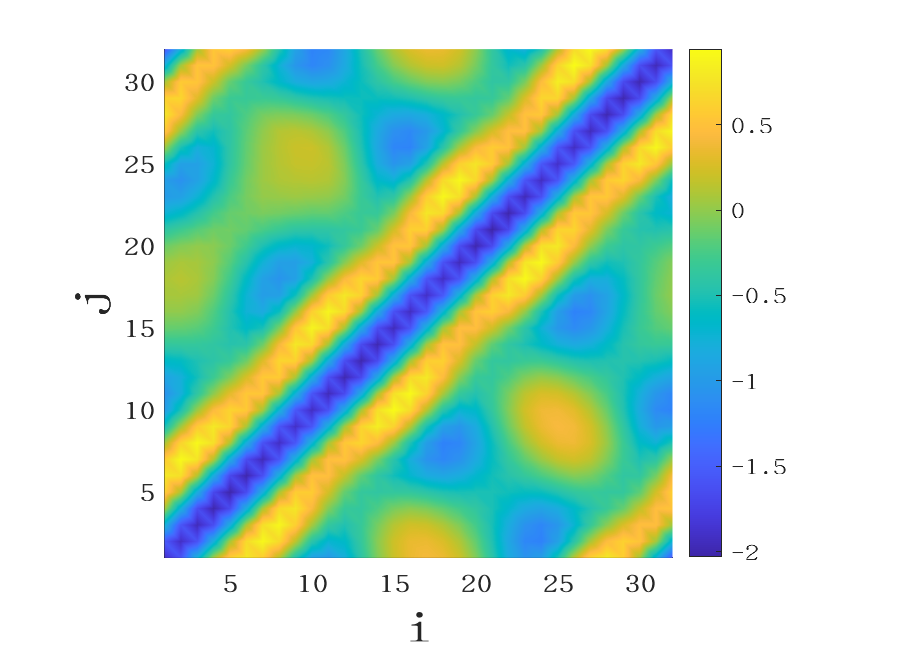}
        \end{overpic}
        \hspace{-0.1cm}   
        \begin{overpic}[width=0.32\textwidth,trim=55 20 55 15, clip=true,tics=10]{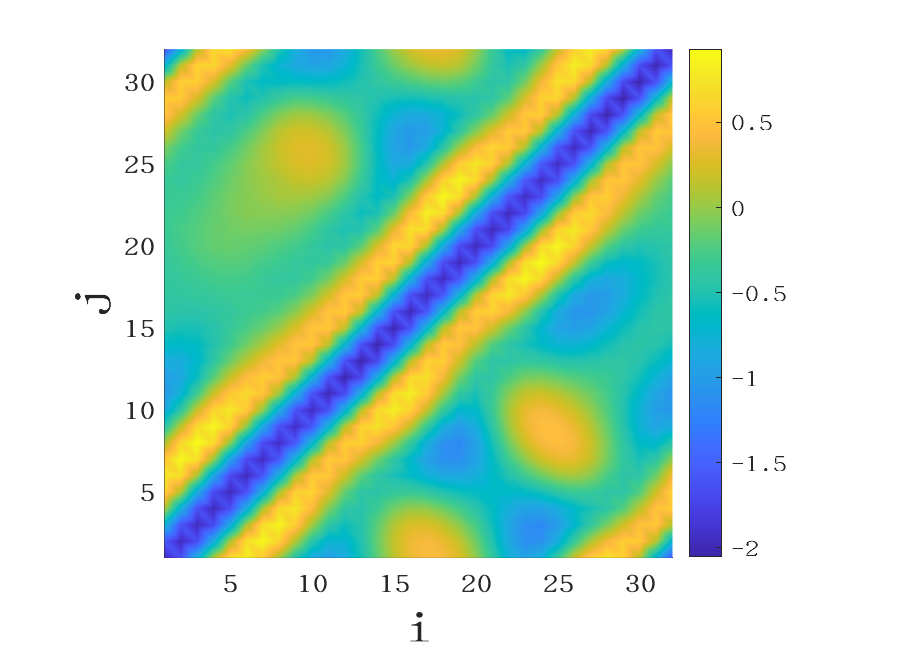}
        \end{overpic}
    \end{center}
    \caption{\textcolor{black}{Example 2: The real parts of the exact and recovered far-field data using various noise levels for Cases 1 (top),  2 (middle) and 3 (bottom).}}
    \label{fig:ex2_re}
\end{figure}
\begin{figure}[ht]
    \begin{center}
        \begin{overpic}[width=0.32\textwidth,trim=55 20 55 15, clip=true,tics=10]{test_im_exact.eps}
           \put (14,86) {\scriptsize {Exact imaginary parts}}
        \end{overpic}
        \hspace{-0.15cm}   
        \begin{overpic}[width=0.32\textwidth,trim=55 20 55 15, clip=true,tics=10]{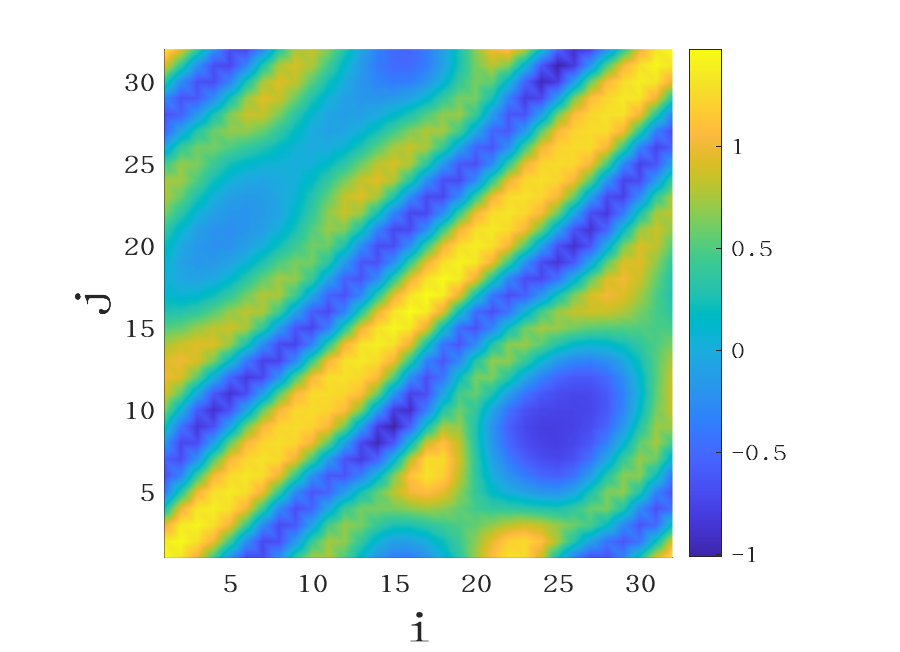}
           \put (30,86) {\scriptsize {$\sigma=0.01$}}
        \end{overpic}
        \hspace{-0.15cm}   
        \begin{overpic}[width=0.32\textwidth,trim=55 20 55 15, clip=true,tics=10]{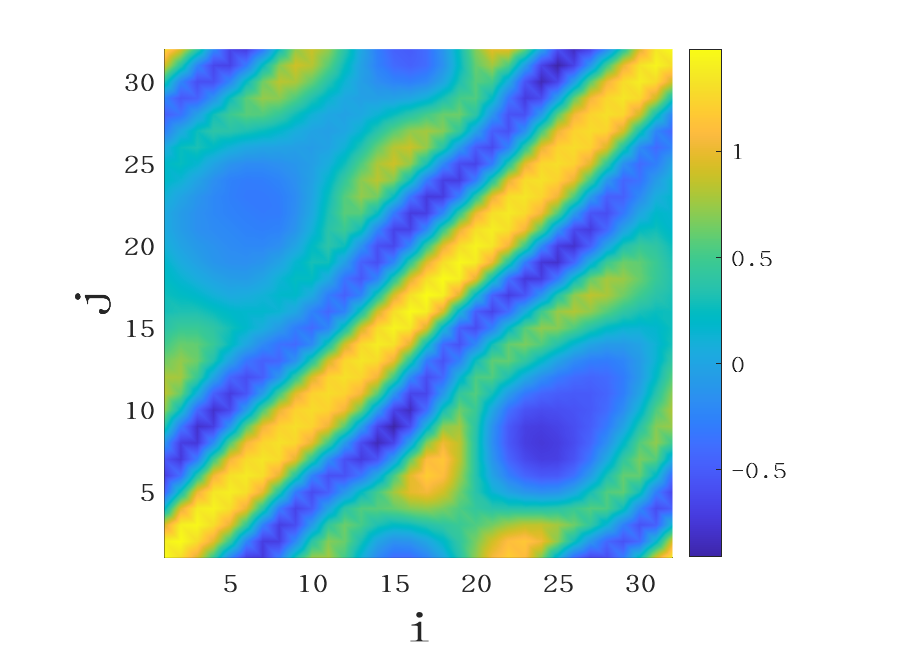}
           \put (30,86) {\scriptsize {$\sigma=0.5$}}
        \end{overpic}
    \end{center}
        \begin{center}
        \begin{overpic}[width=0.325\textwidth,trim=55 20 55 15, clip=true,tics=10]{pear_im_exact.eps}
        \end{overpic}
        \hspace{-0.1cm}   
        \begin{overpic}[width=0.31\textwidth,trim=55 20 55 15, clip=true,tics=10]{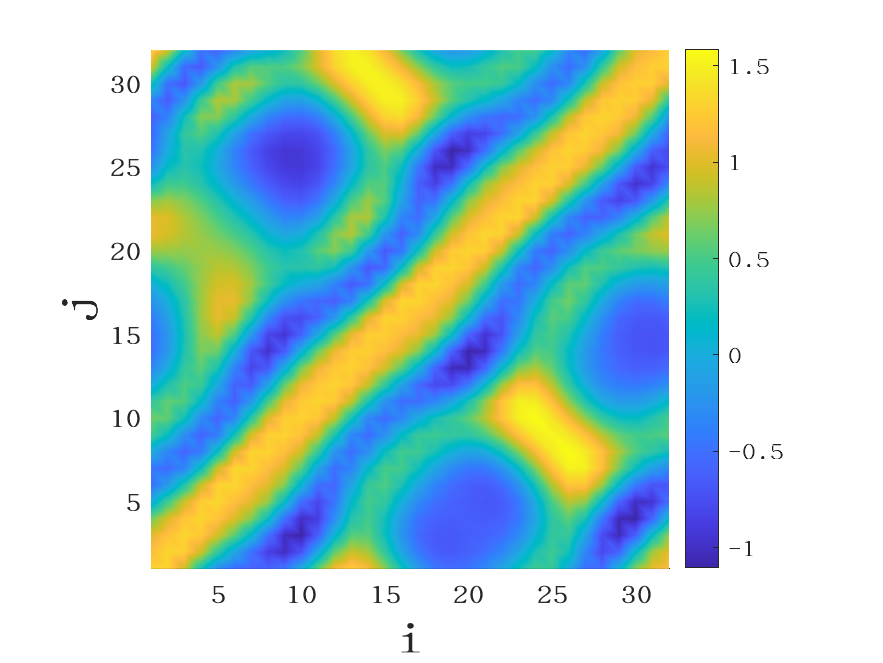}
        \end{overpic}
        \hspace{-0.15cm}   
        \begin{overpic}[width=0.32\textwidth,trim=55 20 55 15, clip=true,tics=10]{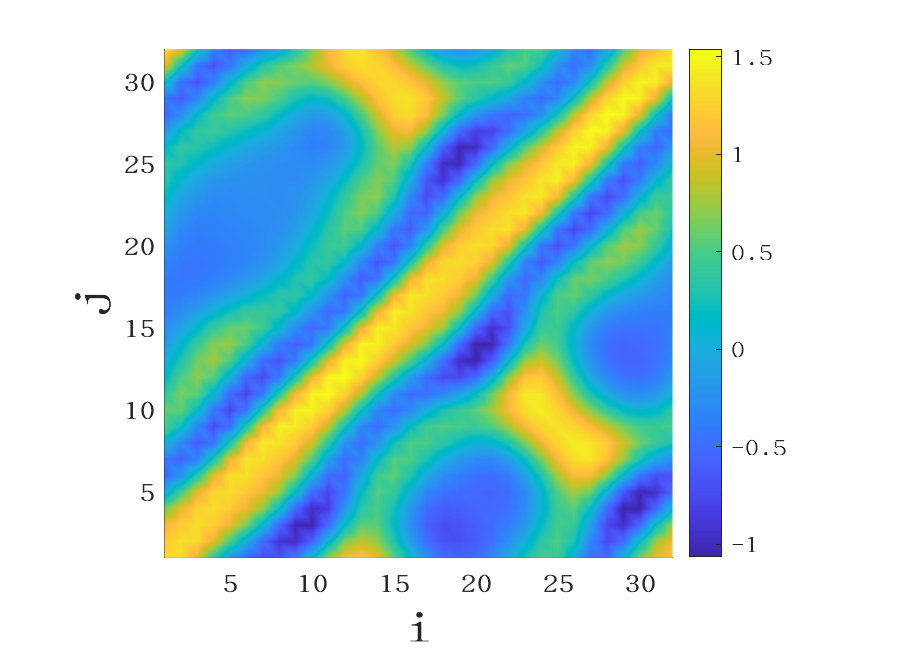}
        \end{overpic}
    \end{center}
    \begin{center}
        \begin{overpic}[width=0.32\textwidth,trim=55 20 55 15, clip=true,tics=10]{square_im_exact.eps}
        \end{overpic}
        \hspace{-0.15cm}   
        \begin{overpic}[width=0.32\textwidth,trim=55 20 55 15, clip=true,tics=10]{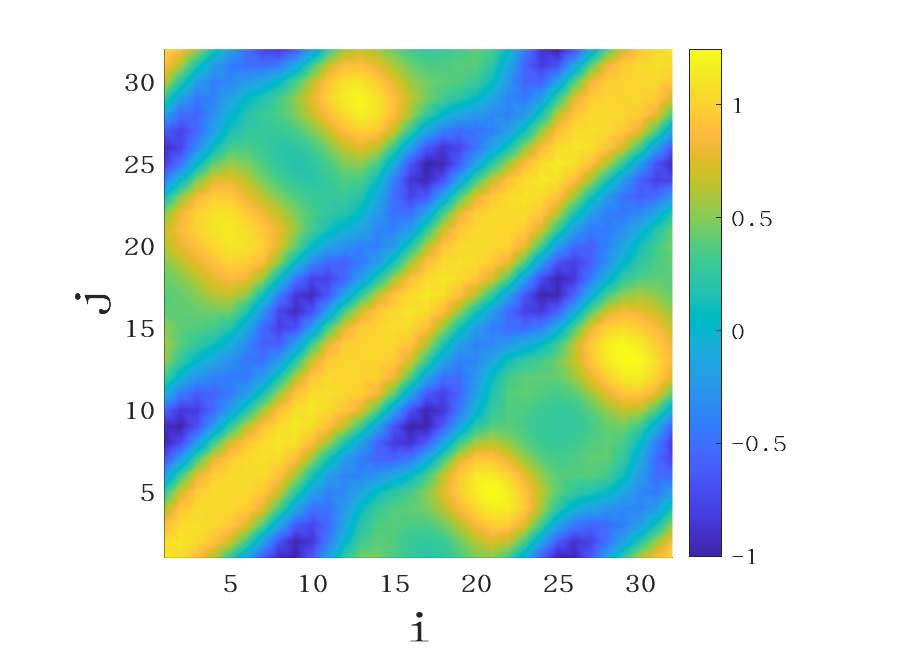}
        \end{overpic}
        \hspace{-0.15cm}   
        \begin{overpic}[width=0.32\textwidth,trim=55 20 55 15, clip=true,tics=10]{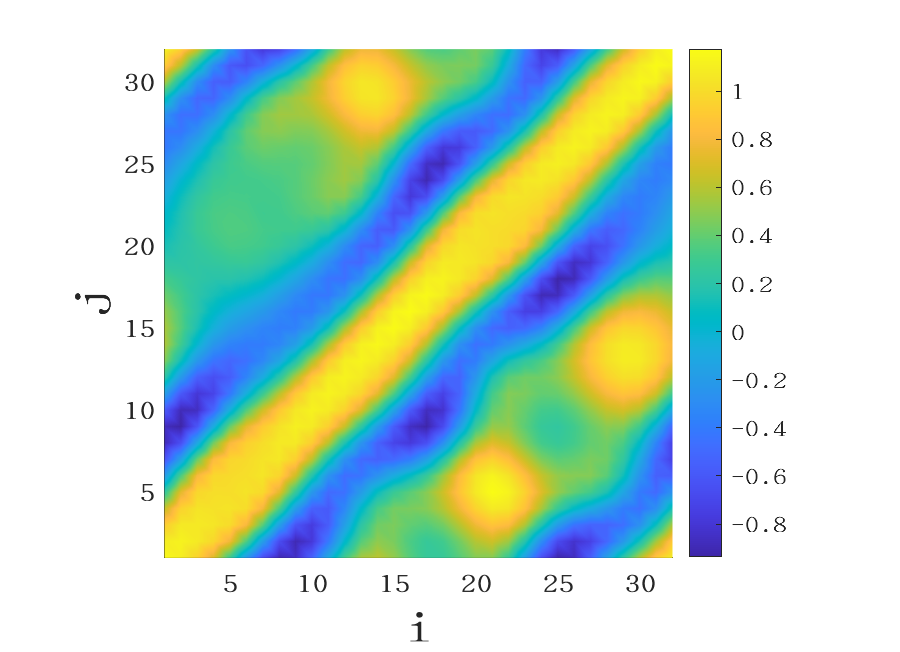}
        \end{overpic}
    \end{center}
    \caption{\textcolor{black}{Example 2: The imaginary parts of the exact and recovered far-field data using various noise levels for Cases 1 (top),  2 (middle) and  3 (bottom).}}
    \label{fig:ex2_im}
\end{figure}

\begin{figure}[ht]
    \begin{center}
        \begin{overpic}[width=0.32\textwidth,trim=40 0 26 15, clip=true,tics=10]{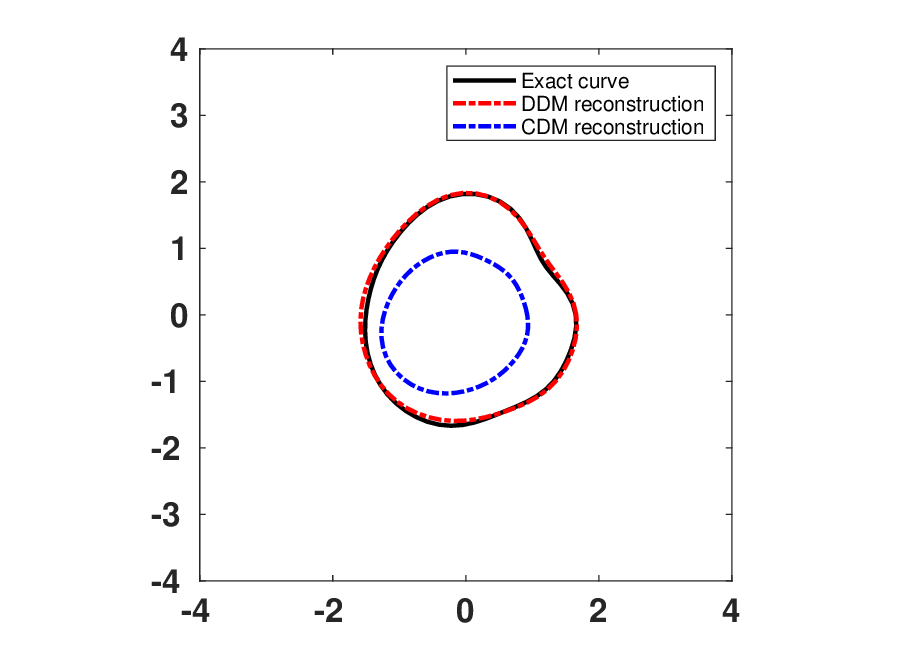}
           \put (40,82) {\scriptsize Case 1}
        \end{overpic}
        \hspace{-0.8cm}   
        \begin{overpic}[width=0.32\textwidth,trim=40 0 26 15, clip=true,tics=10]{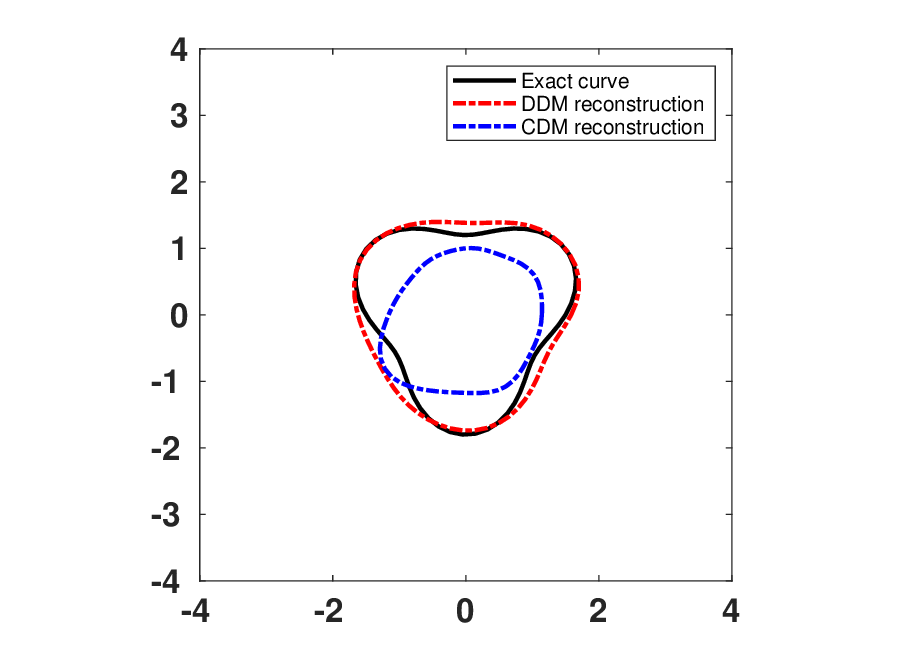}
           \put (40,82){\scriptsize Case 2}
        \end{overpic}
        \hspace{-0.8cm}   
        \begin{overpic}[width=0.32\textwidth,trim=40 0 26 15, clip=true,tics=10]{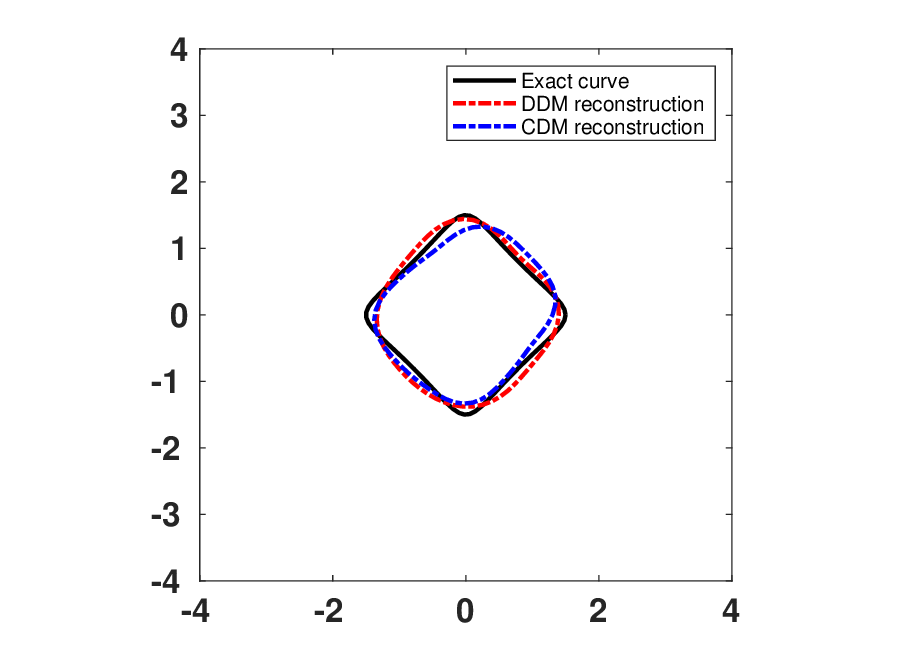}
           \put (40,82){\scriptsize Case 3}
        \end{overpic}
    \end{center}
        \begin{center}
        \begin{overpic}[width=0.32\textwidth,trim=40 0 26 15, clip=true,tics=10]{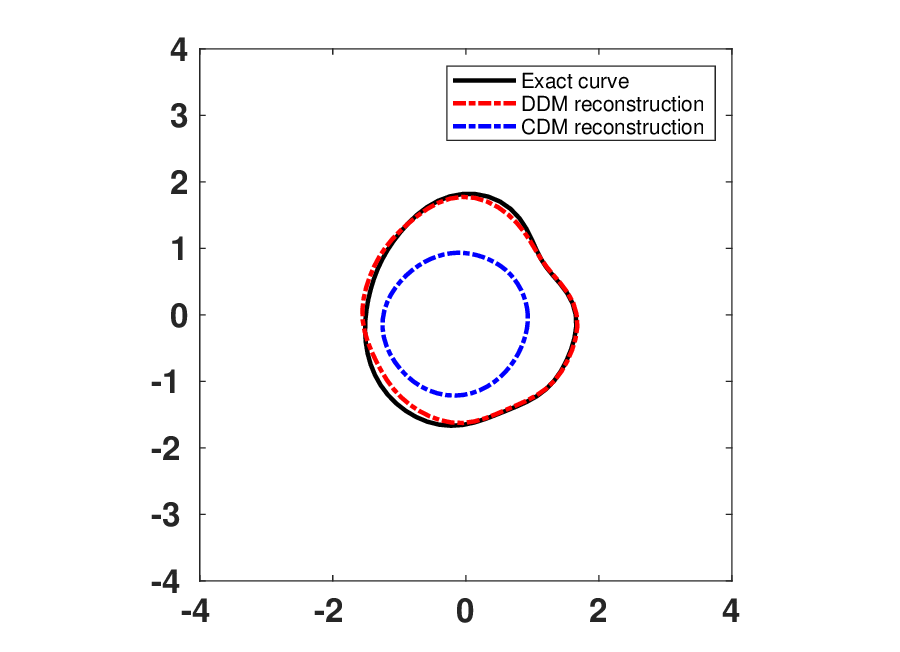}
        \end{overpic}
        \hspace{-0.8cm}   
        \begin{overpic}[width=0.32\textwidth,trim=40 0 26 15, clip=true,tics=10]{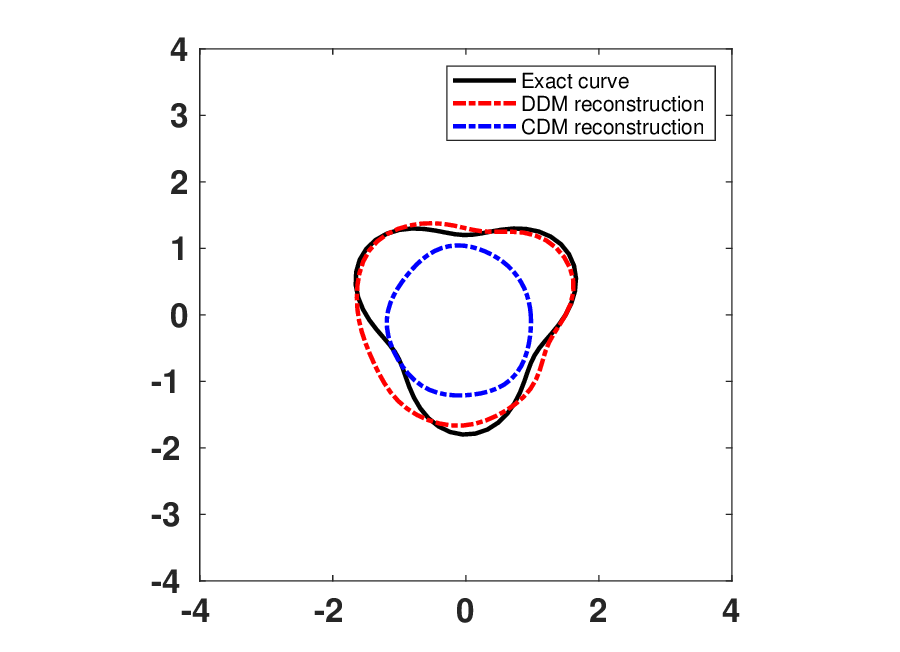}
        \end{overpic}
        \hspace{-0.8cm}   
        \begin{overpic}[width=0.32\textwidth,trim=40 0 26 15, clip=true,tics=10]{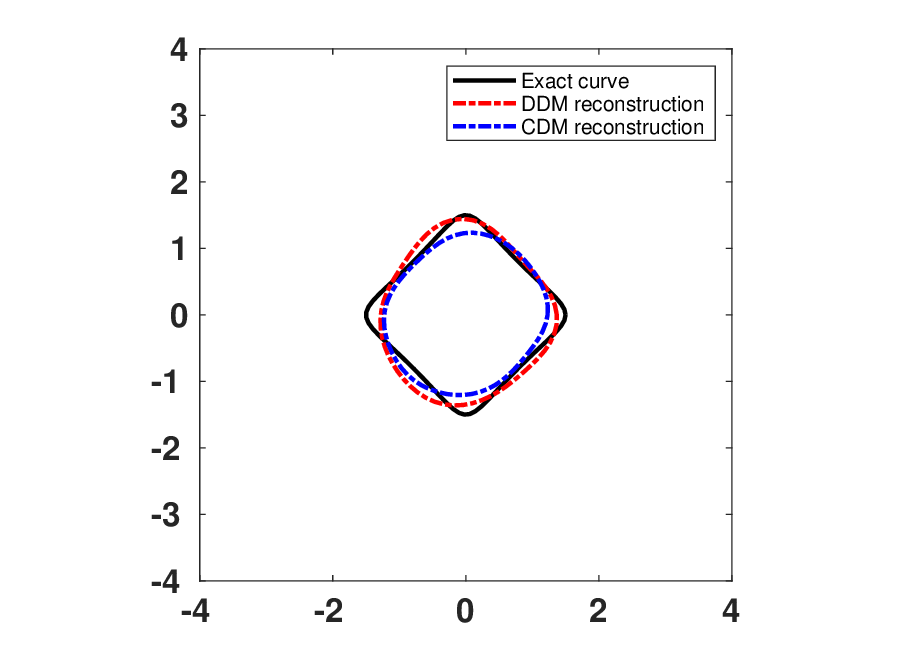}
        \end{overpic}
    \end{center}
    \caption{\textcolor{black}{Example 2: Reconstructions made by DDM and CDM using various noise levels: $\sigma=0.01$ (top), $\sigma=0.5$ (bottom).}}
    \label{fig:ex2_ddm}
\end{figure}
\begin{table}[htbp]
\caption{\textcolor{black}{Example 2: The relative errors $\bar{e}$  for the considered three cases using various noise levels.}}
\centering
\begin{tabular}{llcc}
\toprule
&\multicolumn{1}{l}{}&\multicolumn{1}{c}{$\bar{e}$}\\
\cmidrule(lr){2-4}
& Case 1 & Case 2 & Case 3\\
\midrule
DDM $\sigma=0.01$  & 0.0246 & 0.0884 & 0.0669\\
DDM $\sigma=0.50$  & 0.0348 & 0.0954 & 0.0828\\
CDM $\sigma=0.01$ & 0.3602 & 0.2915 & 0.0949\\
CDM $\sigma=0.50$ & 0.3586 & 0.3077 & 0.1278\\
\bottomrule
\end{tabular}
\label{ex2_accuracy_cost}
\end{table}

\textcolor{black}{In this example, we examine the stability of the proposed DDM using noisy, limited-aperture data. For this purpose, we set noise levels at $\sigma=0.01$ and $\sigma=0.5$. The observation aperture is set to $\phi = [0, \pi/2]$, and the incident aperture to $\psi = [0, 2\pi]$.}

\textcolor{black}{Fig.\ref{fig:ex2_loss_convergence} shows the evolution of the loss functions $\mathcal{L}_{DDM}$ and $\mathcal{L}_{phy}$, as well as the training relative error $\mathrm{Err}$, over the training period.  As illustrated in the right panel of the figure, the convergence behavior observed during training is consistent for $\sigma = 0.01$, and the training relative error $\mathrm{Err}$ remains relatively stable at $\sigma = 0.5$. As expected, as the noise level $\sigma$ increases, the training relative error $\mathrm{Err}$ also increases due to the added uncertainty from higher noise. Figs. \ref{fig:ex2_re} and \ref{fig:ex2_im} display the real and imaginary parts of the recovered far-field data for $\sigma = 0.01$ and $\sigma = 0.5$, respectively. The recovered data for $\sigma = 0.01$ closely resembles that in Example 1, as the noise level is relatively low. However, at $\sigma = 0.5$, the quality of the far-field data reconstruction deteriorates significantly, particularly in the upper-left region.}

\textcolor{black}{The boundary reconstructions by DDM and CDM with $\sigma=0.01$ and $\sigma=0.5$ are shown in Fig.\ref{fig:ex2_ddm}.  The results indicate that, even with noisy data, DDM effectively reconstructs the boundaries for all three cases, while CDM produces results that are closer to a unit circle. This robustness of DDM against observational noise can be attributed to its integration of recovered far-field data into the far-field operator in Equation \eqref{eq:farfield_operator_theta}, which partially mitigates the effects of added noise. However, as $\sigma$ increases, the reconstruction accuracy of DDM gradually deteriorates.  Tab.\ref{ex2_accuracy_cost} presents the relative errors $\bar{e}$ for both DDM and CDM. Notably, even in the noisy setting, DDM maintains a lower relative error $\bar{e}$ than CDM.  The test relative errors $\mathrm{TErr}$ for DDM at noise levels $\sigma = 0.01$ and $\sigma = 0.5$ are 0.2435 and 0.3810, respectively. Notably, $\mathrm{TErr}$ for $\sigma = 0.01$ remains close to the value for $\sigma = 0$ (from Example 1), which supports Theorem \ref{thm:expectation_noisy}.}

\subsubsection{Example 3: Different incident apertures} In the previous two examples, only a limited observation aperture was considered. However, in many real-world scenarios, the incident aperture is also restricted. Therefore, in this example, we examine the effect of a limited incident aperture. We set the incident apertures as $\psi = [0, \pi]$ and $\psi = [0, \pi/2]$, with an observation aperture of $\phi = [0, \pi/2]$ and a noise level of $\sigma = 0.01$. It is important to note that both the incident and observation apertures are highly restricted, making this inverse problem significantly more challenging.

\begin{figure}[th]
    \begin{center}
        \begin{overpic}[width=0.324\textwidth]{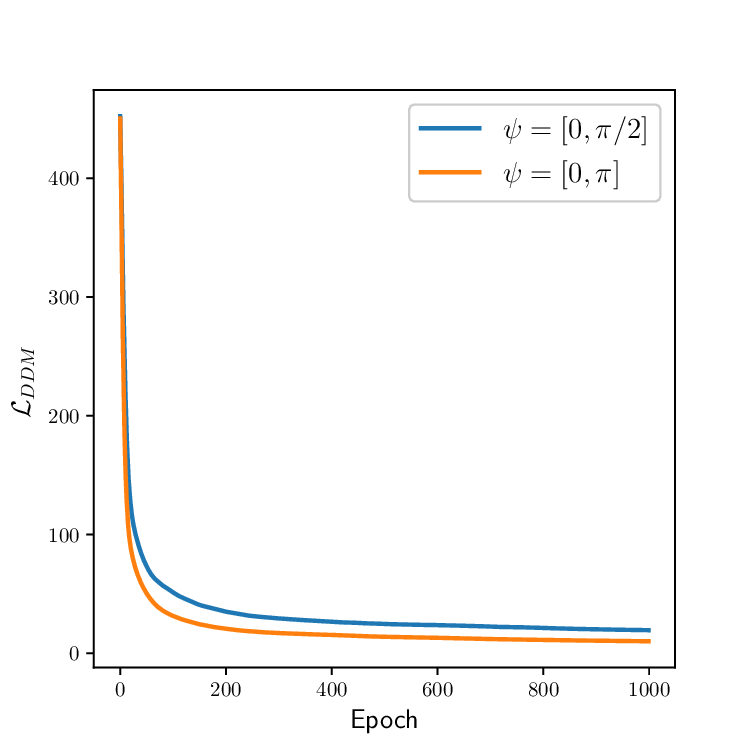}
           \put (40,91) {\scriptsize {$\mathcal{L}_{DDM}$}}
        \end{overpic}
        \hspace{-0.4cm} 
        \begin{overpic}[width=0.324\textwidth]{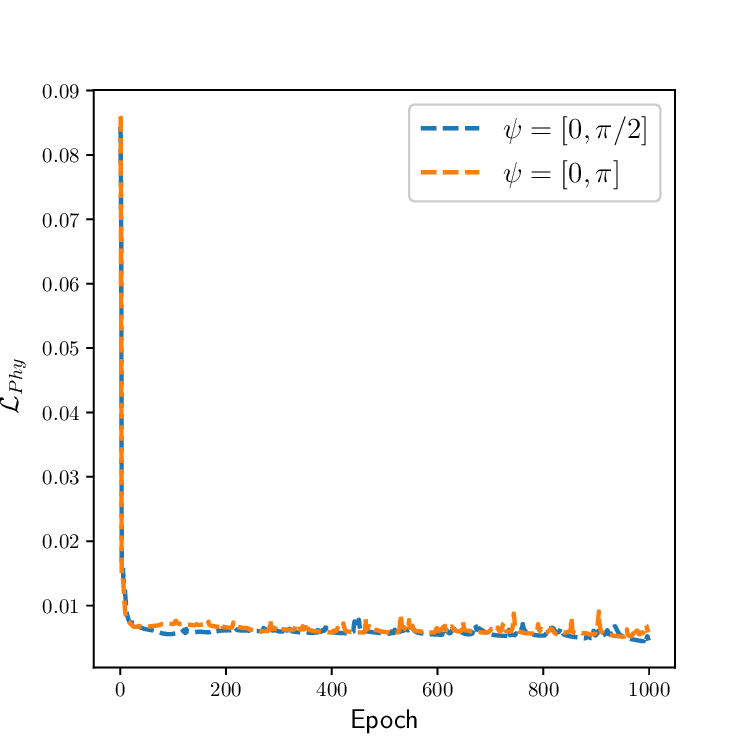}
           \put (44,91) {\scriptsize {$\mathcal{L}_{phy}$}}
        \end{overpic}
        \hspace{-0.4cm} 
        \begin{overpic}[width=0.324\textwidth]{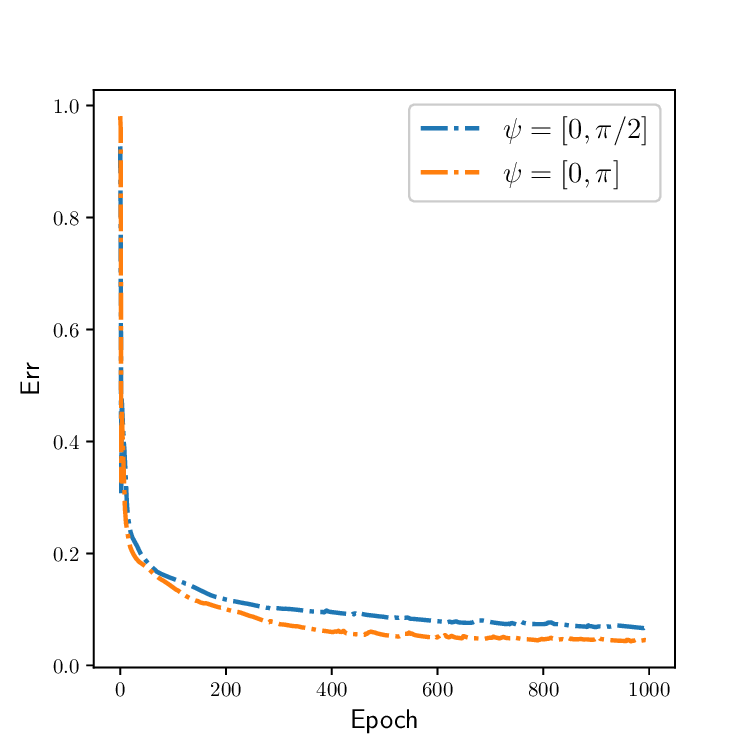}
           \put (47,91) {\scriptsize {Err}}
        \end{overpic}
    \end{center}
    \caption{\textcolor{black}{Example 3: The evolution of the DDM loss function $\mathcal{L}_{DDM}$, the physics-based loss function $\mathcal{L}_{phy}$, and the training relative error $\mathrm{Err}$, throughout the training process.}}
    \label{fig:ex3_loss_convergence}
\end{figure}
\begin{figure}[ht]
    \begin{center}
        \begin{overpic}[width=0.32\textwidth,trim=55 20 55 15, clip=true,tics=10]{test_re_exact.eps}
           \put (19,86) {\scriptsize {Exact real parts}}
        \end{overpic}
        \hspace{-0.15cm} 
        \begin{overpic}[width=0.32\textwidth,trim=55 20 55 15, clip=true,tics=10]{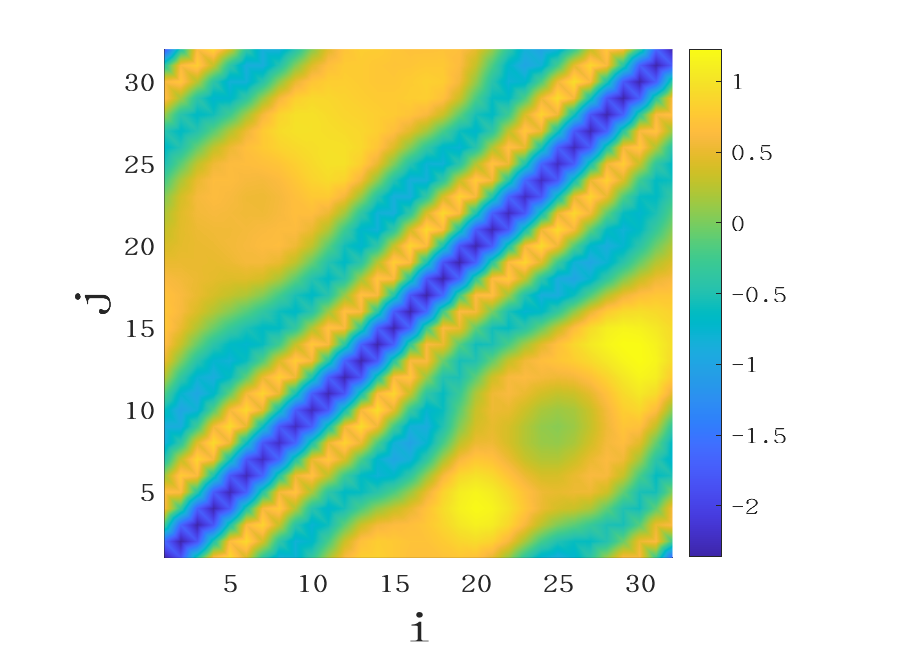}
           \put (27,86) {\scriptsize {$\psi=[0,\pi]$}}
        \end{overpic}
        \hspace{-0.15cm} 
       \begin{overpic}[width=0.32\textwidth,trim=55 20 55 15, clip=true,tics=10]{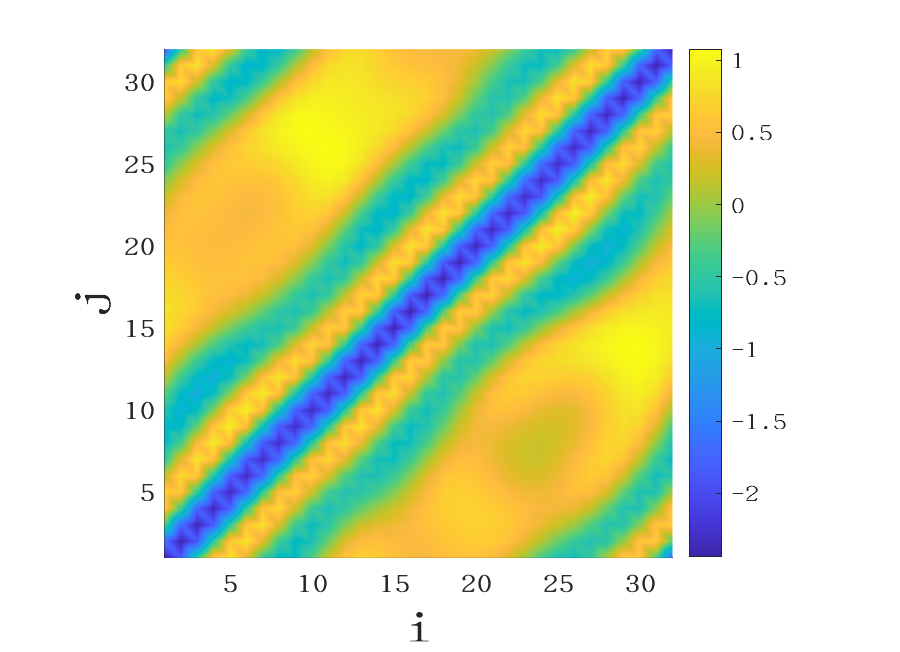}
           \put (27,86) {\scriptsize {$\psi=[0,\pi/2]$}}
        \end{overpic}
    \end{center}
        \begin{center}
        \begin{overpic}[width=0.32\textwidth,trim=55 20 55 15, clip=true,tics=10]{pear_re_exact.eps}
        \end{overpic}
        \hspace{-0.15cm} 
        \begin{overpic}[width=0.32\textwidth,trim=55 20 55 15, clip=true,tics=10]{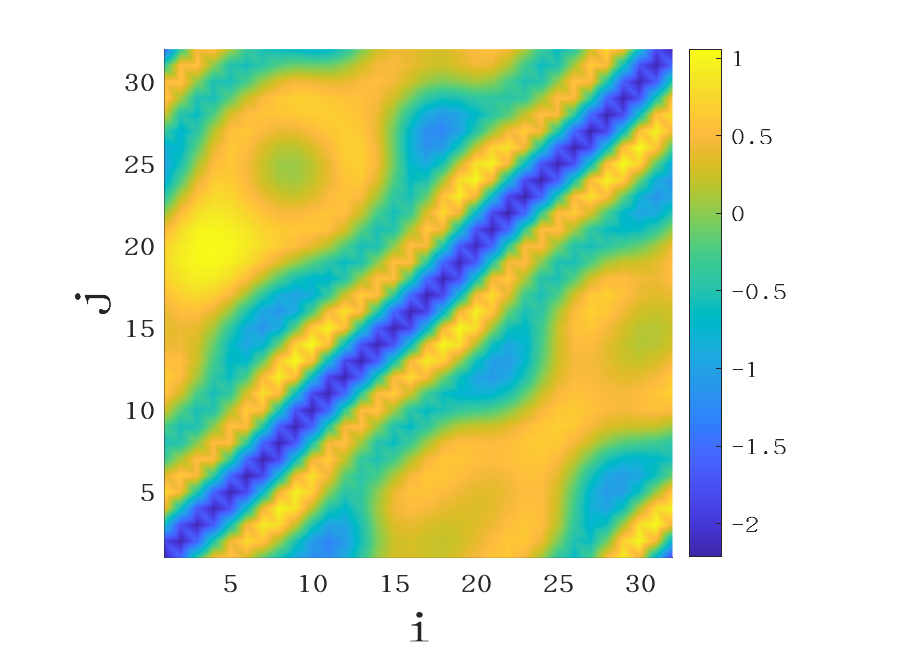}
        \end{overpic}
        \hspace{-0.15cm} 
        \begin{overpic}[width=0.32\textwidth,trim=55 20 55 15, clip=true,tics=10]{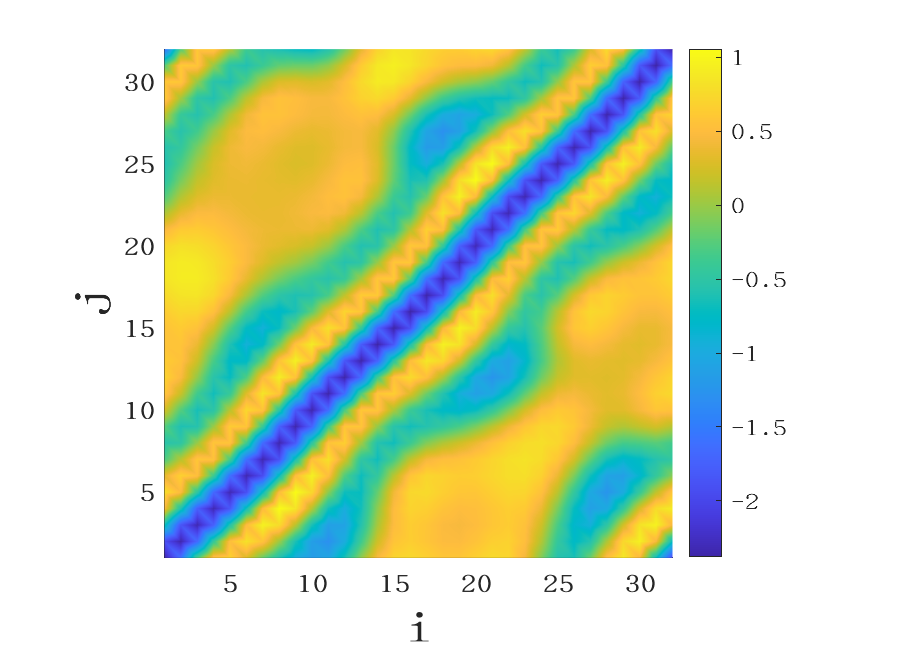}
        \end{overpic}
    \end{center}
            \begin{center}
        \begin{overpic}[width=0.32\textwidth,trim=55 20 55 15, clip=true,tics=10]{square_re_exact.eps}
        \end{overpic}
        \hspace{-0.15cm} 
        \begin{overpic}[width=0.32\textwidth,trim=55 20 55 15, clip=true,tics=10]{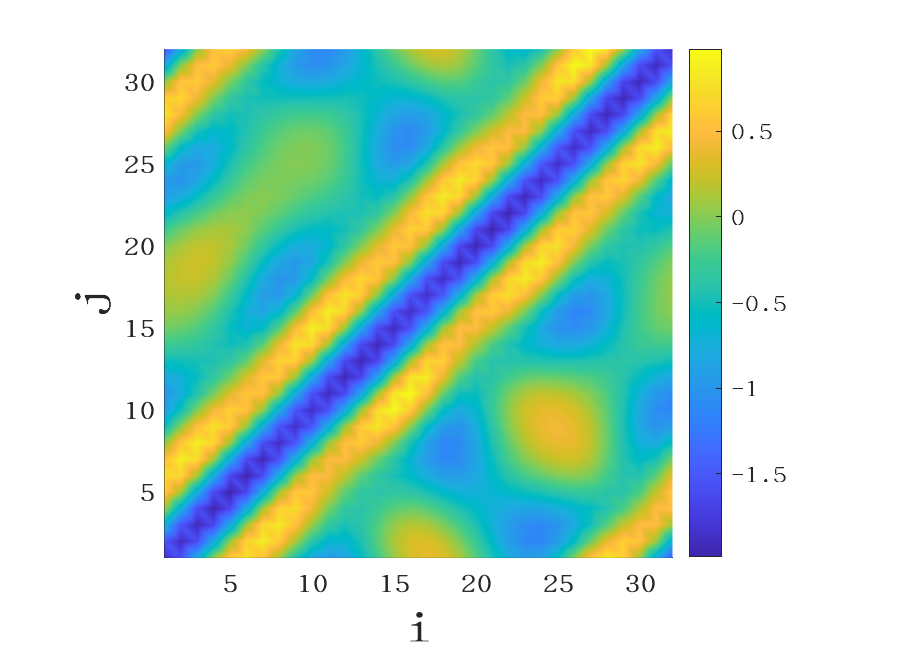}
        \end{overpic}
        \hspace{-0.15cm} 
         \begin{overpic}[width=0.32\textwidth,trim=55 20 55 15, clip=true,tics=10]{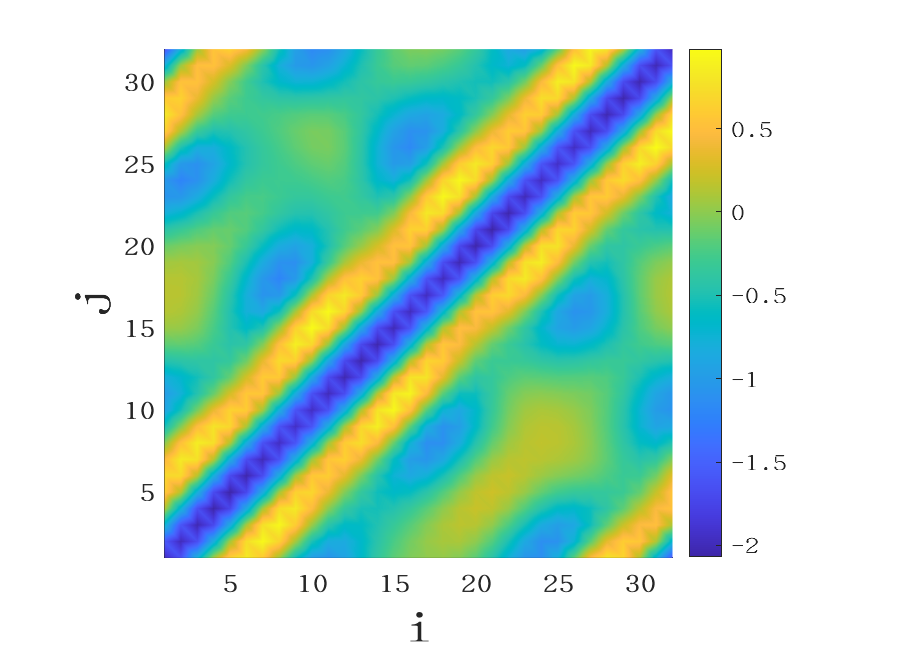}
        \end{overpic}
    \end{center}
    \caption{\textcolor{black}{Example 3: The real parts of the exact and recovered far-field data using various incident apertures for Cases 1 (top),  2 (middle) and  3 (bottom).}}
    \label{fig:ex3_re}
\end{figure}
\begin{figure}[ht]
    \begin{center}
        \begin{overpic}[width=0.32\textwidth,trim=55 20 55 15, clip=true,tics=10]{test_im_exact.eps}
           \put (14,86) {\scriptsize {Exact imaginary parts}}
        \end{overpic}
        \hspace{-0.15cm} 
        \begin{overpic}[width=0.32\textwidth,trim=55 20 55 15, clip=true,tics=10]{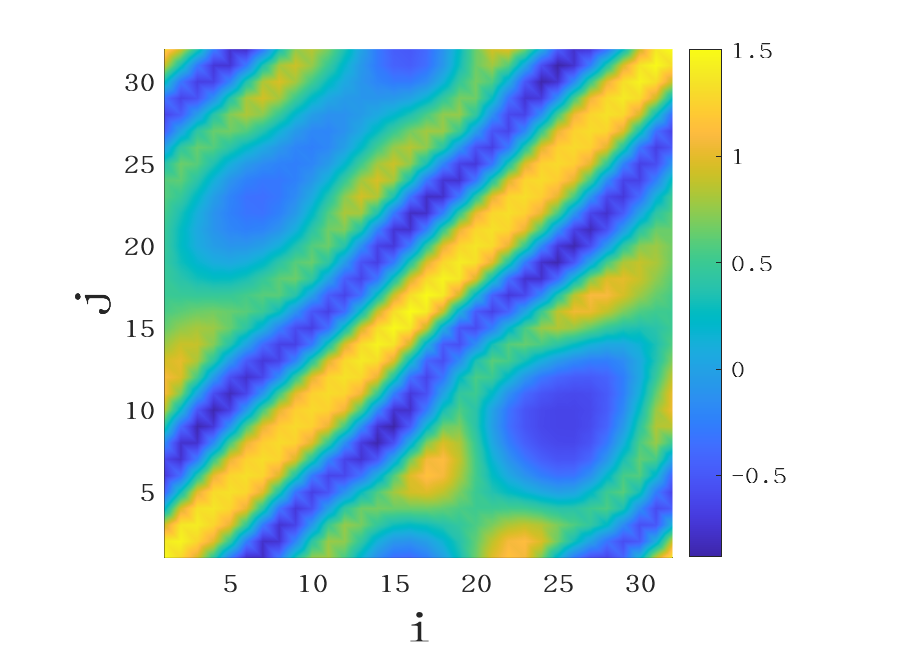}
           \put (28,86) {\scriptsize {$\psi=[0,\pi]$}}
        \end{overpic}
        \hspace{-0.15cm} 
        \begin{overpic}[width=0.32\textwidth,trim=55 20 55 15, clip=true,tics=10]{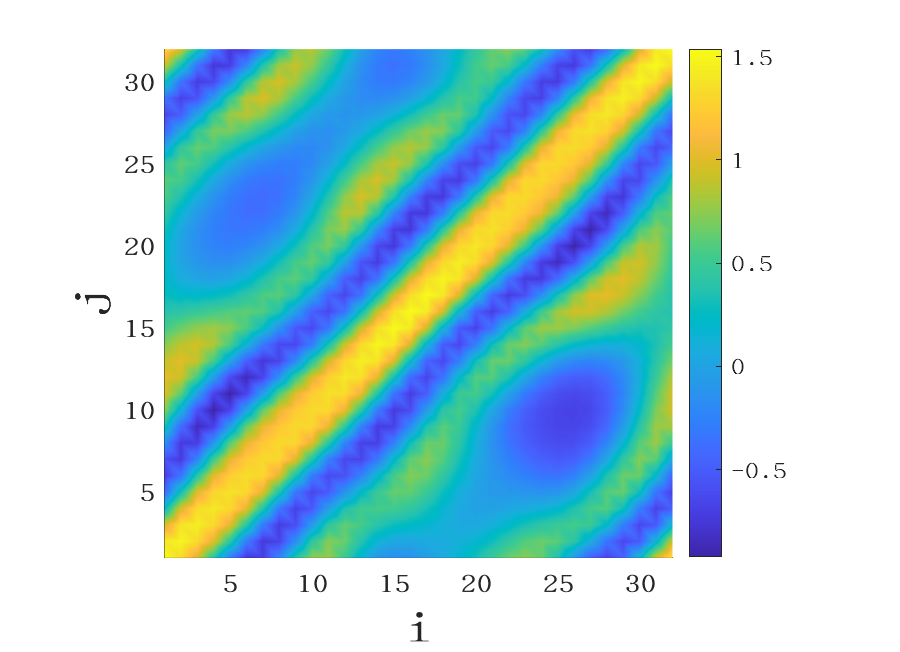}
           \put (27,86) {\scriptsize {$\psi=[0,\pi/2]$}}
        \end{overpic}
    \end{center}
        \begin{center}
        \begin{overpic}[width=0.32\textwidth,trim=55 20 55 15, clip=true,tics=10]{pear_im_exact.eps}
        \end{overpic}
        \hspace{-0.15cm} 
        \begin{overpic}[width=0.32\textwidth,trim=55 20 55 15, clip=true,tics=10]{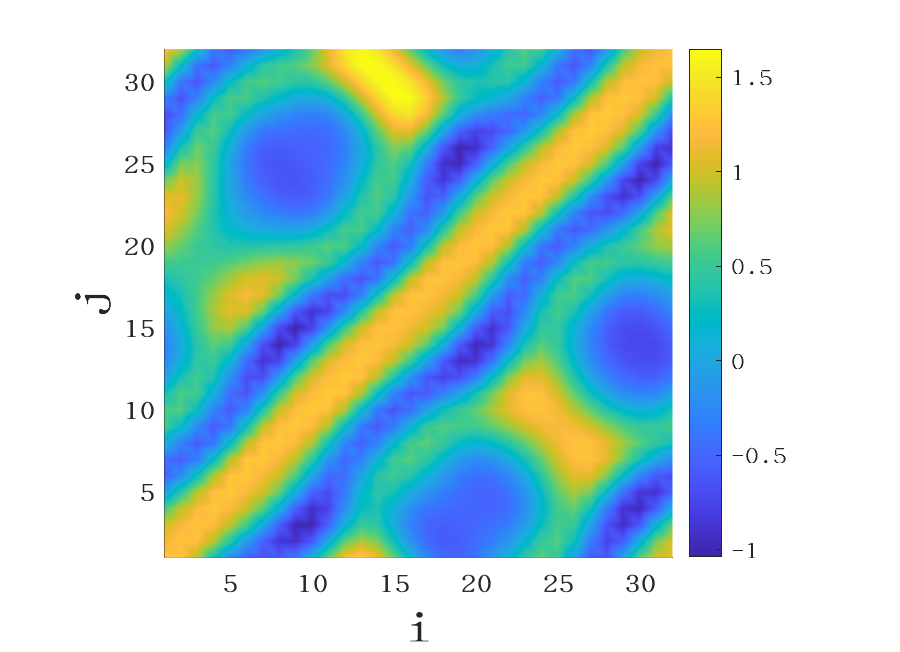}
        \end{overpic}
        \hspace{-0.15cm} 
        \begin{overpic}[width=0.32\textwidth,trim=55 20 55 15, clip=true,tics=10]{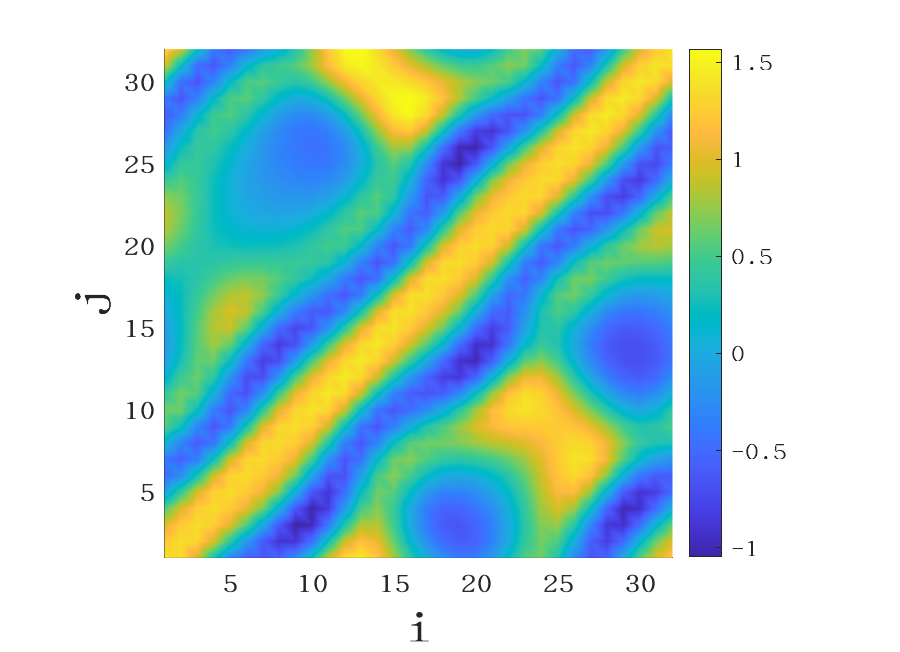}
        \end{overpic}
    \end{center}
            \begin{center}
        \begin{overpic}[width=0.32\textwidth,trim=55 20 55 15, clip=true,tics=10]{square_im_exact.eps}
        \end{overpic}
        \hspace{-0.15cm} 
        \begin{overpic}[width=0.32\textwidth,trim=55 20 55 15, clip=true,tics=10]{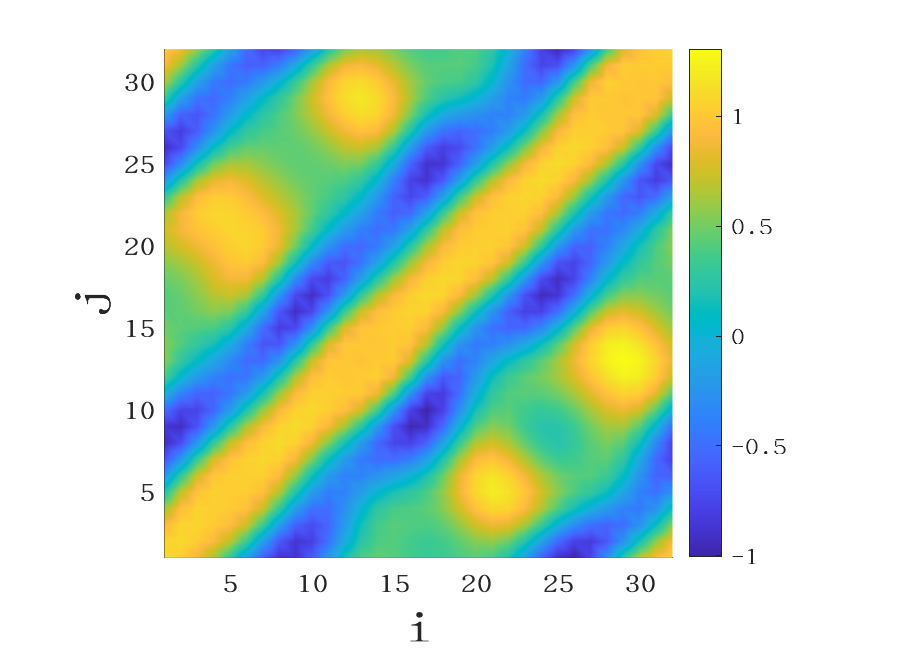}
        \end{overpic}
        \hspace{-0.15cm} 
        \begin{overpic}[width=0.32\textwidth,trim=55 20 55 15, clip=true,tics=10]{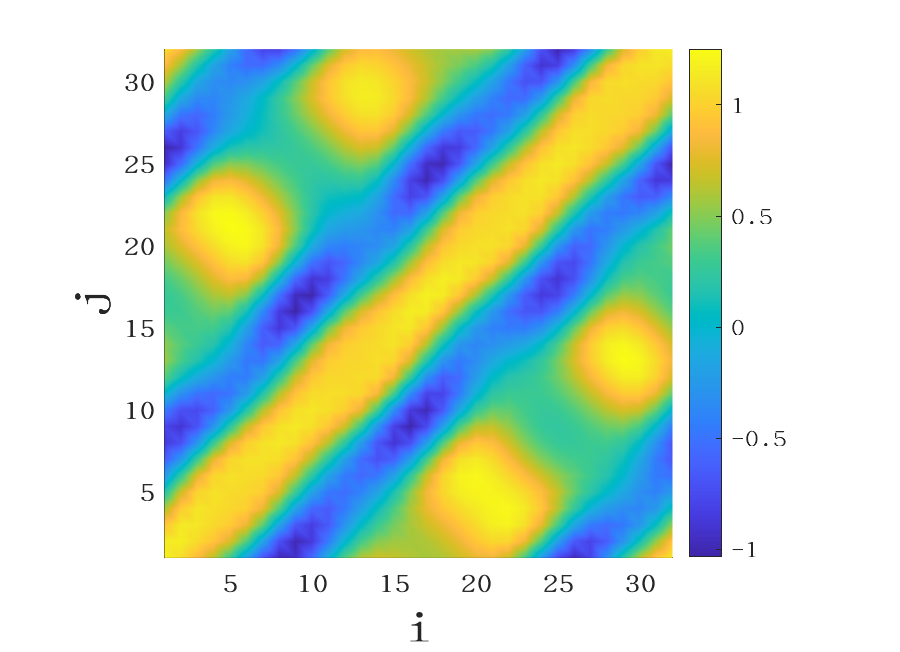}
        \end{overpic}
    \end{center}
    \caption{\textcolor{black}{Example 3:  The imaginary parts of the exact and recovered far-field data using various incident apertures for Cases 1(top),  2(middle) and 3(bottom).}}
    \label{fig:ex3_im}
\end{figure}

\begin{figure}[ht]
    \begin{center}
        \begin{overpic}[width=0.32\textwidth,trim=40 0 26 15, clip=true,tics=10]{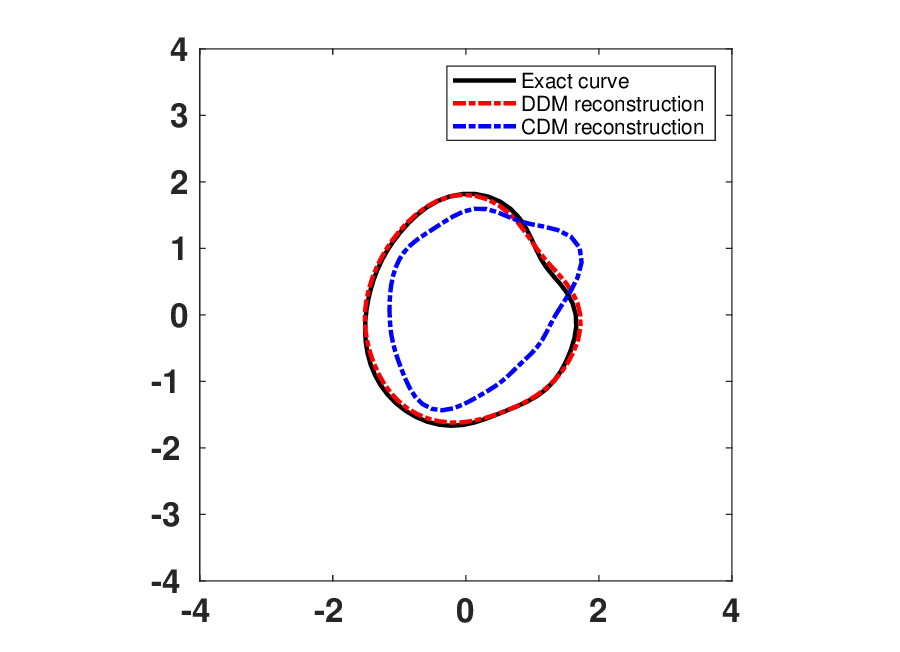}
           \put (40,84) {\scriptsize Case 1}
        \end{overpic}
        \hspace{-0.8cm} 
        \begin{overpic}[width=0.32\textwidth,trim=40 0 26 15, clip=true,tics=10]{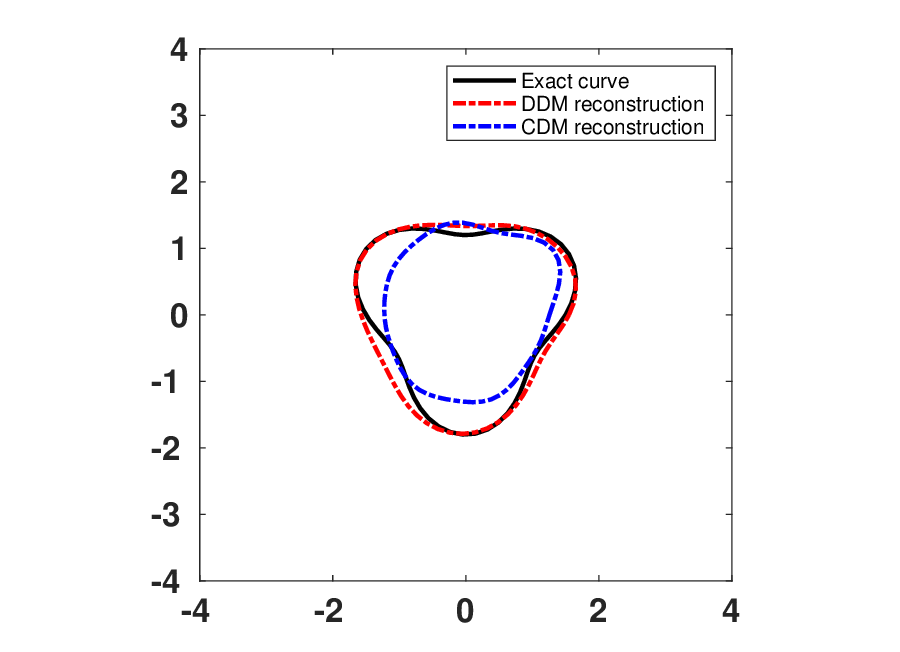}
          \put (40,84){\scriptsize Case 2}
        \end{overpic}
        \hspace{-0.8cm} 
        \begin{overpic}[width=0.32\textwidth,trim=40 0 26 15, clip=true,tics=10]{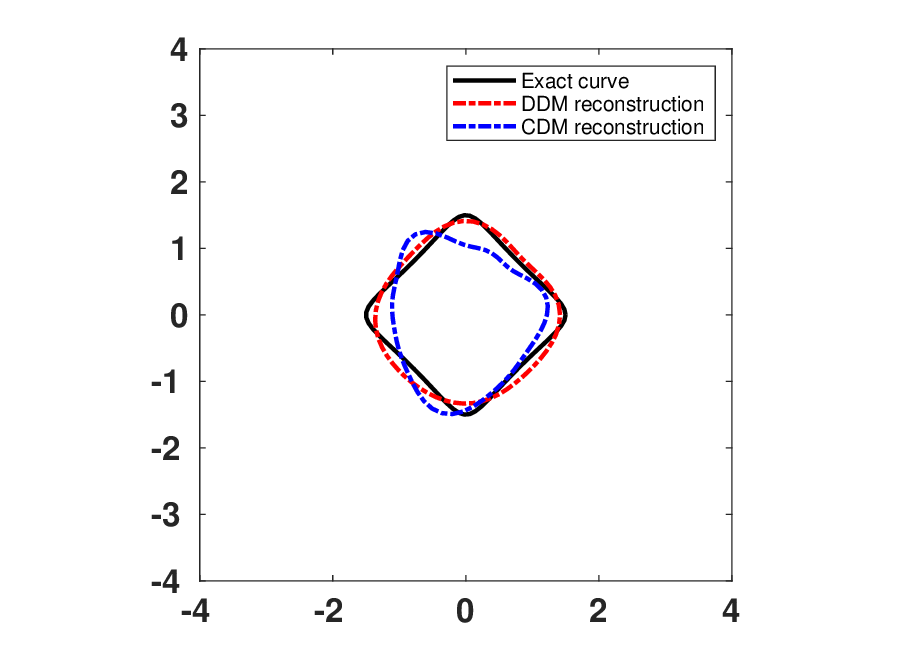}
           \put (40,84){\scriptsize Case 3}
        \end{overpic}
    \end{center}
        \begin{center}
        \begin{overpic}[width=0.32\textwidth,trim=40 0 26 15, clip=true,tics=10]{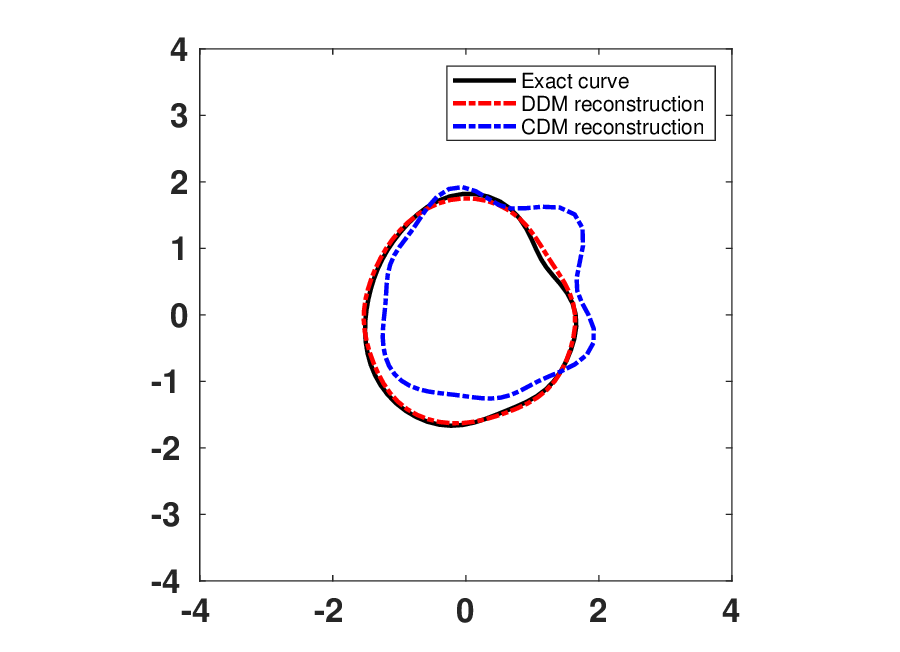}
        \end{overpic}
        \hspace{-0.8cm} 
        \begin{overpic}[width=0.32\textwidth,trim=40 0 26 15, clip=true,tics=10]{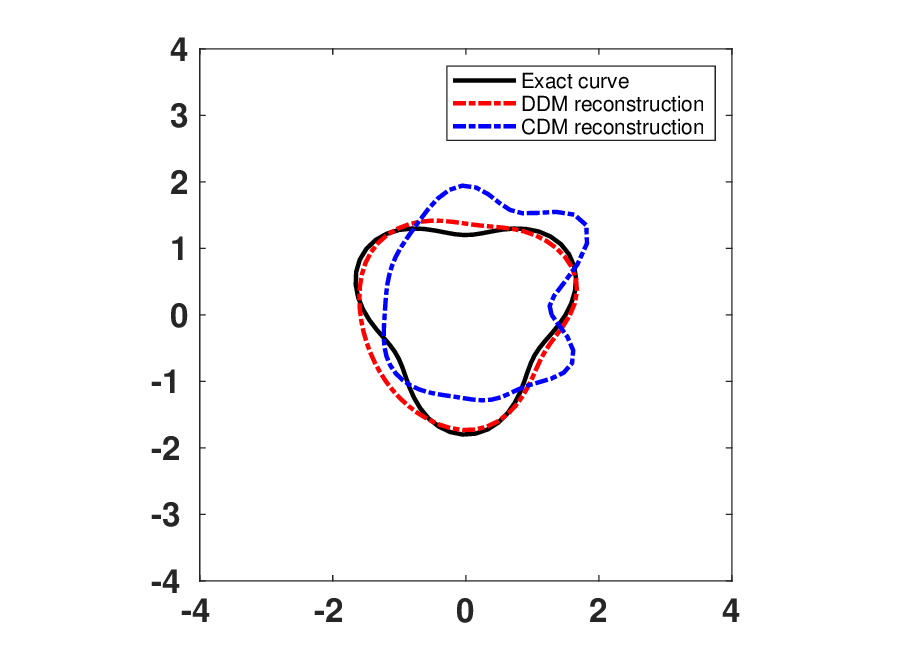}
        \end{overpic}
        \hspace{-0.8cm} 
        \begin{overpic}[width=0.32\textwidth,trim=40 0 26 15, clip=true,tics=10]{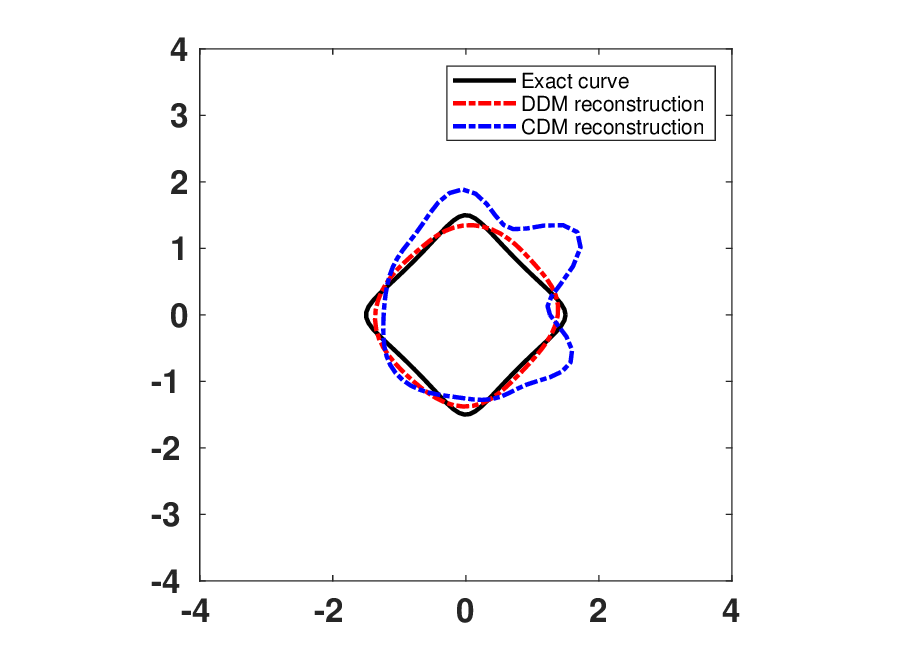}
        \end{overpic}
    \end{center}
    \caption{\textcolor{black}{Example 3: Reconstructions made by DDM and CDM using various incident apertures: $\psi=[0,\pi]$(top) and $\psi=[0,\pi/2]$ (bottom).}}
    \label{fig:ex3_ddm}
\end{figure}
\begin{table}[htbp]
\caption{\textcolor{black}{Example 3: The relative errors $\bar{e}$ of DDM and CDM  using various incident apertures.}}
\centering
\begin{tabular}{llcc}
\toprule
&\multicolumn{1}{l}{}&\multicolumn{1}{c}{$\bar{e}$}\\
\cmidrule(lr){2-4}
 & Case 1 & Case 2 & Case 3\\
\midrule
DDM $\psi=[0,\pi]$ & 0.0255 & 0.0672 & 0.0706\\
DDM $\psi=[0,\pi/2]$  & 0.0265 & 0.0893 & 0.0697\\
CDM $\psi=[0,\pi]$ & 0.2178 & 0.1720 & 0.1712\\
CDM $\psi=[0,\pi/2]$  & 0.1898 & 0.2454 & 0.2325\\
\bottomrule
\end{tabular}
\label{ex3_accuracy_cost}
\end{table}

\textcolor{black}{The corresponding results are presented in Figs. \ref{fig:ex3_loss_convergence}–\ref{fig:ex3_im}. The right panel of Fig. \ref{fig:ex3_loss_convergence} clearly shows that DDM can gradually approximate the exact boundary $\partial D$ even with limited incident and observation apertures. Figs. \ref{fig:ex3_re} and \ref{fig:ex3_im} illustrate that for incident apertures of $\psi = [0, \pi]$ and $\psi = [0, \pi/2]$, the recovered data has some accuracy loss  when compared to Example 2. This is expected, as learning the analytic continuation becomes more challenging in these cases. However, it is noting that DDM is still able to capture the reciprocity property. Specifically, with an incident aperture of $\psi = [0, \pi]$, the far-field data is well recovered. The boundary reconstruction results are shown in Fig. \ref{fig:ex3_ddm} and Tab. \ref{ex3_accuracy_cost}, where it can be observed that DDM achieves better reconstruction accuracy than CDM. This result emphasizes the importance of data completion and the integration of scattering information. }

\subsubsection{Example 4: Full aperture case}

\begin{figure}[ht]
    \begin{center}
        \begin{overpic}[width=0.44\textwidth]{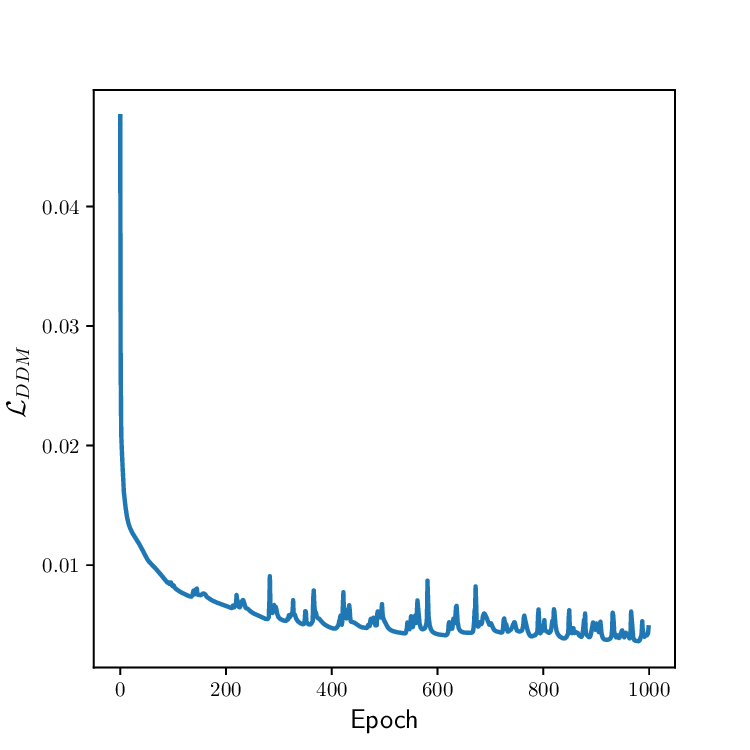}
           \put (40,91) {\scriptsize {$\mathcal{L}_{DDM}$}}
        \end{overpic}
        \hspace{-0.4cm} 
        \begin{overpic}[width=0.44\textwidth]{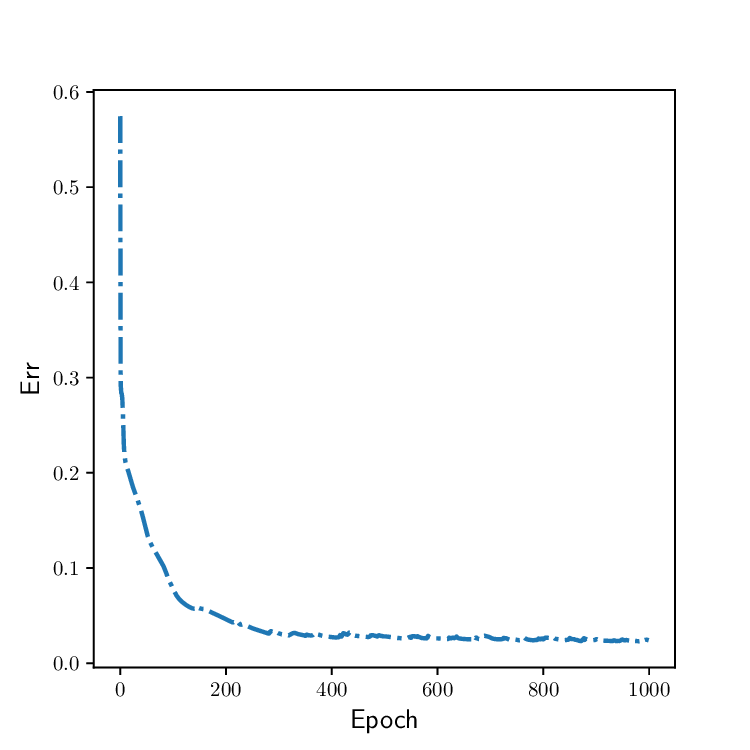}
           \put (45,91) {\scriptsize {Err}}
        \end{overpic}
    \end{center}
    \caption{\textcolor{black}{Example 4: The evolution of the DDM loss function $\mathcal{L}_{DDM}$ and the training relative error $\mathrm{Err}$, throughout the training process.}}
    \label{fig:ex4_loss_convergence}
\end{figure}
\begin{figure}[ht]
    \begin{center}
        \begin{overpic}[width=0.32\textwidth,trim=40 0 26 15, clip=true,tics=10]{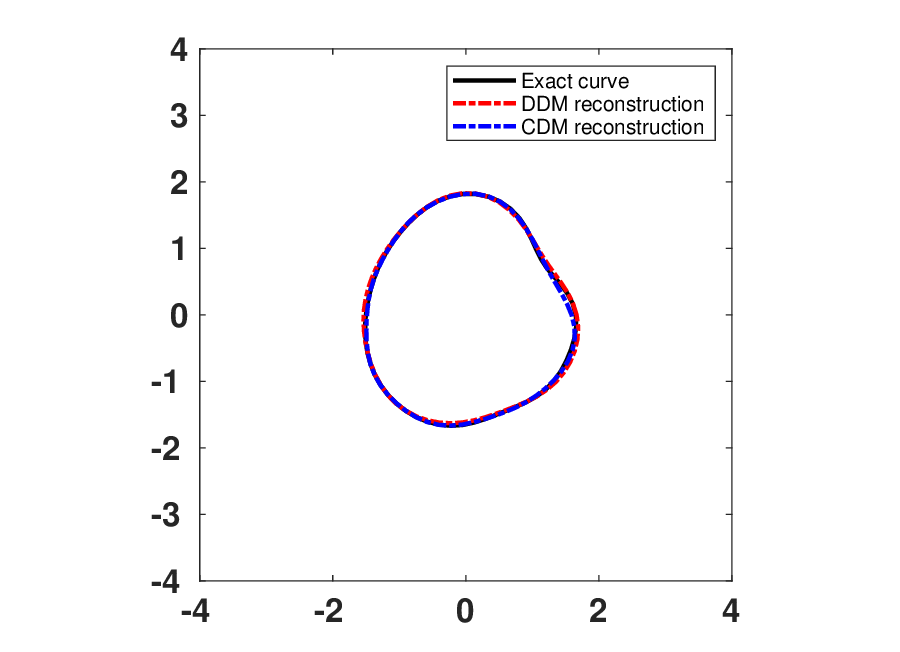}
           \put (40,82) {\scriptsize Case 1}
        \end{overpic}
        \hspace{-0.18cm} 
        \begin{overpic}[width=0.32\textwidth,trim=40 0 26 15, clip=true,tics=10]{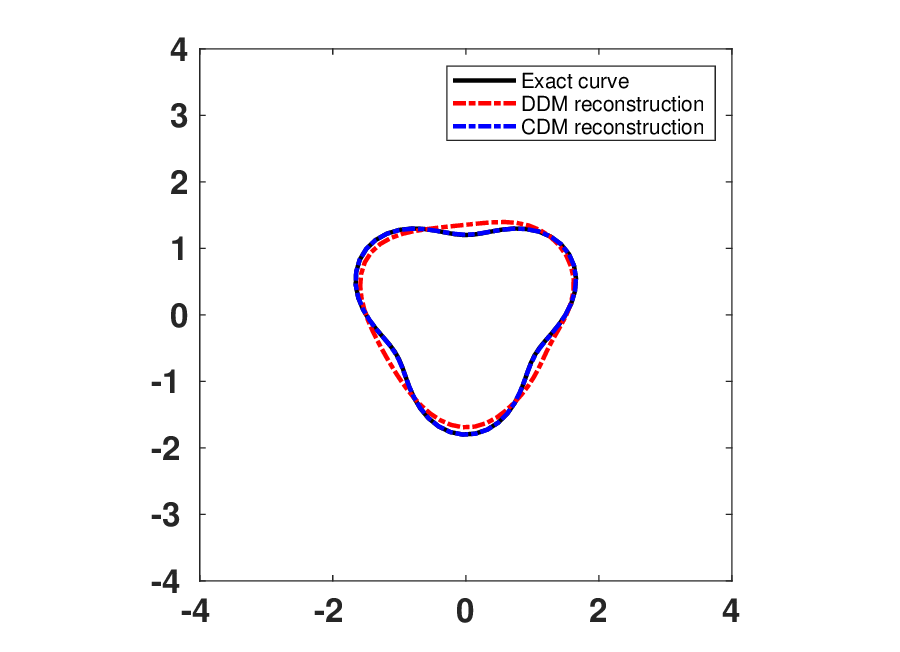}
           \put (40,82){\scriptsize Case 2}
        \end{overpic}
        \hspace{-0.18cm} 
        \begin{overpic}[width=0.32\textwidth,trim=40 0 26 15, clip=true,tics=10]{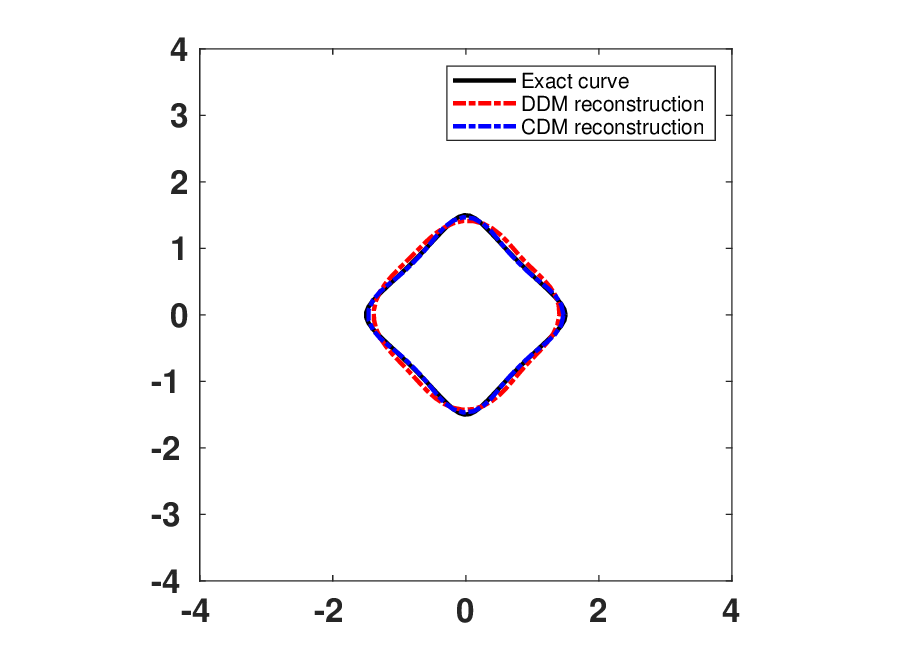}
           \put (40,82){\scriptsize Case 3}
        \end{overpic}
    \end{center}
    \caption{\textcolor{black}{Example 4: Reconstructions made by DDM and CDM using full aperture data.}}
    \label{fig:ex4_ddm}
\end{figure}
\begin{table}[ht]
\caption{\textcolor{black}{Example 4: The relative errors $\bar{e}$ of DDM and CDM  using full aperture data.}}
\centering
\begin{tabular}{cccc}
\toprule
&\multicolumn{1}{c}{}&\multicolumn{1}{c}{$\bar{e}$}\\
\cmidrule(lr){2-4}
 & Case 1 & Case 2 & Case 3\\
\midrule
DDM  & 0.0153 & 0.0595 & 0.0487\\
CDM  & 0.0131 & 0.0018 & 0.0121\\
\bottomrule
\end{tabular}
\label{ex4_accuracy_cost}
\end{table}

\textcolor{black}{In this example, we compare the performance of DDM and CDM using full aperture data and a noise level of $\sigma = 0.01$, specifically with $\psi=[0, 2\pi]$ and $\phi=[0, 2\pi]$. In this case, the DCnet is unnecessary, so $\mathcal{L}_{DDM} = \mathcal{L}_{phy}$. }

\textcolor{black}{The numerical results are shown in Figs. \ref{fig:ex4_loss_convergence}-\ref{fig:ex4_ddm} and Tab.\ref{ex4_accuracy_cost}.  The outcomes show that both methods produce acceptable results in the full aperture scenario. Furthermore, the $L^{2}$ relative error $\bar{e}$ for DDM in Tab.\ref{ex4_accuracy_cost} is lower than in the limited aperture cases, fully supporting the statement in Remark \ref{rem_ddm_convergcen_conditions}. Notably, CDM reconstructions outperform DDM in the full aperture setting, especially in Cases 2 and 3. This result is expected, as Cases 2 and 3 represent out-of-distribution examples, which may introduce generalization errors in DDM. A key advantage of DDM is that it requires no additional training to perform reconstructions, as it directly applies the pre-trained network. By learning the inverse operator, DDM can invert new data without retraining. Conversely, CDM requires optimization for each individual measurement, which contributes to its higher accuracy in the full aperture scenario but limits its ability to perform real-time reconstructions. Improving DDM’s reconstruction accuracy in such cases through fine-tuning of the network is a promising research direction, though beyond the scope of this study. }

\textcolor{black}{In conclusion,  for the limited aperture problem considered in this work, the data retrieval approach and incorporation of underlying scattering principles in DDM produce satisfactory results. Additionally, DDM focuses on learning the inverse operator through deep learning, which allows for faster computation once training is completed.}

\section{Summary}
In this paper, we propose the deep decomposition method(DDM), a scattering-based neural network for determining the shape of a sound-soft obstacle from limited aperture data at a fixed wavenumber or frequency. In DDM, a data completion scheme based on deep learning is implemented to prevent any distortion of the reconstructions. On the other hand, we incorporate the scattering information, such as the far-field operator, the Herglotz operator, and the fundamental solution, into the loss function of DDM for the boundary recovery in order to better utilize the underlying physics information for learning the regularized inverse operator. Because physical information is present, DDM does not require exact shape data during the training stage. Moreover, DDM is a physics-aware machine learning method that can handle ill-posedness in addition to having interpretability. The convergence of DDM is demonstrated, and numerical examples are provided to illustrate the validity of DDM.

\textcolor{black}{Our approach can be directly applied to other boundary conditions, inhomogeneous mediums, and even three-dimensional cases.  The use of multifrequency data to construct the loss function of a scattering-based neural network is an exciting direction for future research, as it will broaden the detectable range. Furthermore, because the data retrieval component of DDM is entirely data-driven, a physics-informed data completion approach represents a promising avenue for future development. These extensions will be investigated in future studies.}


\bibliographystyle{siamplain}

\begin{thebibliography}{10}

\bibitem{ABF}
{\sc H.~Ammari, G.~Bao, and J.~Fleming}, {\em An inverse source problem for
  maxwell's equations in magnetoencephalography}, SIAM Journal on Applied
  Mathematics, 62 (2002), pp.~1369--1382.

\bibitem{AH1}
{\sc L.~Audibert and H.~Haddar}, {\em The generalized linear sampling method
  for limited aperture measurements}, SIAM Journal on Imaging Sciences, 10
  (2017), pp.~845--870.

\bibitem{Bao_Li_Lin_1}
{\sc G.~Bao, P.~Li, J.~Lin, and F.~Triki}, {\em Inverse scattering problems
  with multi-frequencies}, Inverse Problems, 31 (2015), p.~093001.

\bibitem{BL1}
{\sc G.~Bao and J.~Liu}, {\em Numerical solution of inverse scattering problems
  with multi-experimental limited aperture data}, SIAM Journal on Scientific
  Computing, 25 (2003), pp.~1102--1117.

\bibitem{Bao_Lu_Rundell_Xu1}
{\sc G.~Bao, S.~Lu, W.~Rundell, and B.~Xu}, {\em A recursive algorithm for
  multifrequency acoustic inverse source problems}, SIAM Journal on Numerical
  Analysis, 53 (2015), pp.~1608--1628.

\bibitem{BYZZhou1}
{\sc G.~Bao, X.~Ye, Y.~Zang, and H.~Zhou}, {\em Numerical solution of inverse
  problems by weak adversarial networks}, Inverse Problems, 36 (2020),
  p.~115003.

\bibitem{B1}
{\sc B.~Borden}, {\em Mathematical problems in radar inverse scattering},
  Inverse Problems, 18 (2001), pp.~R1--R28.

\bibitem{CK}
{\sc D.~Colton and A.~Kirsch}, {\em A simple method for solving inverse
  scattering problems in the resonance region}, Inverse problems, 12 (1996),
  pp.~383--393.

\bibitem{DRIA}
{\sc D.~Colton and R.~Kress}, {\em Inverse Acoustic and Electromagnetic
  Scattering Theory}, Springer, New York, 2019.

\bibitem{Colton_Monk1}
{\sc D.~Colton and P.~Monk}, {\em A novel method for solving the inverse
  scattering problem for time-harmonic acoustic waves in the resonance region},
  SIAM journal on applied mathematics, 45 (1985), pp.~1039--1053.

\bibitem{Colton_Monk2}
{\sc D.~Colton and P.~Monk}, {\em A novel method for solving the inverse
  scattering problem for time-harmonic acoustic waves in the resonance region
  \textsc{II}}, SIAM journal on applied mathematics, 46 (1986), pp.~506--523.

\bibitem{CPP}
{\sc D.~Colton, M.~Piana, and R.~Potthast}, {\em A simple method using
  morozov's discrepancy principle for solving inverse scattering problems},
  Inverse problems, 13 (1997), pp.~1477--1493.

\bibitem{DLMZ1}
{\sc F.~Dou, X.~Liu, S.~Meng, and B.~Zhang}, {\em Data completion algorithms
  and their applications in inverse acoustic scattering with limited-aperture
  backscattering data}, Journal of Computational Physics, 469 (2022),
  p.~111550.

\bibitem{GLWZ1}
{\sc Y.~Gao, H.~Liu, X.~Wang, and K.~Zhang}, {\em On an artificial neural
  network for inverse scattering problems}, Journal of Computational Physics,
  448 (2022), p.~110771.

\bibitem{Gao_Zhang1}
{\sc Y.~Gao and K.~Zhang}, {\em Machine learning based data retrieval for
  inverse scattering problems with incomplete data}, Journal of Inverse and
  Ill-Posed Problems, 29 (2021), pp.~249--266.

\bibitem{gao2023adaptive}
{\sc Z.~Gao, L.~Yan, and T.~Zhou}, {\em Adaptive operator learning for
  infinite-dimensional bayesian inverse problems}, arXiv preprint
  arXiv:2310.17844,  (2023).

\bibitem{goswami_2023_1}
{\sc S.~Goswami, A.~Bora, Y.~Yu, and G.~E. Karniadakis}, {\em
  \textcolor{black}{Physics-informed deep neural operator networks}}, in
  Machine Learning in Modeling and Simulation: Methods and Applications, 2023,
  pp.~219--254.

\bibitem{goswami_2022}
{\sc S.~Goswami, M.~Yin, Y.~Yu, and G.~E. Karniadakis}, {\em
  \textcolor{black}{A physics-informed variational DeepONet for predicting
  crack path in quasi-brittle materials}}, Computer Methods in Applied
  Mechanics and Engineering, 391 (2022), p.~114587.

\bibitem{INS1}
{\sc M.~Ikehata, E.~Niemi, and S.~Siltanen}, {\em Inverse obstacle scattering
  with limited-aperture data}, Inverse problems and imaging, 6 (2012),
  pp.~77--94.

\bibitem{IJZ}
{\sc K.~Ito, B.~Jin, and J.~Zou}, {\em A direct sampling method to an inverse
  medium scattering problem}, Inverse Problems, 28 (2012), p.~025003.

\bibitem{ochs1}
{\sc L.~R. J.~Ochs}, {\em The limited aperture problem of inverse acoustic
  scattering: Dirichlet boundary conditions}, SIAM Journal on Applied
  Mathematics, 47 (1987), pp.~1320--1341.

\bibitem{jiao_2024}
{\sc A.~Jiao, Q.~Yan, J.~Harlim, and L.~Lu}, {\em \textcolor{black}{Solving
  forward and inverse PDE problems on unknown manifolds via physics-informed
  neural operators}}, arXiv preprint arXiv:2407.05477,  (2024).

\bibitem{R1}
{\sc R.~O. Jr}, {\em The limited aperture problem of inverse acoustic
  scattering: Dirichlet boundary conditions}, SIAM Journal on Applied
  Mathematics, 47 (1987), pp.~1320--1341.

\bibitem{kaltenbach2023SINO}
{\sc S.~Kaltenbach, P.~Perdikaris, and P.~S. Koutsourelakis}, {\em
  \textcolor{black}{Semi-supervised invertible neural operators for Bayesian
  inverse problems}}, Computational Mechanics, 72 (2023), pp.~451--470.

\bibitem{Khoo_Ying1}
{\sc Y.~Khoo and L.~Ying}, {\em Switchnet: a neural network model for forward
  and inverse scattering problems}, SIAM Journal on Scientific Computing, 41
  (2019), pp.~A3182--A3201.

\bibitem{KK1}
{\sc A.~Kirsch and R.~Kress}, {\em An optimization method in inverse acoustic
  scattering}, Boundary Elements IX, vol. 3: Fluid flow and potential
  applications, ed C.A. Brebbia et al, Berlin, Springer,,  (1987), pp.~3--18.

\bibitem{KL1}
{\sc A.~Kirsch and X.~Liu}, {\em A modification of the factorization method for
  the classical acoustic inverse scattering problems}, Inverse Problems, 30
  (2014), p.~035013.

\bibitem{Kress}
{\sc R.~Kress}, {\em Newtons method for inverse obstacle scattering meets the
  method of least squares}, Inverse Problems, 19 (2003), pp.~91--104.

\bibitem{Kuchment}
{\sc P.~Kuchment}, {\em The radon transform and medical imaging}, CBMS-NSF
  Regional Conference Series in Applied Mathematics, Society for Industrial and
  Applied Mathematics,  (2014).

\bibitem{LLTW1}
{\sc J.~Li, H.~Liu, W.~Tsui, and X.~Wang}, {\em An inverse scattering approach
  for geometric body generation: A machine learning perspective}, Mathematics
  in Engineering, 1 (2019), pp.~800--823.

\bibitem{LZ}
{\sc J.~Li and J.~Zou}, {\em A direct sampling method for inverse scattering
  using far-field data}, Inverse Problems and Imaging, 7 (2013), pp.~757--775.

\bibitem{LZZ1}
{\sc K.~Li, B.~Zhang, and H.~Zhang}, {\em Reconstruction of inhomogeneous media
  by iterative reconstruction algorithm with learned projector}, arXiv preprint
  arXiv:2207.13032,  (2022).

\bibitem{li2023surrogate}
{\sc Y.~Li, Y.~Wang, and L.~Yan}, {\em Surrogate modeling for bayesian inverse
  problems based on physics-informed neural networks}, Journal of Computational
  Physics, 475 (2023), p.~111841.

\bibitem{li2020FNO}
{\sc Z.~Li, N.~Kovachki, K.~Azizzadenesheli, B.~Liu, K.~Bhattacharya,
  A.~Stuart, and A.~Anandkumar}, {\em \textcolor{black}{Fourier neural operator
  for parametric partial differential equations}}, arXiv preprint
  arXiv:2010.08895,  (2020).

\bibitem{li2024PIFNO}
{\sc Z.~Li, H.~Zheng, N.~Kovachki, D.~Jin, H.~Chen, B.~Liu, K.~Azizzadenesheli,
  and A.~Anandkumar}, {\em \textcolor{black}{Physics-informed neural operator
  for learning partial differential equations}}, ACM/JMS Journal of Data
  Science, 1 (2024), pp.~1--27.

\bibitem{Liu2024failure}
{\sc W.~Liu, L.~Yan, T.~Zhou, and Y.~Zhou}, {\em Failure-informed adaptive
  sampling for pinns, part iii: Applications to inverse problems}, CSIAM
  Transactions on Applied Mathematics, 5 (2024), pp.~636--670.

\bibitem{Liu_xd1}
{\sc X.~Liu}, {\em A novel sampling method for multiple multiscale targets from
  scattering amplitudes at a fixed frequency}, Inverse Problems, 33 (2017),
  p.~085011.

\bibitem{Liu_Sun_1}
{\sc X.~Liu and J.~Sun}, {\em Data recovery in inverse scattering: from
  limited-aperture to full-aperture}, Journal of Computational Physics, 386
  (2019), pp.~350--364.

\bibitem{lu_deeponet}
{\sc L.~Lu, P.~Jin, G.~Pang, Z.~Zhang, and G.~E. Karniadakis}, {\em
  \textcolor{black}{Learning nonlinear operators via DeepONet based on the
  universal approximation theorem of operators}}, Nature machine intelligence,
  3 (2021), pp.~218--229.

\bibitem{LPYWVJ1}
{\sc L.~Lu, R.~Pestourie, W.~Yao, Z.~Wang, F.~Verdugo, and S.~G. Johnson}, {\em
  Physics-informed neural networks with hard constraints for inverse design},
  SIAM Journal on Scientific Computing, 43 (2021), pp.~B1105--B1132.

\bibitem{McLean1}
{\sc W.~McLean}, {\em Strongly elliptic systems and boundary integral
  equations}, Cambridge university press, 2000.

\bibitem{MYEM1}
{\sc R.~Molinaro, Y.~Yang, B.~Engquist, and S.~Mishra}, {\em Neural inverse
  operators for solving pde inverse problems}, arXiv preprint arXiv:2301.11167,
   (2023).

\bibitem{NB1}
{\sc H.~V. Nguyen and T.~Bui-Thanh}, {\em Tnet: A model-constrained tikhonov
  network approach for inverse problems}, SIAM Journal on Scientific Computing,
  46 (2024), pp.~C77--C100.

\bibitem{NHZ1}
{\sc J.~Ning, F.~Han, and J.~Zou}, {\em A direct sampling-based deep learning
  approach for inverse medium scattering problems}, Inverse Problems, 40
  (2023), p.~015005.

\bibitem{Ning_Han_Zou2}
{\sc J.~Ning, F.~Han, and J.~Zou}, {\em A direct sampling method and its
  integration with deep learning for inverse scattering problems with phaseless
  data}, arXiv preprint arXiv:2403.02584,  (2024).

\bibitem{PMAG1}
{\sc S.~Pakravan, P.~A. Mistani, M.~A. Aragon-Calvo, and F.~Gibou}, {\em
  Solving inverse-pde problems with physics-aware neural networks}, Journal of
  Computational Physics, 440 (2021), p.~110414.

\bibitem{RPKarniadakis_1}
{\sc M.~Raissi, P.~Perdikaris, and G.~E. Karniadakis}, {\em Physics-informed
  neural networks: A deep learning framework for solving forward and inverse
  problems involving nonlinear partial differential equations}, Journal of
  Computational physics, 378 (2019), pp.~686--707.

\bibitem{R-B_HSK1}
{\sc M.~Rasht-Behesht, C.~Huber, K.~Shukla, and G.~E. Karniadakis}, {\em
  Physics-informed neural networks (pinns) for wave propagation and full
  waveform inversions}, Journal of Geophysical Research: Solid Earth, 127
  (2022), p.~e2021JB023120.

\bibitem{wang_sifan_2021}
{\sc S.~Wang, H.~Wang, and P.~Perdikaris}, {\em \textcolor{black}{Learning the
  solution operator of parametric partial differential equations with
  physics-informed DeepONets}}, Science advances, 7 (2021), p.~eabi8605.

\bibitem{zhou2020adaptive}
{\sc L.~Yan and T.~Zhou}, {\em An adaptive surrogate modeling based on deep
  neural networks for large-scale bayesian inverse problems}, Communications in
  Computational Physics, 28 (2020), pp.~2180--2205.

\bibitem{Yangjiaqing1}
{\sc J.~Yang, B.~Zhang, and H.~Zhang}, {\em Reconstruction of complex obstacles
  with generalized impedance boundary conditions from far-field data}, SIAM
  Journal on Applied Mathematics, 74 (2014), pp.~106--124.

\bibitem{YYL1}
{\sc W.~Yin, W.~Yang, and H.~Liu}, {\em A neural network scheme for recovering
  scattering obstacles with limited phaseless far-field data}, Journal of
  Computational Physics, 417 (2020), p.~109594.

\bibitem{YLMK1}
{\sc J.~Yu, L.~Lu, X.~Meng, and G.~E. Karniadakis}, {\em Gradient-enhanced
  physics-informed neural networks for forward and inverse pde problems},
  Computer Methods in Applied Mechanics and Engineering, 393 (2022), p.~114823.

\bibitem{ZBYZ_weak}
{\sc Y.~Zang, G.~Bao, X.~Ye, and H.~Zhou}, {\em Weak adversarial networks for
  high-dimensional partial differential equations}, Journal of Computational
  Physics, 411 (2020), p.~109409.

\bibitem{ZHRB1}
{\sc M.~Zhou, J.~Han, M.~Rachh, and C.~Borges}, {\em A neural network
  warm-start approach for the inverse acoustic obstacle scattering problem},
  Journal of Computational Physics, 490 (2023), p.~112341.

\end{thebibliography}

\end{document}